\documentclass[a4paper,11pt]{article}
\title{The class of $(2P_3,C_4,C_6)$-free graphs, part II: $(2P_3,C_4,C_6,C_7,T_0)$-free graphs} 
\author{Irena Penev\thanks{Computer Science Institute (I\'UUK, MFF), Charles University, Prague, Czech Republic. Supported by GA\v{C}R grant 25-17377S. Email: {ipenev@iuuk.mff.cuni.cz}.}}

\usepackage{fullpage}

\usepackage{amsmath}
\usepackage{amsthm}
\usepackage{amssymb}
\usepackage{amsbsy}
\usepackage{mathrsfs}
\usepackage{amstext}  
\usepackage{enumerate} 

\usepackage{graphics}
\usepackage{graphicx}

\usepackage{chngcntr}
\numberwithin{figure}{section}

\usepackage{changepage}

\newtheorem{theorem}{Theorem}[section]  
\newtheorem{proposition}[theorem]{Proposition} 
\newtheorem{lemma}[theorem]{Lemma} 
\newtheorem{corollary}[theorem]{Corollary} 
 
\newtheorem{claim}{Claim}[theorem]

\begin{document} 
\maketitle 
\noindent 

\begin{abstract} 
\noindent 
This is the second in a series of two papers dealing with $(2P_3,C_4,C_6)$-free graphs, or equivalently, $(2P_3,\text{even hole})$-free graphs. In this two-paper series, we give a full structural description of $(2P_3,C_4,C_6)$-free graphs that contain no simplicial vertices, and we show that such graphs have bounded clique-width. This implies that \textsc{Graph Coloring} can be solved in polynomial time for $(2P_3,C_4,C_6)$-free graphs. In the first paper of the series, we described the structure of $(2P_3,C_4,C_6)$-free graphs that contain an induced $C_7$ or an induced $T_0$ (where $T_0$ is a certain 2-connected graph on nine vertices in which all holes are of length five), and we showed that such graphs either contain a simplicial vertex or have bounded clique-width. In the present paper (the second part of the series), we describe the structure of $(2P_3,C_4,C_6,C_7,T_0)$-free graphs that contain no simplicial vertices, and we show that such graphs have bounded clique-width. Finally this paper gives the full statement of the theorem describing the structure of $(2P_3,C_4,C_6)$-free graphs that contain no simplicial vertices. 
\end{abstract} 

\section{Introduction}

All graphs in this paper are finite, simple, and nonnull. We mostly use standard terminology and notation, formally introduced in section~\ref{sec:terminology}. For now, let us define a few basic terms. As usual, for a positive integer $k$, $K_k$ is the complete graph on $k$ vertices, $P_k$ is the path on $k$ vertices and $k-1$ edges, and  $C_k$ ($k \geq 3$) is the cycle on $k$ vertices. For a positive integer $k$ and a graph $H$, the disjoint union of $k$ copies of the graph $H$ is denoted by $kH$. (See Figure~\ref{fig:ForbiddenIndSgs} for some graphs relevant to this paper.) For a graph $H$, we say that a graph $G$ is {\em $H$-free} if no induced subgraph of $G$ is isomorphic to $H$. For a family of graphs $\mathcal{H}$, we say that a graph $G$ is {\em $\mathcal{H}$-free} if $G$ is $H$-free for all $H \in \mathcal{H}$. A {\em hole} in a graph $G$ is an induced cycle on at least four vertices, and the {\em length} of a hole is the number of vertices (equivalently: edges) that it contains. A hole is {\em even} (resp.\ {\em odd}) if its length is even (resp.\ odd). A {\em clique} (resp.\ {\em stable set}) is a (possibly empty) set of pairwise adjacent (resp.\ nonadjacent) vertices, and a {\em simplicial vertex} is a vertex whose neighbors form a (possibly empty) clique. A {\em proper coloring} of a graph $G$ is an assignment of colors to the vertices of $G$ in such a way that no two adjacent vertices receive the same color. For an integer $k$, a graph $G$ is said to be {\em $k$-colorable} if there exists a proper coloring of $G$ that uses at most $k$ colors. The {\em chromatic number} of $G$, denoted by $\chi(G)$, is the smallest integer $k$ such that $G$ is $k$-colorable. \textsc{Graph Coloring} is the following problem. 

\bigskip 

\begin{minipage}{\textwidth}  
\textsc{Graph Coloring} 

\textbf{Instance:} A graph $G$ and an integer $k$. 

\textbf{Question:} Is $G$ $k$-colorable? 
\end{minipage} 

\bigskip 

The present paper is the second in a series of two papers that deal with $(2P_3,C_4,C_6)$-free graphs. For a detailed introduction, as well as some open questions, see the first paper of the series~\cite{2P3C4C6FreePart1}. We note, however, that all holes of length at least eight contain an induced $2P_3$, and consequently, all holes in a $(2P_3,C_4,C_6)$-free graph are of length five or seven. In particular, all $(2P_3,C_4,C_6)$-free graphs are even-hole-free, and in fact, $(2P_3,C_4,C_6)$-free graphs are precisely the $(2P_3,\text{even hole})$-free graphs. We also note that $4K_1$ is an induced subgraph of $2P_3$, which is in turn an induced subgraph of $P_7$; so, our class of $(2P_3,C_4,C_6)$-free graphs is a proper superclass of the class of $(4K_1,C_4,C_6)$-free graphs, and is a proper subclass of the class of $(P_7,C_4,C_6)$-free graphs. 

\begin{figure}
\begin{center}
\includegraphics[scale=0.5]{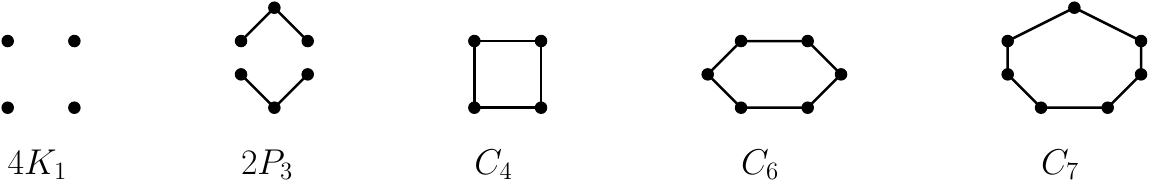}
\end{center} 
\caption{From left to right: graphs $4K_1$, $2P_3$, $C_4$, $C_6$, and $C_7$.} \label{fig:ForbiddenIndSgs} 
\end{figure}

The main result of this two-paper series is a full structural description of $(2P_3,C_4,C_6)$-free graphs that contain no simplicial vertices (see Theorem~\ref{thm-main-structure}). Moreover, we prove that all $(2P_3,C_4,C_6)$-free graphs that contain no simplicial vertices have bounded clique-width (see Theorem~\ref{thm-main-cwd}; for a formal definition of ``clique-width,'' see section~\ref{sec:terminology}). Since simplicial vertices pose no obstacle to coloring in polynomial time,\footnote{Indeed, if $k$ is an integer, and $x$ is a simplicial vertex of a graph $G$ on at least two vertices, then $G$ is $k$-colorable if and only if $d_G(x) \leq k-1$ and $G \setminus x$ is $k$-colorable. Moreover, the simplicial vertices of a graph (if any) can easily be found by simply examining the neighborhood of each vertex and checking if that neighborhood is a clique.} and since \textsc{Graph Coloring} can be solved in polynomial time for graphs of bounded clique-width~\cite{Rao}, it follows that the \textsc{Graph Coloring} problem can be solved in polynomial time for $(2P_3,C_4,C_6)$-free graphs (see Theorem~\ref{thm-main-coloring}). We note, however, that all coloring algorithms that rely on bounded clique-width are quite slow; in particular, the known coloring algorithms for graphs of clique-width at most $k$ all have running time $O(n^{f(k)})$ for some fast growing function $f$ (see~\cite{CWcol}). It would therefore be interesting to develop a faster coloring algorithm for $(2P_3,C_4,C_6)$-free graphs, perhaps relying on our structural results for this class of graphs. This is discussed in more detail at the end of section~\ref{sec:coloring}. 

In the first part of this series, we dealt with $(2P_3,C_4,C_6)$-free graphs that contain an induced $C_7$ or an induced $T_0$ (see Figure~\ref{fig:T0T1}). In the present paper, we deal with $(2P_3,C_4,C_6,C_7,T_0)$-free graphs; we note that all holes in such graphs are of length five, or if we use the terminology from~\cite{L-Holed}, these graphs are ``5-holed.'' In particular, we give a full structural description of $(2P_3,C_4,C_6,C_7,T_0)$-free graphs that contain no simplicial vertices (see Theorem~\ref{thm-structure-C7T0Free}), and we show that such graphs have bounded clique-width (see Theorem~\ref{thm-cwd-in-class-with-C7T0Free}). Combined with the results of the first paper of our series, this yields a full structural description of $(2P_3,C_4,C_6)$-free graphs that contain no simplicial vertices (see Theorem~\ref{thm-main-structure}), as well as the fact that all such graphs have bounded clique-width (see Theorem~\ref{thm-main-cwd}). As explained above, this implies that the \textsc{Graph Coloring} problem can be solved in polynomial time for $(2P_3,C_4,C_6)$-free graphs (see Theorem~\ref{thm-main-coloring}).

\begin{figure}
\begin{center}
\includegraphics[scale=0.5]{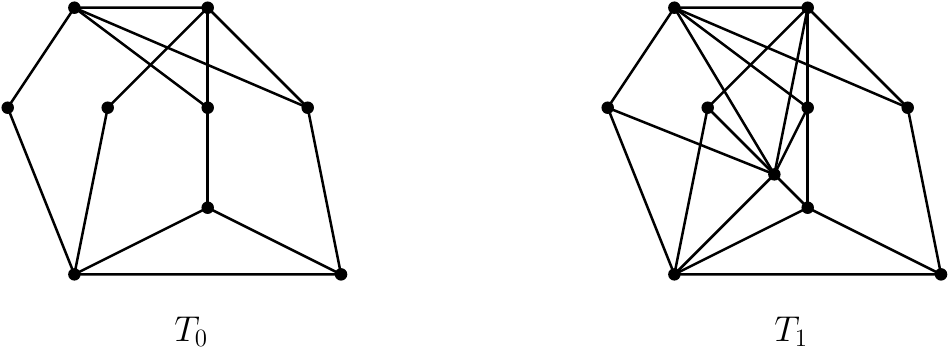}
\end{center} 
\caption{Graphs $T_0$ (left) and $T_1$ (right).} \label{fig:T0T1} 
\end{figure} 


\begin{figure}
\begin{center}
\includegraphics[scale=0.5]{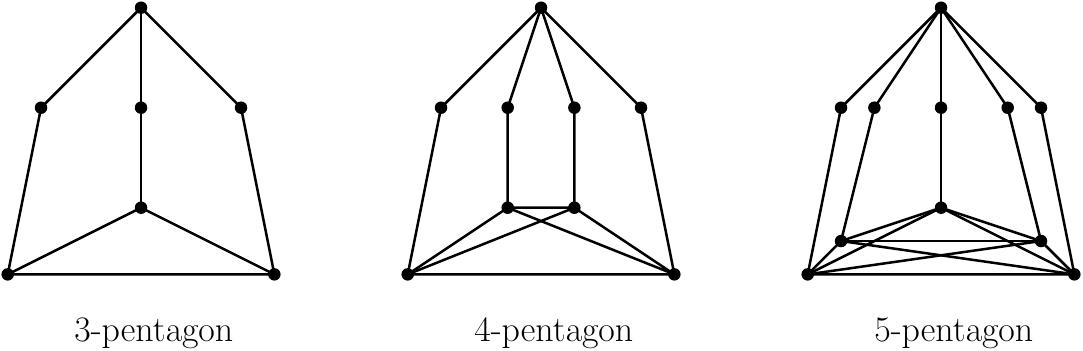}
\end{center} 
\caption{The $t$-pentagon for $t = 3$ (left), $t = 4$ (middle), and $t = 5$ (right).} \label{fig:tPentagon} 
\end{figure}

\subsection{Proof sketch} 

Let us briefly explain our strategy for describing the structure of $(2P_3,C_4,C_6,C_7,T_0)$-free graphs that contain no simplicial vertices, i.e.\ for proving Theorem~\ref{thm-structure-C7T0Free}. First of all, the {\em 3-pentagon} is the graph represented in Figure~\ref{fig:tPentagon} (left); we note that all holes in the 3-pentagon are of length five, and we also note that the 3-pentagon is an induced subgraph of $T_0$. A {\em universal vertex} of a graph is a vertex that is adjacent to all other vertices of that graph. 

Here is a brief sketch of our proof. We start with a $(2P_3,C_4,C_6,C_7,T_0)$-free graph $G$, and we break up our proof into two cases: when $G$ contains an induced 3-pentagon, and when $G$ is 3-pentagon-free. 

Suppose first that our graph $G$ does contain an induced 3-pentagon (this case is dealt with in section~\ref{sec:with-3-pentagon}; we remark that this section is the most technical part of our two-paper series). We then fix the largest integer $t \geq 3$ such that $G$ contains an induced ``$t$-pentagon.'' (We refer the reader to subsection~\ref{subsec:BasicGraphsDef} for a formal definition, but essentially, a $t$-pentagon is a $(2t+1)$-vertex graph consisting of a clique of size $t$, an additional vertex, plus $t$ suitable two-edge paths between this additional vertex and the clique. See Figure~\ref{fig:tPentagon} for an illustration, and note that all holes in a $t$-pentagon are of length five.) We then consider a maximal induced ``$t$-frame'' in our graph $G$ (a $t$-frame is a certain generalization of the $t$-pentagon, where the vertices of the $t$-pentagon become nonempty sets of vertices, and where adjacency between those sets is restricted in certain convenient ways, roughly imitating the $t$-pentagon; for a formal definition, see subsection~\ref{subsec:frames}). Then, by examining the exact structure of this $t$-frame in our $(2P_3,C_4,C_6,C_7,T_0,\text{$(t+1)$-pentagon})$-free graph $G$,\footnote{The fact that $G$ is $(t+1)$-pentagon-free follows from the maximality of $t$.} and by examining how the remaining vertices of $G$ ``attach'' to this $t$-frame, we eventually prove that either $G$ contains a simplicial or a universal vertex, or $G$ is one of the following: a ``5-basket,'' a ``villa,'' or a ``mansion'' (for the relevant definitions, see subsection~\ref{subsec:BasicGraphsDef}, and for a precise statement of our result, see Theorem~\ref{thm-T0free-pyramid-iff}). 

We then turn our attention to the 3-pentagon-free case (this case is dealt with in section~\ref{sec:3-pentagon-free}). So, we assume that our $(2P_3,C_4,C_6,C_7,T_0)$-free graph $G$ is additionally 3-pentagon-free, i.e.\ that $G$ is $(2P_3,C_4,C_6,C_7,\text{3-pentagon})$-free.\footnote{Since the 3-pentagon is an induced subgraph of $T_0$, we see that $(2P_3,C_4,C_6,C_7,T_0,\text{3-pentagon})$-free graphs are in fact exactly the $(2P_3,C_4,C_6,C_7,\text{3-pentagon})$-free graphs.} In this case, we rely heavily on ``Truemper configurations'' and certain results from the literature (for a slightly dated survey of Truemper configurations, see~\cite{Truemper-survey}). A {\em three-path-configuration} (or {\em 3PC} for short) is any theta, pyramid, or prism (see Figure~\ref{fig:Truemper}, and note that the 3-pentagon is a type of pyramid). A {\em wheel} is any graph that consists of a hole, plus an additional vertex that has at least three neighbors in the hole. A {\em Truemper configuration} is any 3PC or wheel. Recall that all holes in a $(2P_3,C_4,C_6,C_7)$-free graphs are of length five. Since the only 3PC in which all holes are of length five is the 3-pentagon, it follows that $(2P_3,C_4,C_6,C_7,\text{3-pentagon})$-free graphs are 3PC-free. On the other hand, the only wheels in which all holes are of length five are the two wheels represented in Figure~\ref{fig:TwoWheels}, and those two wheels are not ``proper'' (for a definition of a proper wheel, see section~\ref{sec:3-pentagon-free}). It follows that $(2P_3,C_4,C_6,C_7,\text{3-pentagon})$-free graphs are, in particular, $(\text{3PC},\text{proper wheel})$-free. Using a decomposition theorem for $(\text{3PC},\text{proper wheel})$-free graphs proven in~\cite{VIK}, as well as certain results of~\cite{4K1C4C6C7Free}, we can show fairly easily that either our graph $G$ contains a simplicial or a universal vertex, or $G$ is a ``5-crown'' (the 5-crown is defined in section~\ref{sec:3-pentagon-free}; for the precise statement of our result, see Theorem~\ref{thm-non-chordal-pyramid-free-iff}). 

\begin{figure} 
\begin{center}
\includegraphics[scale=0.5]{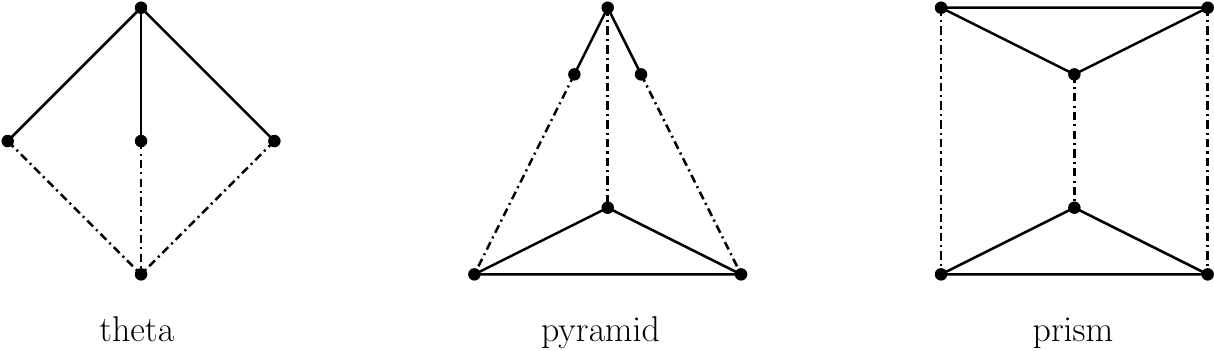}
\end{center} 
\caption{Three-path-configurations (3PCs): theta (left), pyramid (middle), and prism (right). A full line represents an edge, and a dashed line represents a path that has at least one edge.} \label{fig:Truemper} 
\end{figure}

\begin{figure}
\begin{center}
\includegraphics[scale=0.5]{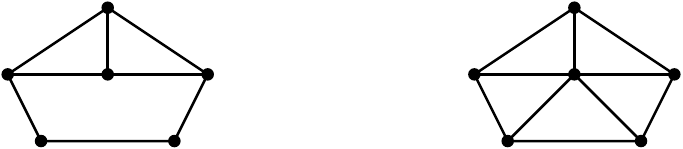}
\end{center} 
\caption{The only wheels in which all holes are of length five.} \label{fig:TwoWheels} 
\end{figure}

\subsection{Paper outline} 

This paper is organized as follows. In section~\ref{sec:terminology}, we give some (mostly standard) terminology and notation that we use throughout the paper, and we also describe some conventions that we use in our figures. Next, in section~\ref{sec:Part1}, we state the main results of the first paper of our two-paper series~\cite{2P3C4C6FreePart1}. Section~\ref{sec:simple} is devoted to some simple propositions that we will need in the remainder of the paper (the propositions stated in subsection~\ref{subsec:Part1SimpleResults} were proven in the first paper of our series, whereas the propositions stated and proven in subsection~\ref{subsec:simple-new} are new). In section~\ref{sec:with-3-pentagon}, we give a full structural description of $(2P_3,C_4,C_6,C_7,T_0)$-free graphs that contain an induced 3-pentagon and contain no simplicial vertices (see Theorem~\ref{thm-T0free-pyramid-iff}), and in section~\ref{sec:3-pentagon-free}, we give a full structural description of $(2P_3,C_4,C_6,C_7,\text{3-pentagon})$-free graphs that contain no simplicial vertices (see Theorem~\ref{thm-non-chordal-pyramid-free-iff}). In section~\ref{sec:structure}, we combine Theorems~\ref{thm-T0free-pyramid-iff} and~\ref{thm-non-chordal-pyramid-free-iff} to obtain Theorem~\ref{thm-structure-C7T0Free}, which gives a full structural description of $(2P_3,C_4,C_6,C_7,T_0)$-free graphs that contain no simplicial vertices. Then, we combine Theorem~\ref{thm-structure-C7T0Free} with the main structural result of the first paper of our series (that result is stated as Theorem~\ref{thm-main-withC7T0} in the present paper) in order to obtain the main structural result of our series: a full structural description of $(2P_3,C_4,C_6)$-free graphs that contain no simplicial vertices (see Theorem~\ref{thm-main-structure}). In section~\ref{sec:cwd}, we show that $(2P_3,C_4,C_6,C_7,T_0)$-free graphs that contain no simplicial vertices have bounded clique-width (see Theorem~\ref{thm-cwd-in-class-with-C7T0Free}); combined with the results of the first paper of our series, this implies that $(2P_3,C_4,C_6)$-free graphs that contain no simplicial vertices have bounded clique-width (see Theorem~\ref{thm-main-cwd}). Finally, in section~\ref{sec:coloring}, we show that \textsc{Graph Coloring} can be solved in polynomial time for $(2P_3,C_4,C_6)$-free graphs (see Theorem~\ref{thm-main-coloring}). As we have already explained, this last result relies on clique-width, and therefore, the running time is in fact very slow (albeit polynomial). In section~\ref{sec:coloring}, we explain that the problem of constructing a reasonably fast coloring algorithm for $(2P_3,C_4,C_6)$-free graphs can be reduced to that of coloring only two types of graphs, namely ``villas'' and ``mansions'' (these two types of graphs are defined in subsection~\ref{subsec:BasicGraphsDef}).

\section{Terminology and notation} \label{sec:terminology} 

When we say that ``sets $X_1,\dots,X_{\ell}$ form a partition of the set $X$,'' or that ``$(X_1,\dots,X_{\ell})$ is a partition of the set $X$,'' we mean that sets $X_1,\dots,X_{\ell}$ are pairwise disjoint and that $X = X_1 \cup \dots \cup X_{\ell}$. However, slightly nonstandardly, we allow some (or all) of the $X_i$'s to be empty. 

The vertex and edge set of a graph $G$ are denoted by $V(G)$ and $E(G)$, respectively. The {\em clique number} of a graph $G$ is the maximum cardinality of any clique in $G$, and it is denoted by $\omega(G)$; the {\em stability number} of $G$ is the maximum cardinality of any stable set in $G$, and it is denoted by $\alpha(G)$. 

For a vertex $x$ in a graph $G$, the {\em open neighborhood} (or simply {\em neighborhood}) of $x$ in $G$, denoted by $N_G(x)$, is the set of all neighbors of $x$ in $G$; the {\em closed neighborhood} of $x$ in $G$, denoted by $N_G[x]$, is defined as $N_G[x] := \{x\} \cup N_G(x)$; the {\em degree} of $x$ in $G$, denoted by $d_G(x)$, is the number of neighbors that $x$ has in $G$, {\em i.e.}\ $d_G(x) := |N_G(x)|$. 

A vertex $v$ of a graph $G$ is {\em universal} in $G$ if $N_G[v] = V(G)$, i.e.\ if $v$ is adjacent to all other vertices of the graph $G$. 

Distinct vertices $x,y$ of a graph $G$ are {\em twins} in $G$ if $N_G[x] = N_G[y]$. (Note that twins are, in particular, adjacent to each other.) 

For distinct vertices $x,y$ of a graph $G$, we say that $x$ {\em dominates} $y$ in $G$ if $N_G[y] \subseteq N_G[x]$. (Note that if $x$ dominates $y$ in $G$, then in particular, $x$ and $y$ are adjacent.) 

For a graph $G$ and a set $X \subseteq V(G)$, the {\em open neighborhood} of $X$ in $G$, denoted by $N_G(X)$, is the set of all vertices in $V(G) \setminus X$ that have a neighbor in $X$, and the {\em closed neighborhood} of $X$ in $G$ is the set $N_G[X] := X \cup N_G(X)$. 

For a graph $G$ and a nonempty set $X \subseteq V(G)$, we denote by $G[X]$ the subgraph of $G$ induced by $X$; for vertices $x_1,\dots,x_t \in V(G)$, we sometimes write $G[x_1,\dots,x_t]$ instead of $G[\{x_1,\dots,x_t\}]$. For a set $X \subsetneqq V(G)$, we set $G \setminus X := G[V(G) \setminus X]$. If $G$ has at least two vertices and $x \in V(G)$, we sometimes write $G \setminus x$ instead of $G \setminus \{x\}$.\footnote{Since our graphs are nonnull, if $x$ is the only vertex of a graph $G$, then $G \setminus x$ is not defined.} 

Given a graph $G$, a vertex $x \in V(G)$, and a set $Y \subseteq V(G) \setminus \{x\}$, we say that $x$ is {\em complete} (resp.\ {\em anticomplete}) to $Y$ in $G$ provided that $x$ is adjacent (resp.\ nonadjacent) to all vertices of $Y$ in $G$, and we say that $x$ is {\em mixed} on $Y$ in $G$ if $x$ is neither complete nor anticomplete to $Y$ in $G$ (i.e.\ if $x$ has both a neighbor and a nonneighbor in $Y$). 

Given a graph $G$ and disjoint sets $X,Y \subseteq V(G)$, we say that $X$ is {\em complete} (resp.\ {\em anticomplete}) to $Y$ in $G$ if every vertex in $X$ is complete (resp.\ anticomplete) to $Y$ in $G$. 

When we say, for graphs $G$ and $Q$, that ``$G$ can be obtained from $Q$ by possibly adding universal vertices to it,'' we mean that $Q$ is an induced subgraph of $G$, and $V(G) \setminus V(Q)$ is a (possibly empty) clique of $G$, complete to $V(Q)$ in $G$. Note that in this case, every vertex in $V(G) \setminus V(Q)$ is a universal vertex of $G$. However, it is possible that $G = Q$, in which case $G$ contains no universal vertices unless $Q$ does. 

A {\em thickening} of a graph $G$ is a graph $G^*$ for which there exists a family $\{X_v\}_{v \in V(G)}$ of pairwise disjoint, nonempty sets such that all the following hold: 
\begin{itemize} 
\item $V(G^*) = \bigcup_{v \in V(G)} X_v$; 
\item for all $v \in V(G)$, $X_v$ is a clique of $G^*$; 
\item for all distinct $u,v \in V(G)$, the following hold: 
\begin{itemize} 
\item if $u,v$ are adjacent in $G$, then $X_u$ and $X_v$ are complete to each other in $G^*$; 
\item if $u,v$ are nonadjacent in $G$, then $X_u$ and $X_v$ are anticomplete to each other in $G^*$. 
\end{itemize} 
\end{itemize} 

When we write, for a graph $G$, that ``$x_0,\dots,x_t$ is an induced path in $G$'' ($t \geq 0$), we mean that $x_0,\dots,x_t$ are pairwise distinct vertices, and that the edges of $G[x_0,\dots,x_t]$ are precisely $x_0x_1,x_1x_2,\dots,x_{t-1}x_t$; the {\em endpoints} of this path are $x_0$ and $x_t$ (note that if $t = 0$, then our path has only only one endpoint), the {\em interior vertices} of the path are all its vertices other than the endpoints, and furthermore, we say that our path is {\em between} its endpoints $x_0$ and $x_t$. The {\em length} of a path is the number of edges that it contains; we may also refer to a path of length $k$ as a ``$k$-edge path'' or as a ``$(k+1)$-vertex path.'' A path is {\em trivial} if it is of length zero (i.e.\ if it has only one vertex and no edges), and it is {\em nontrivial} if it is of length at least one. 

For an integer $k \geq 4$, a {\em $k$-hole} in a graph $G$ is an induced $C_k$ in $G$. When we write that ``$x_0,x_1,\dots,x_{k-1},x_0$ is a $k$-hole in $G$'' ($k \geq 4$), we mean that $x_0,x_1,\dots,x_{k-1}$ are pairwise distinct vertices of $G$, and that the edges of $G[x_0,x_1,\dots,x_{k-1}]$ are precisely the following: $x_0x_1,x_1x_2,\dots,x_{k-2}x_{k-1},x_{k-1}x_0$. A graph is {\em chordal} if it contains no holes, that is, if all its induced cycles (if any) are triangles. 

For an integer $k \geq 4$, a {\em $k$-hyperhole} is a graph $H$ whose vertex set can be partitioned into nonempty cliques $X_0,X_1,\dots,X_{k-1}$ (with indices in $\mathbb{Z}_k$) such that for all $i \in \mathbb{Z}_k$, $X_i$ is complete to $X_{i-1} \cup X_{i+1}$ and is anticomplete to $V(H) \setminus (X_{i-1} \cup X_i \cup X_{i+1})$; under these circumstances, we also say that $(X_0,X_1,\dots,X_{k-1})$ is a {\em $k$-hyperhole partition} of the $k$-hyperhole $H$. Note that, for an integer $k \geq 4$, a $k$-hyperhole is simply a thickening of $C_k$. 

The {\em complement} of a graph $G$, denoted by $\overline{G}$, is the graph whose vertex set is $V(G)$ and in which two distinct vertices are adjacent if and only if they are nonadjacent in $G$. A graph is {\em anticonnected} if its complement is connected. Note that every anticonnected graph on at least two vertices contains a pair of nonadjacent vertices. 

As usual, a {\em connected component} (or simply {\em component}) of a graph $G$ is a maximal connected induced subgraph of $G$. An {\em anticomponent} of a graph $G$ is a maximal anticonnected induced subgraph of $G$; clearly, $Q$ is an anticomponent of $G$ if and only if $\overline{Q}$ is a component of $\overline{G}$. An anticomponent is {\em trivial} if it has only one vertex, and it is {\em nontrivial} if it has at least two vertices. Clearly, the vertex sets of the anticomponents of a graph $G$ are complete to each other in $G$. Moreover, note that the unique vertex of any trivial anticomponent of a graph is a universal vertex of that graph. 

A {\em cutset} of a graph $G$ is a (possibly empty) set $C \subsetneqq V(G)$ such that $G \setminus C$ is disconnected. A {\em clique-cutset} of a graph $G$ is a cutset of $G$ that is also a clique of $G$. (In particular, if $G$ is disconnected, then $\emptyset$ is a clique-cutset of $G$.) Note that if $v$ is a simplicial vertex of a graph $G$, then either $G$ is complete, or $N_G(v)$ is a clique-cutset of $G$. 

The {\em clique-width} of a graph $G$, denoted by $\text{cwd}(G)$, is the minimum number of labels needed to construct $G$ using the following four operations: 
\begin{enumerate} 
\item creation of a new vertex $v$ with label $i$; 
\item disjoint union of two labeled graphs; 
\item joining by an edge every vertex labeled $i$ to every vertex labeled $j$ (where $i \neq j$); 
\item renaming label $i$ to label $j$.  
\end{enumerate}

\subsection{A convention for our figures} 

In all our figures, crosshatched disks represent (possibly empty) cliques. A straight line between two crosshatched disks indicates that the two cliques are complete to each other. The presence of a wavy line between two crosshatched disks indicates that edges may possibly be present between the two cliques; however, adjacency between the two cliques is not arbitrary in this case, and instead, the exact details of adjacency are always given in the formal definition of the graph represented in the figure in question. If there is no line of any kind (straight or wavy) between two crosshatched disks, this indicates that the two cliques are anticomplete to each other.

\section{The main results of the first paper of the series} \label{sec:Part1}

In this subsection, we state the main results of~\cite{2P3C4C6FreePart1}. We begin with some definitions. $T_0$ and $T_1$ are the graphs represented in Figure~\ref{fig:T0T1}. (Note that $T_0$ is an induced subgraph of $T_1$.) Next, we define graphs $M_0$, $M_1$, $M_2$, and $M_3$, as well as the family $\mathcal{M}$ of graphs, as follows. 

$M_0$ is the graph with vertex set $\{x_0,x_1,x_2,x_3,x_4,x_5,x_6,y_0,y_3,z_0,z_3,z_4\}$ and with adjacency as follows: 
\begin{itemize} 
\item $x_0,x_1,x_2,x_3,x_4,x_5,x_6,x_0$ is a 7-hole; 
\item $y_0$ is complete to $\{x_0,x_1,x_4\}$ and anticompelte to $\{x_2,x_3,x_5,x_6\}$; 
\item $y_3$ is complete to $\{x_3,x_4,x_0\}$ and anticomplete to $\{x_1,x_2,x_5,x_6\}$; 
\item $z_0$ is complete to $\{x_0,x_1,x_2,x_3,x_4\}$ and anticomplete to $\{x_5,x_6\}$; 
\item $z_3$ is complete to $\{x_3,x_4,x_5,x_6,x_0\}$ and anticompelte to $\{x_1,x_2\}$; 
\item $z_4$ is complete to $\{x_4,x_5,x_6,x_0,x_1\}$ and anticompelte to $\{x_2,x_3\}$; 
\item $\{y_0,y_3,z_0,z_3,z_4\}$ is a clique. 
\end{itemize} 

$M_1$ is the graph with vertex set $\{x_0,x_1,x_2,x_3,x_4,x_5,x_6,y_0,z_2\}$ and with adjacency as follows: 
\begin{itemize} 
\item $x_0,x_1,x_2,x_3,x_4,x_5,x_6,x_0$ is a 7-hole; 
\item $y_0$ is complete to $\{x_0,x_1,x_4\}$ and anticomplete to $\{x_2,x_3,x_5,x_6\}$; 
\item $z_2$ is complete to $\{x_2,x_3,x_4,x_5,x_6\}$ and anticomplete to $\{x_0,x_1\}$; 
\item $y_0$ and $z_2$ are nonadjacent.  
\end{itemize} 

$M_2$ is the graph with vertex set $\{x_0,x_1,x_2,x_3,x_4,x_5,x_6,y_0,z_1,z_2\}$ and with adjacency as follows: 
\begin{itemize} 
\item $M_2 \setminus z_1 = M_1$; 
\item $z_1$ is complete to $\{x_1,x_2,x_3,x_4,x_5\}$ and anticomplete to $\{x_6,x_0\}$; 
\item $z_1$ is complete to $\{y_0,z_2\}$. 
\end{itemize} 

$M_3$ is the graph with vertex set $\{x_0,x_1,x_2,x_3,x_4,x_5,x_6,y_0,z_2,z_3\}$ and with adjacency as follows: 
\begin{itemize} 
\item $M_3 \setminus z_3 = M_1$; 
\item $z_3$ is complete to $\{x_3,x_4,x_5,x_6,x_0\}$ and anticomplete to $\{x_1,x_2\}$; 
\item $z_3$ is complete to $\{y_0,z_2\}$. 
\end{itemize} 

We now define the family $\mathcal{M}$, as follows: 
\begin{displaymath} 
\begin{array}{rcl} 
\mathcal{M} & := & \big\{M_0 \setminus S \mid S \subseteq \{y_0,y_3,z_0,z_3,z_4\}\big\} \cup \{M_1,M_2,M_3\}. 
\end{array} 
\end{displaymath} 
Thus, $\mathcal{M}$ is the collection of graphs whose members are $M_1,M_2,M_3$, plus all induced subgraphs of $M_0$ that still contain the hole $x_0,x_1,x_2,x_3,x_4,x_5,x_6,x_0$. We note that each graph in $\mathcal{M}$ has at most 12 vertices and contains an induced $C_7$. 

\medskip 

The following two theorems are the main results of~\cite{2P3C4C6FreePart1}. 

\begin{theorem} [Theorem~5.2 of~\cite{2P3C4C6FreePart1}] \label{thm-main-withC7T0} For any graph $G$, the following are equivalent: 
\begin{enumerate}[(a)] 
\item $G$ is $(2P_3,C_4,C_6)$-free, contains an induced $C_7$ or $T_0$, and contains no simplicial vertices; 
\item $G$ has exactly one nontrivial anticomponent, and this anticomponent is a thickening of a graph in $\mathcal{M} \cup \{T_0,T_1\}$; 
\item $G$ can be obtained from a thickening of a graph in $\mathcal{M} \cup \{T_0,T_1\}$ by possibly adding universal vertices to it. 
\end{enumerate} 
\end{theorem}

\begin{theorem} [Theorem~6.6 of~\cite{2P3C4C6FreePart1}]  \label{thm-cwd-in-class-with-C7-T0} Let $G$ be a $(2P_3,C_4,C_6)$-free graph that contains an induced $C_7$ or an induced $T_0$. Assume that $G$ contains no simplicial vertices. Then $\text{cwd}(G) \leq 12$. 
\end{theorem}

\section{A few simple propositions} \label{sec:simple}

This section is devoted to some simple propositions that we use in the remainder of the paper. The propositions stated in section~\ref{subsec:Part1SimpleResults} were proven in the first paper of our series. In subsection~\ref{subsec:simple-new}, we prove a few additional propositions. 

\subsection{A few simple propositions from the first paper of the series} \label{subsec:Part1SimpleResults} 

The propositions that we cite in this subsection were proven in subsection~2.2 of~\cite{2P3C4C6FreePart1}. 

\begin{proposition} \cite{2P3C4C6FreePart1} \label{prop-non-trivial-anticomp-univ-vertices} Let $G$ and $Q$ be graphs. Assume that $Q$ is anticonnected and contains at least two vertices. Then the following are equivalent: 
\begin{itemize} 
\item $G$ contains exactly one nontrivial anticomponent, and that anticomponent is $Q$; 
\item $G$ can be obtained from $Q$ by possibly adding universal vertices to it. 
\end{itemize} 
\end{proposition} 

\begin{proposition} \cite{2P3C4C6FreePart1} \label{prop-one-nontrivial-anticomp-simplicial} Assume that a graph $G$ contains exactly one nontrivial anticomponent, call it $Q$. Then $G$ contains a simplicial vertex if and only if $Q$ contains a simplicial vertex. 
\end{proposition}

\begin{proposition} \cite{2P3C4C6FreePart1} \label{prop-H-free-no-universal} Assume that a graph $G$ contains exactly one nontrivial anticomponent, call it $Q$. Let $H$ be a graph that contains no universal vertices. Then $G$ is $H$-free if and only if $Q$ is $H$-free. 
\end{proposition}

\begin{proposition} \cite{2P3C4C6FreePart1} \label{prop-non-adj-comp-clique} Let $G$ be a $C_4$-free graph, let $S \subseteq V(G)$, and let $v_1,v_2 \in V(G) \setminus S$ be distinct, nonadjacent vertices, complete to $S$ in $G$. Then $S$ is a clique. 
\end{proposition} 

\begin{proposition} \cite{2P3C4C6FreePart1} \label{prop-P3-free} A graph is $P_3$-free if and only if all its components are complete graphs. 
\end{proposition} 

\begin{proposition} \cite{2P3C4C6FreePart1} \label{prop-C4Free-CoBip} Let $G$ be a $C_4$-free graph, and let $X$ and $Y$ be disjoint cliques in $G$. Then all the following hold: 
\begin{enumerate}[(a)] 
\item for all $x,x' \in X$, one of $N_G(x) \cap Y$ and $N_G(x') \cap Y$ is a subset of the other; 
\item $X$ can be ordered as $X = \{x_1,\dots,x_{|X|}\}$ so that $N_G(x_{|X|}) \cap Y \subseteq \dots \subseteq N_G(x_1) \cap Y$; 
\item for all $y,y' \in Y$, one of $N_G(y) \cap X$ and $N_G(y') \cap X$ is a subset of the other; 
\item $Y$ can be ordered as $Y = \{y_1,\dots,y_{|Y|}\}$ so that $N_G(y_{|Y|}) \cap X \subseteq \dots \subseteq N_G(y_1) \cap X$. 
\end{enumerate} 
\end{proposition}

\begin{proposition} \cite{2P3C4C6FreePart1} \label{prop-2P3C4-free-clique-cut-simplicial} Let $G$ be a $(2P_3,C_4)$-free graph that admits a clique-cutset. Then $G$ contains a simplicial vertex. 
\end{proposition}

\begin{proposition} \cite{2P3C4C6FreePart1} \label{prop-G-x-disjoint-cliques-2P3-free} Let $G$ be a graph, and let $X \subsetneqq V(G)$. Assume that for all $x \in X$, the set $V(G) \setminus N_G[x]$ can be partitioned into cliques of $G$, pairwise anticomplete to each other.\footnote{It is possible that $V(G) \setminus N_G[x] = \emptyset$, i.e.\ that $N_G[x] = V(G)$.} Then $G$ is $2P_3$-free if and only if $G \setminus X$ is $2P_3$-free. 
\end{proposition}

\subsection{A few more simple propositions} \label{subsec:simple-new}

\begin{proposition} \label{prop-mixed-on-Yi} Let $\ell \geq 2$ be an integer, let $G$ be a graph, let $x \in V(G)$, and let $Y_1,\dots,Y_{\ell}$ be disjoint, nonempty subsets of $V(G) \setminus \{x\}$. Assume that $x$ is mixed on $Y_1 \cup \dots \cup Y_{\ell}$. Then there exist distinct $i,j \in \{1,\dots,\ell\}$ such that $x$ has a neighbor in $Y_i$ and a nonneighbor in $Y_j$. 
\end{proposition} 
\begin{proof} 
Since $x$ is mixed on $Y_1 \cup \dots \cup Y_{\ell}$, we know that there exist $y,y' \in Y_1 \cup \dots \cup Y_{\ell}$ such that $x$ is adjacent to $y$ and nonadjacent to $y'$. If $y,y'$ belong to different $Y_i$'s, then we are done. By symmetry, we may now assume that $y,y' \in Y_1$. Now, fix any $y_2 \in Y_2$. If $x$ is adjacent to $y_2$, then $x$ has a neighbor (namely $y_2$) in $Y_2$ and a nonneighbor (namely $y'$) in $Y_1$, and we are done. On the other hand, if $x$ is nonadjacent to $y_2$, then $x$ has a neighbor (namely $y$) in $Y_1$ and a nonneighbor (namely $y_2$) in $Y_2$, and again we are done. 
\end{proof}

\begin{proposition} \label{prop-C4-free-one-anticomp} Any $C_4$-free graph has at most one nontrivial anticomponent. 
\end{proposition} 
\begin{proof} 
We prove the contrapositive: if a graph has more than one nontrivial anticomponent, then it is not $C_4$-free. So, suppose that a graph $G$ has more than one nontrivial anticomponent, and let $A$ and $B$ be the vertex sets of two distinct nontrivial anticomponents of $G$. Then $A$ and $B$ are complete to each other, and each of them contains a pair of distinct, nonadjacent vertices. Fix distinct, nonadjacent vertices $a,a' \in A$, and fix distinct, nonadjacent vertices $b,b' \in B$. Then $a,b,a',b',a$ is a 4-hole in $G$, and it follows that $G$ is not $C_4$-free. 
\end{proof}


\begin{figure}
\begin{center}
\includegraphics[scale=0.5]{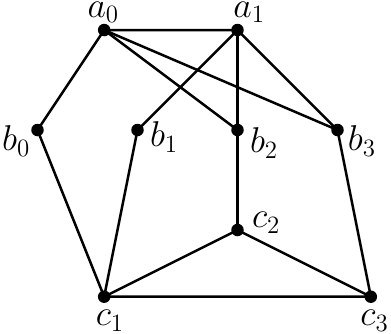}
\end{center} 
\caption{Graph $T_0$ (proof of Proposition~\ref{prop-S-clique-T0-free}).} \label{fig:T0-labeled} 
\end{figure} 

\begin{proposition} \label{prop-S-clique-T0-free} Let $G$ be a graph, and set $S := \big\{x \in V(G) \mid \alpha\big(G[N_G(x)]\big) \geq 3\big\}$.\footnote{Thus, $S$ is precisely the set of all vertices in $G$ that have three pairwise nonadjacent neighbors.} If $S$ is a clique, then $G$ is $T_0$-free. 
\end{proposition} 
\begin{proof} 
Note that $T_0$ contains a pair of nonadjacent vertices, each of which has three pairwise nonadjacent neighbors. Indeed, using the notation from Figure~\ref{fig:T0-labeled}, we see that vertices $a_1$ and $c_1$ are nonadjacent, that $a_1$ is adajcent to the pairwise nonadjacent vertices $b_1,b_2,b_3$, and that $c_1$ is adjacent to the pairwise nonadjacent vertices $b_0,b_1,c_2$. The result is now immediate. 
\end{proof}

\section{$\boldsymbol{(2P_3,C_4,C_6,C_7,T_0)}$-free graphs that contain an induced 3-pentagon} \label{sec:with-3-pentagon} 

Recall that the {\em 3-pentagon} is the graph represented in Figure~\ref{fig:tPentagon} (left); alternatively, see the definition at the beginning of subsection~\ref{subsec:BasicGraphsDef}. The main goal of this section is to prove Theorem~\ref{thm-T0free-pyramid-iff}, which gives a full structural description of $(2P_3,C_4,C_6,C_7,T_0)$-free graphs that contain an induced 3-pentagon and contain no simplicial vertices. The section is organized as follows. In subsection~\ref{subsec:BasicGraphsDef}, we define some ``basic graphs'' (those graphs that appear in the statement of Theorem~\ref{subsec:BasicGraphsDef}), and we prove a few preliminary results about those graphs. In subsection~\ref{subsec:frames}, we define ``$t$-frames'' ($t \geq 3$), and we prove a preliminary result about such graphs (see Lemma~\ref{lemma-t-frame}). In subsection~\ref{subsec:with3pentagon-main}, we prove Lemma~\ref{lemma-T0free-pyramid} and Theorem~\ref{thm-T0free-pyramid-iff}, where Lemma~\ref{lemma-T0free-pyramid} is the main technical result that we need to prove Theorem~\ref{thm-T0free-pyramid-iff} (the main result of this section). We note that $t$-frames play a crucial role in the proof of Lemma~\ref{lemma-T0free-pyramid}.

\subsection{Basic graphs: $\boldsymbol{t}$-pentagons, 5-baskets, villas, and mansions} \label{subsec:BasicGraphsDef}

For an integer $t \geq 3$, the {\em $t$-pentagon} is the $(2t+1)$-vertex graph $P$ with vertex set $V(P) = \{a,b_1,\dots,b_t,c_1,\dots,c_t\}$ and adjacency as follows: 
\begin{itemize} 
\item $a$ is complete $\{b_1,\dots,b_t\}$ and anticomplete to $\{c_1,\dots,c_t\}$; 
\item $\{b_1,\dots,b_t\}$ is a stable set; 
\item $\{c_1,\dots,c_t\}$ is a clique; 
\item for all $i,j \in \{1,\dots,t\}$, $b_i$ is adjacent to $c_j$ if and only if $i = j$. 
\end{itemize}

A {\em 5-basket} (originally defined in~\cite{4K1C4C6C7Free}) is a graph $Q$ whose vertex set can be partitioned into sets $A,B_1,B_2,B_3,C_1,C_2,C_3,F$ such that all the following hold: 
\begin{itemize} 
\item $A,B_1,B_2,B_3,C_1,C_2,C_3$ are nonempty cliques; 
\item $F$ is a (possibly empty) clique; 
\item cliques $B_1,B_2,B_3$ are pairwise anticomplete to each other; 
\item cliques $C_1,C_2,C_3$ are pairwise complete to each other; 
\item there exists an index $i^* \in \{1,2,3\}$ such that 
\begin{itemize} 
\item $A$ is complete to $(B_1 \cup B_2 \cup B_3) \setminus B_{i^*}$, and 
\item $A$ can be ordered as $A = \{a_1,\dots,a_t\}$ so that $N_Q(a_t) \cap B_{i^*} \subseteq \dots \subseteq N_Q(a_1) \cap B_{i^*} = B_{i^*}$;\footnote{Thus, $a_1$ is complete to $B_1 \cup B_2 \cup B_3$. Furthermore, $B_{i^*}$ can be ordered as $B_{i^*} = \{b_1,\dots,b_p\}$ so that $a_1 \in N_Q(b_p) \cap A \subseteq \dots \subseteq N_Q(b_1) \cap A$.} 
\end{itemize} 
\item $A$ is anticomplete to $C_1 \cup C_2 \cup C_3$; 
\item for all indices $i \in \{1,2,3\}$, $B_i$ is complete to $C_i$ and anticomplete to $(C_1 \cup C_2 \cup C_3) \setminus C_i$; 
\item there exists an index $j^* \in \{1,2,3\}$ such that $F$ is complete to $V(Q) \setminus (B_{j^*} \cup C_{j^*} \cup F)$ and anticomplete to $B_{j^*} \cup C_{j^*}$. 
\end{itemize} 
Under such circumstances, we say that $(A;B_1,B_2,B_3;C_1,C_2,C_3;F)$ is a {\em 5-basket partition} of the 5-basket $Q$. 

Note that there are effectively two different types of 5-basket (depending on whether or not $i^*$ and $j^*$ are the same). These two types of 5-basket (up to a permutation of the index set $\{1,2,3\}$) are represented in Figure~\ref{fig:5basket}. Note that the 3-pentagon is a 5-basket. 

\begin{figure} 
\begin{center}
\includegraphics[scale=0.5]{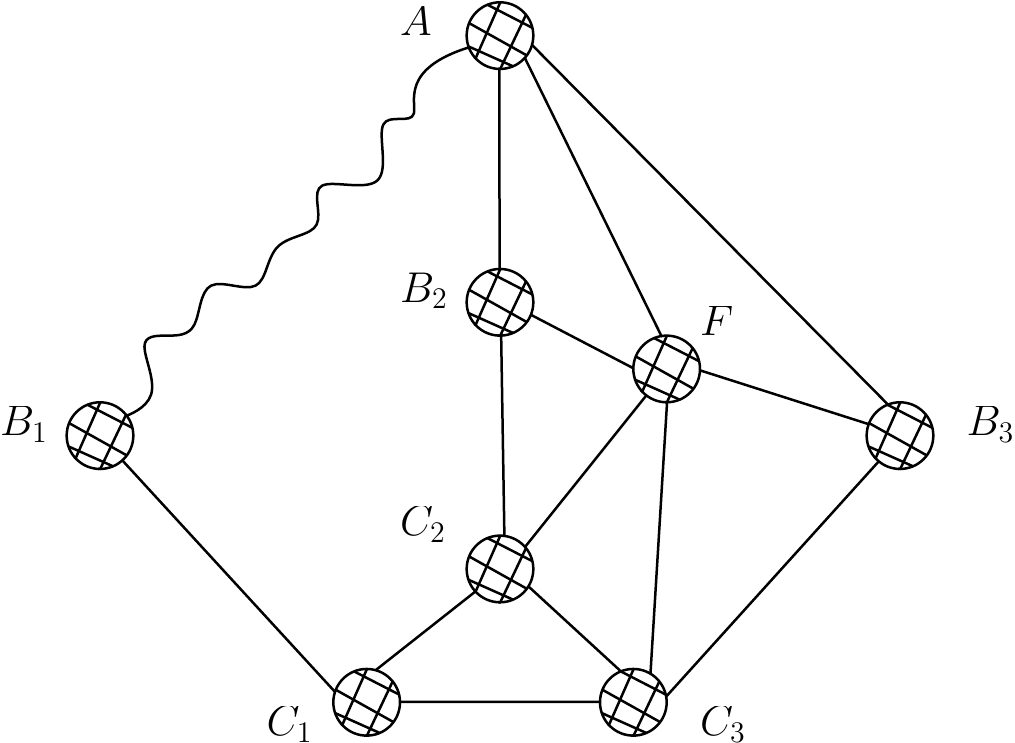}
\end{center} 

\vspace{1cm} 

\begin{center}
\includegraphics[scale=0.5]{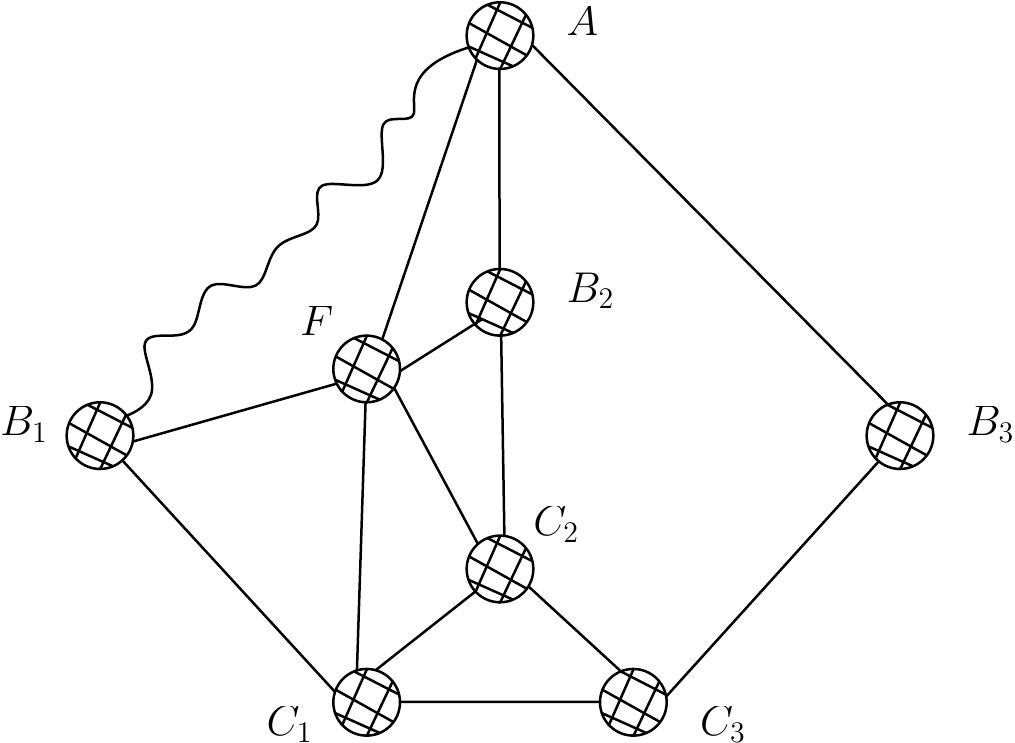}
\end{center} 
\caption{A 5-basket with an associated 5-basket partition $(A;B_1,B_2,B_3;C_1,C_2,C_3;F)$, and with $i^* = j^* = 1$ (top) or $i^* = 1$ and $j^* = 3$ (bottom). The clique $F$ may possibly be empty, and the remaining cliques represented by crosshatched disks are all nonempty.} \label{fig:5basket} 
\end{figure}

\begin{lemma} [Lemma~2.5 of~\cite{4K1C4C6C7Free}] \label{lemma-5-basket-4K1Free} Every 5-basket is $(4K_1,C_4,C_6,C_7)$-free and contains an induced 3-pentagon.
\end{lemma} 

\begin{proposition} \label{prop-5-basket-in-class} Every 5-basket is $(2P_3,C_4,C_6,C_7,T_0)$-free and contains an induced 3-pentagon. Moreover, every 5-basket is anticonnected and contains no simplicial and no universal vertices. 
\end{proposition} 
\begin{proof} 
By Lemma~\ref{lemma-5-basket-4K1Free}, every 5-basket is $(4K_1,C_4,C_6,C_7)$-free. Since both $2P_3$ and $T_0$ contain an induced $4K_1$, it follows that every 5-basket is $(2P_3,C_4,C_6,C_7,T_0)$-free. Further, Lemma~\ref{lemma-5-basket-4K1Free} guarantees that every 5-basket contains an induced 3-pentagon. 

It remains to show that 5-baskets are anticonnected and contain no simplicial and no universal vertices. So, fix a 5-basket $Q$, and let $(A;B_1,B_2,B_3;C_1,C_2,C_3;F)$ be a 5-basket partition of $Q$. Then $B_1,B_2,B_3$ are nonempty, pairwise disjoint, and pairwise anticomplete to each other; therefore, $Q[B_1 \cup B_2 \cup B_3]$ is anticonnected. Moreover, $F$ is anticomplete to one of $B_1,B_2,B_3$, and so $Q[B_1 \cup B_2 \cup B_3 \cup F]$ is anticonnected. Further, for each $i \in \{1,2,3\}$, $C_i$ is anticomplete to two of $B_1,B_2,B_3$; so, $Q[B_1 \cup B_2 \cup B_3 \cup C_1 \cup C_2 \cup C_3 \cup F] = Q \setminus A$ is anticonnected. Finally, $A$ is anticomplete to $C_1 \cup C_2 \cup C_3 \neq \emptyset$, and we deduce that $Q$ is anticonnected. Since $Q$ has more than one vertex, it follows that $Q$ has no universal vertices. It remains to show that $Q$ contains no simplicial vertices. First, every vertex in $B_1 \cup B_2 \cup B_3 \cup F$ has a neighbor both in $A$ and in $C_1 \cup C_2 \cup C_3$; since $A$ and $C_1 \cup C_2 \cup C_3$ are nonempty, disjoint, and anticomplete to each other, we see that no vertex in $B_1 \cup B_2 \cup B_3 \cup F$ is simplicial in $Q$. Further, $A$ is complete to at least two of $B_1,B_2,B_3$; since $B_1,B_2,B_3$ are nonempty, disjoint, and pairwise anticomplete to each other, it follows that no vertex in $A$ is simplicial in $Q$. Finally, for all $i \in \{1,2,3\}$, every vertex in $C_i$ has a neighbor both in $B_i$ and in $(C_1 \cup C_2 \cup C_3) \setminus C_i$, and we know that $B_i$ and $(C_1 \cup C_2 \cup C_3) \setminus C_i$ are nonempty, disjoint, and anticomplete to each other; therefore, no vertex in $C_1 \cup C_2 \cup C_3$ is simplicial in $Q$. This proves that $Q$ contains no simplicial vertices. 
\end{proof} 

\medskip 

For an integer $t \geq 3$, a {\em $t$-villa} is a graph $Q$ whose vertex set can be partitioned into nonempty cliques $A,B_1,\dots,B_t,C_1,\dots,C_t$, with adjacency as follows: 
\begin{itemize} 
\item $A$ is complete to $B_1 \cup \dots \cup B_t$ and anticomplete to $C_1 \cup \dots \cup C_t$; 
\item cliques $B_1,\dots,B_t$ are pairwise anticomplete to each other; 
\item cliques $C_1,\dots,C_t$ are pairwise complete to each other; 
\item for all distinct $i,j \in \{1,\dots,t\}$, $B_i$ is anticomplete to $C_j$; 
\item for all $i \in \{1,\dots,t\}$, $B_i$ can be ordered as $B_i = \{b_1^i,\dots,b_{r_i}^i\}$ so that $\emptyset \neq N_Q(b_{r_i}^i) \cap C_i \subseteq \dots \subseteq N_Q(b_1^i) \cap C_i = C_i$.\footnote{Note that this implies that $C_i$ can be ordered as $C_i = \{c_1^i,\dots,c_{s_i}^i\}$ so that $\emptyset \neq N_Q(c_{s_i}^i) \cap B_i \subseteq \dots \subseteq N_Q(c_1^i) \cap B_i = B_i$.}
\end{itemize} 
Under these circumstances, we say that $(A;B_1,\dots,B_t;C_1,\dots,C_t)$ is a {\em $t$-villa partition} of the $t$-villa $Q$. (See Figure~\ref{fig:villa}.) 

A {\em villa} is any graph that is a $t$-villa for some integer $t \geq 3$. 

Note that for each integer $t \geq 3$, the $t$-pentagon is a $t$-villa. 

\begin{figure} 
\begin{center}
\includegraphics[scale=0.5]{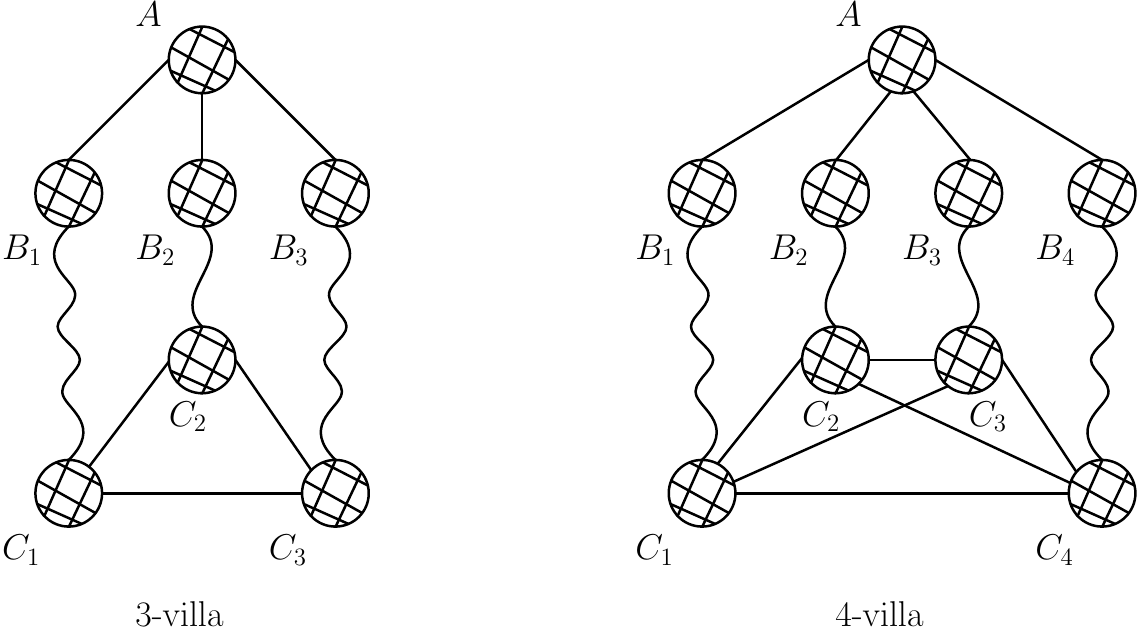}
\end{center} 
\caption{A 3-villa with an associated 3-villa partition $(A;B_1,B_2,B_3;C_1,C_2,C_3)$ (left), and a 4-villa with an associated 4-villa partition $(A;B_1,B_2,B_3,B_4;C_1,C_2,C_3,C_4)$. All cliques represented by crosshatched disks are nonempty.} \label{fig:villa} 
\end{figure}

\medskip 

A {\em simplicial elimination ordering} of a graph $G$ is an ordering $v_1,\dots,v_t$ of the vertices of $G$ such that for all $i \in \{1,\dots,t\}$, the vertex $v_i$ is simplicial in the graph $G \setminus \{v_1,\dots,v_{i-1}\}$. It is well known that a graph is chordal if and only if it admits a simplicial elimination ordering~\cite{FulkersonGross}. 

\begin{proposition} \label{prop-t-villa-in-class} Let $t \geq 3$ be an integer. Every $t$-villa is $(2P_3,C_4,C_6,C_7,T_0)$-free and contains an induced $t$-pentagon. Moreover, every $t$-villa is anticonnected and contains no simplicial and no universal vertices. 
\end{proposition} 
\begin{proof} 
Let $Q$ be a $t$-villa, and let $(A;B_1,\dots,B_t;C_1,\dots,C_t)$ be a $t$-villa partition of $Q$. Set $B := B_1 \cup \dots \cup B_t$ and $C := C_1 \cup \dots \cup C_t$. Further, for all $i \in \{1,\dots,t\}$: 
\begin{itemize} 
\item let $B_i = \{b_1^i,\dots,b_{r_i}^i\}$ be an ordering of $B_i$ such that $\emptyset \neq N_Q(b_{r_i}^i) \cap C_i \subseteq \dots \subseteq N_Q(b_1^i) \cap C_i = C_i$, 
\item let $C_i = \{c_1^i,\dots,c_{s_i}^i\}$ be an ordering of $C_i$ such that $\emptyset \neq N_Q(c_{s_i}^i) \cap B_i \subseteq \dots \subseteq N_Q(c_1^i) \cap B_i = B_i$, 
\end{itemize} 
as in the definition of a $t$-villa.

\begin{adjustwidth}{1cm}{1cm} 
\begin{claim} \label{lemma-t-villa-in-class-claim-t-pentagon}
$Q$ contains an induced $t$-pentagon. 
\end{claim} 
\end{adjustwidth} 
{\em Proof of Claim~\ref{lemma-t-villa-in-class-claim-t-pentagon}.} We fix any $a \in A$, and we observe that $Q[a,b_1^1,\dots,b_1^t,c_1^1,\dots,c_1^t]$ is a $t$-pentagon.~$\blacklozenge$

\begin{adjustwidth}{1cm}{1cm} 
\begin{claim} \label{lemma-t-villa-in-class-claim-anticonn} 
$Q$ is anticonnected and contains no simplicial and no universal vertices. 
\end{claim} 
\end{adjustwidth} 
{\em Proof of Claim~\ref{lemma-t-villa-in-class-claim-anticonn}.} We first show that $Q$ is anticonnected. First of all, we know that $t \geq 3$ and that the cliques $B_1,\dots,B_t$ are nonempty, pairwise disjoint, and pairwise anticomplete to each other; therefore, $Q[B]$ is anticonnected. Next, for all $i \in \{1,\dots,t\}$, $C_i$ is anticomplete to $B \setminus B_i \neq \emptyset$, and consequently, $Q[B \cup C]$ is anticonnected. Finally, $A$ is anticomplete to $C \neq \emptyset$, and so $Q[A \cup B \cup C] = Q$ is anticonnected. Since $Q$ contains more than one vertex, this implies, in particular, that $Q$ contains no universal vertices. 

It remains to show that $Q$ contains no simplicial vertices. First, $A$ is complete to $B$, and $B$ is not a clique; so, no vertex of $A$ is simplicial in $Q$. Next, every vertex in $B$ has a neighbor both in $A$ and in $C$; since $A$ and $C$ are nonempty, disjoint, and anticomplete to each other, we see that no vertex of $B$ is simlicial in $Q$. Finally, for all $i \in \{1,\dots,t\}$, every vertex in $C_i$ has a neighbor in $B_i$ and is complete to $C \setminus C_i \neq \emptyset$, and so since $B_i$ and $C \setminus C_i$ are disjoint and anticomplete to each other, we see that no vertex in $C_i$ is simplicial in $Q$. This proves that $Q$ has no simplicial vertices.~$\blacklozenge$

\medskip 

In view of Claims~\ref{lemma-t-villa-in-class-claim-t-pentagon} and~\ref{lemma-t-villa-in-class-claim-anticonn}, it now only remains to show that $Q$ is $(2P_3,C_4,C_6,C_7,T_0)$-free. 


\begin{adjustwidth}{1cm}{1cm} 
\begin{claim} \label{lemma-t-villa-in-class-claim-A-1}
Every hole in $Q$ contains exactly one vertex of $A$. 
\end{claim} 
\end{adjustwidth} 
{\em Proof of Claim~\ref{lemma-t-villa-in-class-claim-A-1}.} First of all, note that the following is a simplicial elimination ordering of $Q \setminus A$: 
\begin{displaymath} 
b_{r_1}^1,\dots,b_1^1,b_{r_2}^2,\dots,b_1^2,\dots,b_{r_t}^t,\dots,b_1^t,c_{s_1}^1,\dots,c_1^1,c_{s_2}^2,\dots,c_1^2,\dots,c_{s_t}^t,\dots,c_1^t. 
\end{displaymath} 
Therefore, by~\cite{FulkersonGross}, $Q \setminus A$ is chordal, and it follows that every hole in $Q$ contains at least one vertex of $A$. On the other hand, any two distinct vertices of $A$ are twins in $Q$, and no hole contains a pair of twins; therefore, no hole in $Q$ contains more than one vertex of $A$.~$\blacklozenge$

\begin{adjustwidth}{1cm}{1cm} 
\begin{claim} \label{lemma-t-villa-in-class-C5} 
Every hole in $Q$ is of length five. In particular, $Q$ is $(C_4,C_6,C_7)$-free. 
\end{claim} 
\end{adjustwidth} 
{\em Proof of Claim~\ref{lemma-t-villa-in-class-C5}.} Let $x_0,x_1,\dots,x_{k-1},x_0$ be a hole in $Q$ (with $k \geq 4$, and with indices in $\mathbb{Z}_k$). We must show that $k = 5$. By Claim~\ref{lemma-t-villa-in-class-claim-A-1}, and by symmetry, we may assume that $x_0 \in A$ and $x_1,\dots,x_{k-1} \notin A$. Since $A$ is complete to $B$ and anticomplete to $C$, we see that $x_{k-1},x_1 \in B$ and $x_2,\dots,x_{k-2} \in C$. Since $B_1,\dots,B_t$ are cliques, and since $x_1,x_{k-1}$ are nonadjacent (because $k \geq 4$), we see that vertices $x_1,x_{k-1}$ belong to distinct $B_i$'s. By symmetry, we may assume that $x_1 \in B_1$ and $x_{k-1} \in B_2$. But then $x_2 \in C_1$ and $x_{k-2} \in C_2$, and in particular, $x_2,x_{k-2}$ are adjacent. This implies that $k-2 = 3$, that is, $k = 5$.~$\blacklozenge$

\begin{adjustwidth}{1cm}{1cm} 
\begin{claim} \label{lemma-t-villa-in-class-claim-Q-2P3-free} 
$Q$ is $2P_3$-free. 
\end{claim} 
\end{adjustwidth} 
{\em Proof of Claim~\ref{lemma-t-villa-in-class-claim-Q-2P3-free}.} By the definition of a $t$-villa, $A$ is a clique, complete to $B$ and anticomplete to $C$. Therefore, for all $a \in A$, we have that $V(Q) \setminus N_Q[a] = C$, and in particular, $V(Q) \setminus N_Q[a]$ is clique. So, by Proposition~\ref{prop-G-x-disjoint-cliques-2P3-free}, it suffices to show that $Q_1 := Q \setminus A = Q[B \cup C]$ is $2P_3$-free. Now, since $C$ is a clique, we see that for all $c \in C$, we have that $C \subseteq N_{Q_1}[c]$, and consequently, $V(Q_1) \setminus N_{Q_1}[c] \subseteq B$. By the definition of a $t$-villa, $B$ can be partitioned into cliques (namely, $B_1,\dots,B_t$), pairwise anticomplete to each other. So, by Proposition~\ref{prop-G-x-disjoint-cliques-2P3-free}, it suffices to show that $Q_1 \setminus C = Q[B]$ is $2P_3$-free. But the components of $Q[B]$ are precisely the complete graphs $Q[B_1],\dots,Q[B_t]$, and so Proposition~\ref{prop-P3-free} guarantees that $Q[B]$ is $2P_3$-free.~$\blacklozenge$

\begin{adjustwidth}{1cm}{1cm} 
\begin{claim} \label{lemma-t-villa-in-class-claim-Q-T0-free} 
$Q$ is $T_0$-free. 
\end{claim} 
\end{adjustwidth} 
{\em Proof of Claim~\ref{lemma-t-villa-in-class-claim-Q-T0-free}.} In view of Proposition~\ref{prop-S-clique-T0-free}, it suffices to show that all vertices in $Q$ that have three pairwise nonadjacent neighbors belong to the clique $A$. For this, we just need to show that the closed neighborhood of each vertex in $V(Q) \setminus A = B \cup C$ is the union of two cliques. But this follows immediately from the definition of a $t$-villa. Indeed, fix any $x \in B \cup C$. By symmetry, we may assume that $x \in B_1 \cup C_1$. If $x \in B_1$, then $N_Q[x] \subseteq A \cup B_1 \cup C_1$, and $A \cup B_1$ and $C_1$ are cliques. On the other hand, if $x \in C_1$, then $N_Q[x] \subseteq B_1 \cup C$, and $B_1$ and $C$ are cliques.~$\blacklozenge$

\medskip 

Claims~\ref{lemma-t-villa-in-class-C5},~\ref{lemma-t-villa-in-class-claim-Q-2P3-free}, and~\ref{lemma-t-villa-in-class-claim-Q-T0-free} together imply that $Q$ is $(2P_3,C_4,C_6,C_7,T_0)$-free, and we are done. 
\end{proof} 

\medskip 

For an integer $t \geq 3$, a {\em $t$-mansion} is a graph $Q$ whose vertex set can be partitioned into cliques $A,B_1,\dots,B_t,C_1,\dots,C_t,F,X,Y$ such that all the following hold: 
\begin{itemize} 
\item cliques $A,B_1,\dots,B_t,C_1,\dots,C_t,F$ are all nonempty (cliques $X,Y$ may possibly be empty); 
\item $Q \setminus (F \cup X \cup Y)$ is a $t$-villa, and $(A;B_1,\dots,B_t;C_1,\dots,C_t)$ is an associated $t$-villa partition of $Q \setminus (F \cup X \cup Y)$; 
\item there exists some $j^* \in \{1,\dots,t\}$ such that all the following hold: 
\begin{itemize} 
\item $F$ is complete to $A$, $(B_1 \cup \dots \cup B_t) \setminus B_{j^*}$, and $(C_1 \cup \dots \cup C_t) \setminus C_{j^*}$, and is anticomplete to $B_{j^*} \cup C_{j^*}$, 
\item $B_{j^*}$ is complete to $C_{j^*}$, 
\item $X$ is complete to $A \cup B_{j^*}$ and is anticomplete to $(B_1 \cup \dots \cup B_t) \setminus B_{j^*}$ and $C_1 \cup \dots \cup C_t$, 
\end{itemize} 
\item $F$ is complete to $X \cup Y$; 
\item $X$ is anticomplete to $Y$; 
\item $Y$ is complete to $C_1 \cup \dots \cup C_t$ and anticomplete to $A \cup B_1 \cup \dots \cup B_t$. 
\end{itemize} 
Under these circumstance, we also say that $(A;B_1,\dots,B_t;C_1,\dots,C_t;F;X;Y)$ is {\em $t$-mansion partition} of the $t$-mansion $Q$. (A 3-mansion is represented in Figure~\ref{fig:3mansion}.) 

A {\em mansion} is any graph that is a $t$-mansion for some integer $t \geq 3$.  

\begin{figure} 
\begin{center}
\includegraphics[scale=0.5]{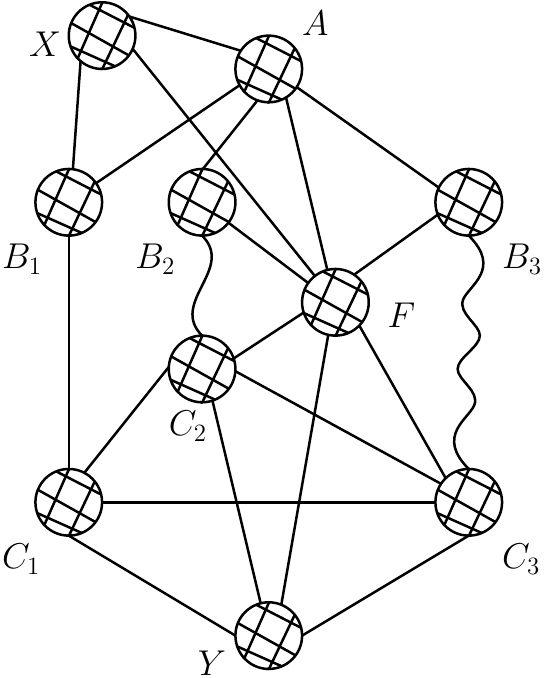}
\end{center} 
\caption{A 3-mansion with an associated 3-mansion partiton $(A;B_1,B_2,B_3;C_1,C_2,C_3;F;X;Y)$, and with $j^* = 1$. Cliques $X$ and $Y$ may possibly be empty; the other cliques represented by crosshatched disks are all nonempty.} \label{fig:3mansion} 
\end{figure}

\begin{proposition} \label{prop-t-mansion-in-class} Let $t \geq 3$ be an integer. Every $t$-mansion is $(2P_3,C_4,C_6,C_7,T_0)$-free and contains an induced $t$-pentagon. Moreover, every $t$-mansion is anticonnected and contains no simplicial and no universal vertices. 
\end{proposition} 
\begin{proof} 
Fix a $t$-mansion $Q$, and let $(A;B_1,\dots,B_t;C_1,\dots,C_t;F;X;Y)$ be a $t$-mansion partition of $Q$. Set $B := B_1 \cup \dots \cup B_t$ and $C := C_1 \cup \dots \cup C_t$. Let $j^*$ be as in the definition of a $t$-mansion; by symmetry, we may assume that $j^* = 1$.

\begin{adjustwidth}{1cm}{1cm} 
\begin{claim} \label{lemma-t-mansion-in-class-claim-Q-anticonn} 
$Q$ contains an induced $t$-pentagon. Moreover, $Q$ is anticonnected and contains no simplicial vertices and no universal vertices. 
\end{claim} 
\end{adjustwidth} 
{\em Proof of Claim~\ref{lemma-t-mansion-in-class-claim-Q-anticonn}.} First of all, by the definition of a $t$-mansion, $Q \setminus (F \cup X \cup Y)$ is a $t$-villa, and so by Proposition~\ref{prop-t-villa-in-class}, $Q \setminus (F \cup X \cup Y)$ contains an induced $t$-pentagon; consequently, $Q$ contains an induced $t$-pentagon. 

Next, by Proposition~\ref{prop-t-villa-in-class}, the $t$-villa $Q \setminus (F \cup X \cup Y)$ is anticonnected. Now, we know that $F$ is anticomplete to $B_1 \cup C_1$, that $X$ is anticompelte to $C$, and that $Y$ is anticomplete to $A$; since $B_1 \cup C_1$, $C$, and $A$ are all nonempty, we deduce that $Q$ is anticonnected. Since $Q$ contains more than one vertex, it follows that $Q$ contains no universal vertices. 

It remains to show that $Q$ contains no simplicial vertices. By Proposition~\ref{prop-t-villa-in-class}, the $t$-villa $Q \setminus (F \cup X \cup Y)$ contains no simplicial vertices; therefore, no vertex of $V(Q) \setminus (F \cup X \cup Y) = A \cup B \cup C$ is simplicial in $Q$. Further, we know that $F$ is complete to $A \cup (C \setminus C_1)$; since $A$ and $C \setminus C_1$ are nonempty, disjoint, and anticomplete to each other, we deduce that no vertex of $F$ is simplicial in $Q$. Next, $X$ is complete to $B_1 \cup F$; since $B_1$ and $F$ are nonempty, disjoint, and anticomplete to each other, it follows that no vertex of $X$ is simplicial in $Q$. Finally, $Y$ is complete to $C_1 \cup F$; since $C_1$ and $F$ are nonempty, disjoint, and anticomplete to each other, it follows that no vertex of $Y$ is simplicial in $Q$. This proves that $Q$ contains no simplicial vertices.~$\blacklozenge$ 

\medskip 

In view of Claim~\ref{lemma-t-mansion-in-class-claim-Q-anticonn}, it now only remains to show that $Q$ is $(2P_3,C_4,C_6,C_7)$-free. 

\begin{adjustwidth}{1cm}{1cm} 
\begin{claim} \label{lemma-t-mansion-in-class-claim-Q-F-hole-5} 
All holes in $Q \setminus F$ are of length five. 
\end{claim} 
\end{adjustwidth} 
{\em Proof of Claim~\ref{lemma-t-mansion-in-class-claim-Q-F-hole-5}.} First of all, by definition of a $t$-mansion, all vertices in $X \cup Y$ are simplicial in the graph $Q \setminus F$. Since no hole contains a simplicial vertex, we deduce that all holes in $Q \setminus F$ are in fact holes in $Q \setminus (F \cup X \cup Y)$. But by the definition of a $t$-mansion, $Q \setminus (F \cup X \cup Y)$ is a $t$-villa. So, by Proposition~\ref{prop-t-villa-in-class}, $Q \setminus (F \cup X \cup Y)$ is $(2P_3,C_4,C_6,C_7)$-free, and consequently, all holes in $Q \setminus (F \cup X \cup Y)$ are of length five.~$\blacklozenge$ 

\begin{adjustwidth}{1cm}{1cm} 
\begin{claim} \label{lemma-t-mansion-in-class-claim-Q-hole-5} 
All holes in $Q$ are of length five. In particular, $Q$ is $(C_4,C_6,C_7)$-free. 
\end{claim} 
\end{adjustwidth} 
{\em Proof of Claim~\ref{lemma-t-mansion-in-class-claim-Q-hole-5}.} Let $x_0,x_1,\dots,x_{k-1},x_0$ be a hole in $Q$ (with $k \geq 4$, and with indices in $\mathbb{Z}_k$). We must show that $k = 5$. We may assume that $\{x_0,x_1,\dots,x_{k-1}\} \cap F \neq \emptyset$, for otherwise, the result follows from Claim~\ref{lemma-t-mansion-in-class-claim-Q-F-hole-5}. Next, note that any two distinct vertices of $F$ are twins. Since no hole contains a pair of twins, we now deduce that $F$ contains exactly one vertex of the hole $x_0,x_1,\dots,x_{k-1},x_0$; by symmetry, we may assume that $x_0 \in F$ and $x_1,\dots,x_{k-1} \notin F$. Next, note that every vertex of $F$ dominates all vertices of $B_2 \cup \dots \cup B_t$ in $Q$; since no vertex of a hole dominates another vertex of that hole, we deduce that $\{x_0,x_1,\dots,x_{k-1}\} \cap (B_2 \cup \dots \cup B_t) = \emptyset$. Thus, $x_0,x_1,\dots,x_{k-1},x_0$ is a hole in $Q_1 := Q \setminus (B_2 \cup \dots \cup B_t) = Q[A \cup B_1 \cup C \cup F \cup X \cup Y]$. But note that $Q_1$ is a 5-hyperhole with an associated 5-hyperhole partition $(A \cup X,B_1,C_1,(C \setminus C_1) \cup Y,F)$. Clearly, all holes in a 5-hyperhole are of length five, and we deduce that $k = 5$.~$\blacklozenge$

\begin{adjustwidth}{1cm}{1cm} 
\begin{claim} \label{lemma-t-mansion-in-class-claim-Q-2P3-free} 
$Q$ is $2P_3$-free. 
\end{claim} 
\end{adjustwidth} 
{\em Proof of Claim~\ref{lemma-t-mansion-in-class-claim-Q-2P3-free}.} Note that for all $a \in A$, we have that $V(Q) \setminus N_Q[a] = C \cup Y$, whereas for all $f \in F$, we have that $V(Q) \setminus N_Q[f] = B_1 \cup C_1$. Since $C \cup Y$ and $B_1 \cup C_1$ are both cliques, Proposition~\ref{prop-G-x-disjoint-cliques-2P3-free} guarantees that $Q$ is $2P_3$-free if and only if $Q_1 := Q \setminus (A \cup F)$ is $2P_3$-free. So, it is enough to show that $Q_1$ is $2P_3$-free. Next, note that for all $c \in C \cup Y$, we have that $V(Q_1) \setminus N_{Q_1}[c] \subseteq B \cup X$, and that $B \cup X$ can be partitioned into cliques (namely, $B_1 \cup X,B_2,\dots,B_t$) that are pairwise anticomplete to each other. So, by Proposition~\ref{prop-G-x-disjoint-cliques-2P3-free}, it is enough to show that $Q_2 := Q_1 \setminus (C \cup Y)$ is $2P_3$-free. But note that $V(Q_2) = B \cup X$ can be partitioned into cliques (namely $B_1 \cup X,B_2,\dots,B_t$), pairwise anticomplete to each other. Thus, all components of $Q_2$ are complete graphs, and so Proposition~\ref{prop-P3-free} implies that $Q_2$ is $P_3$-free.~$\blacklozenge$ 

\begin{adjustwidth}{1cm}{1cm} 
\begin{claim} \label{lemma-t-mansion-in-class-claim-Q-T0-free} 
$Q$ is $T_0$-free. 
\end{claim} 
\end{adjustwidth} 
{\em Proof of Claim~\ref{lemma-t-mansion-in-class-claim-Q-T0-free}.} By the definition of a $t$-mansion, we know that $A \cup F$ is a clique. So, in view of Proposition~\ref{prop-S-clique-T0-free}, it suffices to show that all vertices of $Q$ that have three pairwise nonadjacent neighbors belong to $A \cup F$. Clearly, it is enough to show that the closed neighborhood of each vertex in $V(Q) \setminus (A \cup F) = B \cup C \cup X \cup Y$ can be partitioned into two cliques. Fix $z \in B \cup C \cup X \cup Y$. Since $j^* = 1$, we may assume by symmetry that $z \in B_1 \cup B_2 \cup C_1 \cup C_2 \cup X \cup Y$. Using the definition of a $t$-mansion, we easily see that all the following hold: 
\begin{itemize} 
\item if $z \in B_1$, then $N_Q[z] = A \cup B_1 \cup C_1 \cup X$, and $A \cup X$ and $B_1 \cup C_1$ are cliques; 
\item if $z \in B_2$, then $N_Q[z] \subseteq A \cup B_2 \cup C_2 \cup F$, and $A \cup B_2 \cup F$ and $C_2$ are cliques; 
\item if $z \in C_1$, then $N_Q[z] = B_1 \cup C \cup Y$, and $B_1$ and $C \cup Y$ are cliques; 
\item if $z \in C_2$, then $N_Q[z] \subseteq B_2 \cup C \cup F \cup Y$, and $B_2 \cup F$ and $C \cup Y$ are cliques; 
\item if $z \in X$, then $N_Q[z] = A \cup B_1 \cup F \cup X$, and $A \cup B_1$ and $F \cup X$ are cliques; 
\item if $z \in Y$, then $N_Q[z] = C \cup F \cup Y$, and $C$ and $F \cup Y$ are cliques. 
\end{itemize} 
So, in any case, $N_G[z]$ is the union of two cliques, which completes the proof of the claim.~$\blacklozenge$ 

\medskip 

Our proof is now complete: Claims~\ref{lemma-t-mansion-in-class-claim-Q-hole-5},~\ref{lemma-t-mansion-in-class-claim-Q-2P3-free}, and~\ref{lemma-t-mansion-in-class-claim-Q-T0-free} together guarantee that $Q$ is $(2P_3,C_4,C_6,C_7,T_0)$-free, which is what we needed to show. 
\end{proof} 

\subsection{Frames} \label{subsec:frames} 

For an integer $t \geq 3$, a {\em $t$-frame} is a graph $Q$ whose vertex set can be partitioned into nonempty sets $A,B_1,\dots,B_t,C_1,\dots,C_t$ such that all the following hold: 
\begin{itemize} 
\item every vertex in $A$ has a neighbor in at least two of $B_1,\dots,B_t$; 
\item $A$ is anticomplete to $C_1,\dots,C_t$; 
\item $B_1,\dots,B_t$ are pairwise anticomplete to each other; 
\item $C_1,\dots,C_t$ are pairwise complete to each other; 
\item for all distinct $i,j \in \{1,\dots,t\}$, $B_i$ is anticomplete to $C_j$; 
\item for all $i \in \{1,\dots,t\}$, all the following hold: 
\begin{itemize} 
\item every vertex in $B_i$ has a neighbor in $A$, 
\item every vertex in $B_i$ has a neighbor in $C_i$, 
\item every vertex in $C_i$ has a neighbor in $B_i$. 
\end{itemize} 
\end{itemize} 
Under these circumstances, we say that $(A;B_1,\dots,B_t;C_1,\dots,C_t)$ is a {\em $t$-frame partition} of the $t$-frame $Q$. 

\medskip 

Note that for any integer $t \geq 3$, the $t$-pentagon is a $t$-frame. As we shall see, $t$-frames will play a crucial role in the proof of Lemma~\ref{lemma-T0free-pyramid} (a technical lemma that is the main ingredient of our proof of Theorem~\ref{thm-T0free-pyramid-iff}, the main result of this section). For now, note that the definition of a $t$-frame places no restrictions at all on adjacency between vertices that belong to the same set of a $t$-frame partition, and consequently, $t$-frames may contain any graph at all as an induced subgraph. However, as Lemma~\ref{lemma-t-frame} (below) establishes, any $t$-frame that happens to be $(2P_3,C_4,C_6,C_7,T_0,\text{$(t+1)$-pentagon})$-free is either a villa or a 5-basket.

\begin{lemma} \label{lemma-t-frame} Let $t \geq 3$ be an integer, and let $Q$ be a $(2P_3,C_4,C_6,C_7,T_0,\text{$(t+1)$-pentagon})$-free $t$-frame with an associated $t$-frame partition $(A;B_1,\dots,B_t;C_1,\dots,C_t)$. Set $B := B_1 \cup \dots \cup B_t$ and $C := C_1 \cup \dots \cup C_t$. Let $A = \{a_1,\dots,a_r\}$ be an ordering of $A$ such that $d_Q(a_r) \leq \dots \leq d_Q(a_1)$. Further, for all $i \in \{1,\dots,t\}$, let $B_i = \{b_1^i,\dots,b_{r_i}^i\}$ be an ordering of $B_i$ such that $d_Q(b_{r_i}^i) \leq \dots \leq d_Q(b_1^i)$, and let $C_i = \{c_1^i,\dots,c_{s_i}^i\}$ be an ordering of $C_i$ such that $d_Q(c_{s_i}^i) \leq \dots \leq d_Q(c_1^i)$. Then all the following hold: 
\begin{enumerate}[(a)] 
\item \label{ref-lemma-t-frame-ABiCi-cliques} sets $A,B_1,\dots,B_t,C_1,\dots,C_t$ are cliques;\footnote{Since $C_1,\dots,C_t$ are complete to each other (by the definition of a $t$-frame), this in particular implies that $C$ is a clique.} 
\item \label{ref-lemma-t-frame-AB} either $A$ is complete to $B$, or all the following hold: 
\begin{itemize} 
\item $t = 3$, 
\item there exists some $i^* \in \{1,2,3\}$ such that $B \setminus B_{i^*} \subseteq N_Q(a_r) \cap B \subseteq \dots \subseteq N_Q(a_1) \cap B = B$, and in particular, $A$ is complete to $B \setminus B_{i^*}$, 
\item for all $i \in \{1,2,3\}$, $B_i$ is complete to $C_i$; 
\end{itemize} 
\item \label{ref-lemma-t-frame-BiCi-order} for all $i \in \{1,\dots,t\}$, both the following hold: 
\begin{itemize} 
\item $\{c_1^i\} \subseteq N_Q(b_{r_i}^i) \cap C_i \subseteq \dots \subseteq N_Q(b_1^i) \cap C_i = C_i$,\footnote{In particular, $b_1^i$ is complete to $C_i$, and $c_1^i$ is complete to $B_i$.} 
\item $\{b_1^i\} \subseteq N_Q(c_{s_i}^i) \cap B_i \subseteq \dots \subseteq N_Q(c_1^i) \cap B_i = B_i$; 
\end{itemize} 
\item \label{ref-lemma-t-frame-t-villa} if $A$ is complete to $B$, then $Q$ is a $t$-villa with an associated $t$-villa partition $(A;B_1,\dots,B_t;C_1,\dots,C_t)$; 
\item \label{ref-lemma-t-frame-5-basket} if $A$ is not complete to $B$, then $t = 3$, and $Q$ is a 5-basket with an associated 5-basket partition $(A;B_1,B_2,B_3;C_1,C_2,C_3;\emptyset)$. 
\end{enumerate} 
\end{lemma} 
\begin{proof} 
Clearly,~(\ref{ref-lemma-t-frame-t-villa}) follows from~(\ref{ref-lemma-t-frame-ABiCi-cliques}),~(\ref{ref-lemma-t-frame-BiCi-order}), and the definition of a $t$-frame. On the other hand,~(\ref{ref-lemma-t-frame-5-basket}) follows from~(\ref{ref-lemma-t-frame-ABiCi-cliques}),~(\ref{ref-lemma-t-frame-AB}), and the definition of a $t$-frame. Thus, it is enough to prove~(\ref{ref-lemma-t-frame-ABiCi-cliques}),~(\ref{ref-lemma-t-frame-AB}), and~(\ref{ref-lemma-t-frame-BiCi-order}), which we do by proving a sequence of claims. Before we begin, we observe that $Q$ is $P_7$-free (because $Q$ is $2P_3$-free, and $2P_3$ is an induced subgraph of $P_7$), a fact that we will use throughout our proof. We now introduce some notation. For each $a \in A$, let $\mathcal{I}_a$ be the set of all indices $i \in \{1,\dots,t\}$ such that $a$ has a neighbor in $B_i$. By the definition of a $t$-frame, we know that $|\mathcal{I}_a| \geq 2$ for all $a \in A$.

\begin{adjustwidth}{1cm}{1cm} 
\begin{claim} \label{lemma-t-frame-claim-Ia-nested} 
For all $a,a' \in A$, one of the sets $\mathcal{I}_{a},\mathcal{I}_{a'}$ is included in the other. 
\end{claim} 
\end{adjustwidth} 
{\em Proof of Claim~\ref{lemma-t-frame-claim-Ia-nested}.} Suppose otherwise, and fix $a,a' \in A$ such that $\mathcal{I}_{a} \not\subseteq \mathcal{I}_{a'}$ and $\mathcal{I}_{a'} \not\subseteq \mathcal{I}_{a}$. (In particular, $a \neq a'$.) By symmetry, we may assume that $1 \in \mathcal{I}_{a} \setminus \mathcal{I}_{a'}$ and $2 \in \mathcal{I}_{a'} \setminus \mathcal{I}_{a}$. Fix a neighbor $b_1 \in B_1$ of $a$, and fix a neighbor $b_2 \in B_2$ of $a'$. (Note that $a$ is anticompelte to $B_2$, and is consequently nonadjacent to $b_2$. Similarly, $a'$ is anticomplete to $B_1$, and is consequently nonadjacent to $b_1$.) For each $i \in \{1,2\}$, let $c_i \in C_i$ be a neighbor of $b_i$ (such a $c_i$ exists by the definition of a $t$-frame); by the definition of a $t$-frame, $c_1,c_2$ are adjacent. Then $a,a'$ are nonadjacent, for otherwise $a,a',b_2,c_2,c_1,b_1,a$ would be a 6-hole in $Q$, a contradiction. Since $|\mathcal{I}_{a}| \geq 2$, we may assume by symmetry that $3 \in \mathcal{I}_{a}$. Fix a neighbor $b_3 \in B_3$ of $a$. But if $a',b_3$ are adjacent, then $a',b_2,c_2,c_1,b_1,a,b_3,a'$ is a 7-hole in $Q$, whereas if $a',b_3$ are nonadjacent, then $a',b_2,c_2,c_1,b_1,a,b_3$ is an induced 7-vertex path in $Q$, a contradiction in either case.~$\blacklozenge$ 

\begin{adjustwidth}{1cm}{1cm} 
\begin{claim} \label{lemma-t-frame-claim-nonadj-Ia-t-1} 
For all $a \in A$, we have that $t-1 \leq |\mathcal{I}_a| \leq t$. 
\end{claim} 
\end{adjustwidth} 
{\em Proof of Claim~\ref{lemma-t-frame-claim-nonadj-Ia-t-1}.} Fix $a \in A$. The fact that $|\mathcal{I}_a| \leq t$ follows simply from the fact that $\mathcal{I}_a \subseteq \{1,\dots,t\}$. So, we just need to show that $|\mathcal{I}_a| \geq t-1$. Suppose otherwise, that is, assume that $|\mathcal{I}_a| \leq t-2$. By the definition of a $t$-frame, we have that $|\mathcal{I}_a| \geq 2$. So, $2 \leq |\mathcal{I}_a| \leq t-2$, and in particular, $t \geq 4$. By symmetry, we may assume that $1,2 \in \mathcal{I}_a$ and $3,4 \notin \mathcal{I}_a$. Fix $b_1 \in B_1$ and $b_2 \in B_2$ such that $a$ is complete to $\{b_1,b_2\}$. Fix any $b_3 \in B_3$ and $c_4 \in C_4$, and (using the definition of a $t$-frame) fix any neighbor $c_3 \in C_3$ of $b_3$. But then $Q[b_1,a,b_2,b_3,c_3,c_4]$ is a $2P_3$, a contradiction.~$\blacklozenge$

\begin{adjustwidth}{1cm}{1cm} 
\begin{claim} \label{lemma-t-frame-claim-nonadj-aaBi-nested} 
For all $i \in \{1,\dots,t\}$, and all distinct, nonadjacent vertices $a,a' \in A$, one of the sets $N_Q(a) \cap B_i$ and $N_Q(a') \cap B_i$ is included in the other. 
\end{claim} 
\end{adjustwidth}
{\em Proof Claim~\ref{lemma-t-frame-claim-nonadj-aaBi-nested}.} 
Suppose otherwise. By symmetry, we may assume that there exist distinct, nonadjacent vertices $a,a' \in A$ such that $N_Q(a) \cap B_1 \not\subseteq N_Q(a') \cap B_1$ and $N_Q(a') \cap B_1 \not\subseteq N_Q(a) \cap B_1$. Fix $b_1,b_1' \in B_1$ such that $b_1$ is adjacent to $a$ and nonadjacent to $a'$, and such that $b_1'$ is adjacent to $a'$ and nonadjacent to $a$. 

We begin by showing that $b_1,b_1'$ are nonadjacent. Suppose otherwise. Fix $b_2 \in B_2$ and $b_3 \in B_3$. If $\{a,a'\}$ is complete to $\{b_2,b_3\}$, then $a,b_2,a',b_3,a$ is a 4-hole in $Q$, a contradiction. So, by symmetry, we may assume that $a,b_2$ are nonadjacent. Using the definition of a $t$-frame, we now fix a neighbor $c_2 \in C_2$ of $b_2$, and we fix any $c_3 \in C_3$. But now $Q[a,b_1,b_1',b_2,c_2,c_3]$ is a $2P_3$, a contradiction. This proves that $b_1,b_1'$ are indeed nonadjacent. 

Next, we claim that $Q[A \cup B_1 \cup C_1]$ contains an induced path of length at least four between $a$ and $a'$, with all internal vertices of this path belonging to $B_1 \cup C_1$. By the definition of a $t$-frame, there exist some $c_1,c_1' \in C_1$ such that $b_1$ is adjacent to $c_1$, and $b_1'$ is adjacent to $c_1'$. If $b_1'$ is adjacent to $c_1$, then $a,b_1,c_1,b_1',a'$ is the path that we need. Similarly, if $b_1$ is adjacent to $c_1'$, then $a,b_1,c_1',b_1',a'$ is the path that we need. So, we may assume that $b_1$ is nonadjacent to $c_1'$, and that $b_1'$ is nonadjacent to $c_1$. In particular, $c_1 \neq c_1'$. Then $c_1$ and $c_1'$ are adjacent, for otherwise, $Q[a,b_1,c_1,a',b_1',c_1']$ would be a $2P_3$, a contradiction. But now $a,b_1,c_1,c_1',b_1',a'$ is the path that we need. This proves that there exist vertices $x_1,\dots,x_{\ell} \in B_1 \cup C_1$ such that $a,x_1,\dots,x_{\ell},a'$ is an induced path of length at least four in $Q$ (thus, $\ell \geq 3$). 

Now, since $Q$ is $(2P_3,C_4,C_6,C_7)$-free, we know that all holes in $Q$ are of length five. Moreover, since $Q$ is $P_7$-free, we know that all induced paths in $Q$ are of length at most five.\footnote{Recall that the {\em length} of a path is the number of edges that it contains. So, $P_7$ is of length six.} We will derive a contradiction by exhibiting either a hole of length at least six or an induced path of length at least six in $Q$. 

By Claim~\ref{lemma-t-frame-claim-Ia-nested}, and by symmetry, we may assume that $\mathcal{I}_{a'} \subseteq \mathcal{I}_a$. Moreover, since $|\mathcal{I}_{a'}| \geq 2$ (by the definition of a frame), we may assume by symmetry that $2 \in \mathcal{I}_{a'}$ (and consequently, $2 \in \mathcal{I}_a$). Fix $b_2,b_2' \in B_2$ such that $a$ is adjacent to $b_2$, and $a'$ is adjacent to $b_2'$. Then $a'$ is nonadjacent to $b_2$, for otherwise, $a',b_2,a,x_1,\dots,x_{\ell},a'$ would be a hole of length at least six in $Q$, a contradiction. Similarly, $a$ is nonadjacent to $b_2'$. (In particular, $b_2 \neq b_2'$.) But now if $b_2,b_2'$ are adjacent, then $b_2,a,x_1,\dots,x_{\ell},a',b_2',b_2$ is a hole of length at least seven in $Q$, and if $b_2,b_2'$ are nonadjacent, then $b_2,a,x_1,\dots,x_{\ell},a',b_2'$ is an induced path in $Q$ of length at least six, a contradiction in either case.~$\blacklozenge$

\begin{adjustwidth}{1cm}{1cm} 
\begin{claim} \label{lemma-t-frame-claim-nonadj-aaB-nested} 
For all distinct, nonadjacent vertices $a,a' \in A$, one of the sets $N_Q(a) \cap B$ and $N_Q(a') \cap B$ is included in the other. 
\end{claim} 
\end{adjustwidth}
{\em Proof Claim~\ref{lemma-t-frame-claim-nonadj-aaB-nested}.} Suppose otherwise. By Claim~\ref{lemma-t-frame-claim-nonadj-aaBi-nested}, and by symmetry, we may assume that for some distinct, nonadjacent vertices $a,a' \in A$, we have that $N_Q(a') \cap B_1 \subsetneqq N_Q(a) \cap B_1$ and $N_Q(a) \cap B_2 \subsetneqq N_Q(a') \cap B_2$. Fix $b_1 \in B_1$ that is adjacent to $a$ and nonadjacent to $a'$, and fix $b_2 \in B_2$ that is adjacent to $a'$ and nonadjacent to $a$. Fix a neighbor $c_1 \in C_1$ of $b_1$, and fix a neighbor $c_2 \in C_2$ of $b_2$. 

We now claim that $\{a,a'\}$ is anticomplete to $B_3 \cup \dots \cup B_t$. Suppose otherwise. By symmetry, we may assume that some $b_3 \in B_3$ is adjacent to $a$. But if $b_3$ is also adjacent to $a'$, then $b_3,a,b_1,c_1,c_2,b_2,a',b_3$ is a 7-hole in $Q$, whereas if $b_3$ is nonadjacent to $a'$, then $b_3,a,b_1,c_1,c_2,b_2,a'$ is an induced 7-vertex path in $Q$, a contradiction in either case. This proves that $\{a,a'\}$ is indeed anticomplete to $B_3 \cup \dots \cup B_t$. Consequently, $3,\dots,t \notin \mathcal{I}_{a} \cup \mathcal{I}_{a'}$. Since $|\mathcal{I}_{a}|,|\mathcal{I}_{a'}| \geq 2$ (by the definition of a $t$-frame), we now deduce that $\mathcal{I}_{a} = \mathcal{I}_{a'} = \{1,2\}$. 

Using the fact that $\mathcal{I}_{a} = \mathcal{I}_{a'} = \{1,2\}$, we now fix some $b_1' \in B_1$ that is adjacent to $a'$, and we fix some $b_2' \in B_2$ that is adjacent to $a$. Since $N_Q(a') \cap B_1 \subsetneqq N_Q(a) \cap B_1$ and $N_Q(a) \cap B_2 \subsetneqq N_Q(a') \cap B_2$, we deduce that $\{b_1',b_2'\}$ is in fact complete to $\{a,a'\}$. But then $a,b_1',a',b_2',a$ is a 4-hole in $G$, a contradiction.~$\blacklozenge$

\begin{adjustwidth}{1cm}{1cm} 
\begin{claim} \label{lemma-t-frame-claim-A-clique} 
$A$ is a clique. 
\end{claim} 
\end{adjustwidth} 
{\em Proof of Claim~\ref{lemma-t-frame-claim-A-clique}.} Suppose otherwise, and fix distinct, nonadjacent vertices $a,a' \in A$. By Claim~\ref{lemma-t-frame-claim-nonadj-aaB-nested}, and by symmetry, we may assume that $N_Q(a) \cap B \subseteq N_Q(a') \cap B$. By the definition of a $t$-frame, $a$ has a neighbor in at least two of $B_1,\dots,B_t$; by symmetry, we may assume that there exist $b_1 \in B_1$ and $b_2 \in B_2$ that are both adjacent to $a$. Since $N_Q(a) \cap B \subseteq N_Q(a') \cap B$, we see that $b_1,b_2$ are also adjacent to $a'$. But now $a,b_1,a',b_2,a$ is a 4-hole in $G$, a contradiction.~$\blacklozenge$

\begin{figure}
\begin{center}
\includegraphics[scale=0.5]{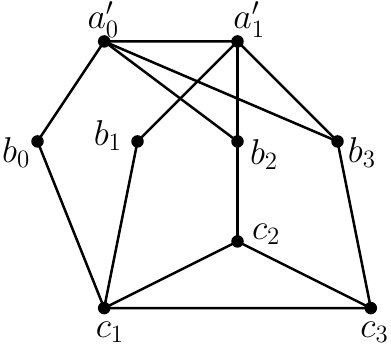}
\end{center} 
\caption{$Q[a_0',a_1',b_0,b_1,b_2,b_3,c_1,c_2,c_3]$ is a $T_0$ (proof of Claim~\ref{lemma-t-frame-claim-Bi-clique}).} \label{fig:lemma-t-frame-claim-Bi-clique} 
\end{figure}

\begin{adjustwidth}{1cm}{1cm} 
\begin{claim} \label{lemma-t-frame-claim-Bi-clique}
Sets $B_1,\dots,B_t$ are all cliques. 
\end{claim} 
\end{adjustwidth} 
{\em Proof of Claim~\ref{lemma-t-frame-claim-Bi-clique}.} Suppose otherwise. By symmetry, we may assume that some distinct vertices $b_0,b_1 \in B_1$ are nonadjacent. 

Suppose first that $b_0,b_1$ have a common neighbor $a \in A$. For each $i \in \{0,1\}$, fix a neighbor $c_i \in C_1$ of $b_i$. Then $b_0$ is nonadjacent to $c_1$, for otherwise, $a,b_0,c_1,b_1,a$ would be a 4-hole in $Q$, a contradiction. Analogously, $b_1$ is nonadjacent to $c_0$. In particular, $c_0 \neq c_1$. Then $c_0,c_1$ are adjacent, for otherwise, we fix any $c_2 \in C_2$, and we observe that $a,b_0,c_0,c_2,c_1,b_1,a$ is a 6-hole in $G$, a contradiction. Now, for all $i \in \{2,\dots,t\}$, fix an arbitrary $b_i \in B_i$, and fix a neighbor $c_i \in C_i$ of $b_i$. If $a$ is complete to $\{b_2,\dots,b_t\}$, then $Q[a,b_0,b_1,\dots,b_t,c_0,c_1,\dots,c_t]$ is a $(t+1)$-pentagon, a contradiction. Therefore, $a$ is not complete to $\{b_2,\dots,b_t\}$. By symmetry, we may assume that $a$ is nonadjacent to $b_2$. But now $Q[b_0,a,b_1,b_2,c_2,c_3]$ is a $2P_3$, a contradiction. 

Thus, $b_0,b_1$ have no common neighbors in $A$. However, by the definition of a $t$-frame, every vertex in $B$ has a neighbor in $A$. So, for each $i \in \{0,1\}$, fix a neighbor $a_i' \in A$ of $b_i$. Since $b_0,b_1$ have no common neighbors in $A$, we deduce that $a_0'$ is nonadjacent to $b_1$, and $b_0$ is nonadjacent to $a_1'$. In particular, $a_0' \neq a_1'$, and so by Claim~\ref{lemma-t-frame-claim-A-clique}, $a_0'$ is adjacent to $a_1'$. 

For each $i \in \{0,1\}$, fix a neighbor $c_i \in C_1$ of $b_i$. Suppose first that $b_0$ is nonadjacent to $c_1$, and $b_1$ is nonadjacent to $c_0$. Then $c_0,c_1$ are nonadjacent, for otherwise, $a_0',b_0,c_0,c_1,b_1,a_1',a_0'$ would be a 6-hole in $Q$, a contradiction. But now we fix an arbitrary $c_2 \in C_2$, and we observe that $a_0',b_0,c_0,c_2,c_1,b_1,a_1',a_0'$ is a 7-hole in $Q$, again a contradiction. This proves that either $b_0$ is adjacent to $c_1$, or $b_1$ is adjacent to $c_0$. By symmetry, we may assume that the former holds. Thus, $c_1$ is complete to $\{b_0,b_1\}$. 

We now claim that $\{a_0',a_1'\}$ is complete to $B_2 \cup \dots \cup B_t$. Suppose otherwise. By symmetry, we may assume that $a_0'$ has a nonneighbor $b_2 \in B_2$. Fix a neighbor $c_2 \in C_2$ of $b_2$, and fix any $c_3 \in C_3$. If $b_2$ is nonadjacent to $a_1'$, then $Q[a_0',a_1',b_1,b_2,c_2,c_3]$ is a $2P_3$, a contradiction. So, $b_2$ is adjacent to $a_1'$. But now $a_0',a_1',b_2,c_2,c_1,b_0,a_0'$ is a 6-hole in $Q$, a contradiction. This proves that $\{a_0',a_1'\}$ is complete to $B_2 \cup \dots \cup B_t$. 

Now, fix any $b_2 \in B_2$ and $b_3 \in B_3$, and for each $i \in \{2,3\}$, fix a neighbor $c_i \in C_i$ of $b_i$. But we now see that adjacency between vertices $a_0',a_1',b_0,b_1,b_2,b_3,c_1,c_2,c_3$ is exactly as in Figure~\ref{fig:lemma-t-frame-claim-Bi-clique}, and we deduce that $Q[a_0',a_1',b_0,b_1,b_2,b_3,c_1,c_2,c_3]$ is a $T_0$, contrary to the fact that $Q$ is $T_0$-free.~$\blacklozenge$

\begin{adjustwidth}{1cm}{1cm} 
\begin{claim} \label{lemma-t-frame-claim-a-comp-BBi}
Every vertex of $A$ is complete to all but possibly one of $B_1,\dots,B_t$. 
\end{claim} 
\end{adjustwidth} 
{\em Proof of Claim~\ref{lemma-t-frame-claim-a-comp-BBi}.} Suppose otherwise. By symmetry, we may assume that some $a \in A$ has a nonneighbor $b_1' \in B_1$ and a nonneighbor $b_2' \in B_2$. By Claim~\ref{lemma-t-frame-claim-nonadj-Ia-t-1}, $a$ has a neighbor in all but possibly one of $B_1,\dots,B_t$. Therefore, $a$ has a neighbor in at least one of $B_1,B_2$; by symmetry, we may assume that $a$ has a neighbor $b_1 \in B_1$. By Claim~\ref{lemma-t-frame-claim-Bi-clique}, $b_1,b_1'$ are adjacent. Now, using the definition of a $t$-frame, we fix a neighbor $c_2 \in C_2$ of $b_2'$, and we fix any $c_3 \in C_3$. Then $Q[a,b_1,b_1',b_2',c_2,c_3]$ is a $2P_3$, a contradiction.~$\blacklozenge$

\begin{adjustwidth}{1cm}{1cm} 
\begin{claim} \label{lemma-t-frame-claim-Na1Na2-subset}
For all $a,a' \in A$, one of the sets $N_Q(a) \cap B$ and $N_Q(a') \cap B$ is included in the other. 
\end{claim} 
\end{adjustwidth} 
{\em Proof of Claim~\ref{lemma-t-frame-claim-Na1Na2-subset}.} Suppose otherwise, and fix $a,a' \in A$ such that neither one of $N_Q(a) \cap B$ and $N_Q(a') \cap B$ is a subset of the other. In particular, $a \neq a'$, and by Claim~\ref{lemma-t-frame-claim-A-clique}, $a,a'$ are adjacent. Now, fix $b,b' \in B$ such that $b$ is adjacent to $a$ and nonadjacent to $a'$, and such that $b'$ is adjacent to $a'$ and nonadjacent to $a$. If there exists some $i \in \{1,\dots,t\}$ such that $b,b' \in B_i$, then Claim~\ref{lemma-t-frame-claim-Bi-clique} guarantees that $b,b'$ are adjacent, and we see that $a,b,b',a',a$ is a 4-hole in $Q$, a contradiction. Thus, there exist distinct $i,j \in \{1,\dots,t\}$ such that $b \in B_i$ and $b' \in B_j$. Using the definition of a $t$-frame, we fix a neighbor $c \in C_i$ of $b$, and we fix a neighbor $c' \in C_j$ of $b'$. But then $a,b,c,c',b',a',a$ is a 6-hole in $G$, a contradiction.~$\blacklozenge$

\begin{adjustwidth}{1cm}{1cm} 
\begin{claim} \label{lemma-t-frame-claim-Na-nested}
There exists an index $i^* \in \{1,\dots,t\}$ such that $A \cup (B \setminus B_{i^*}) \subseteq N_Q[a_r] \subseteq \dots \subseteq N_Q[a_1] = A \cup B$.  
\end{claim} 
\end{adjustwidth} 
{\em Proof of Claim~\ref{lemma-t-frame-claim-Na-nested}.} By Claim~\ref{lemma-t-frame-claim-A-clique}, $A$ is a clique, and by the definition of a $t$-frame, $N_Q(A) \subseteq B$. Therefore, for all $a \in A$, we have that $A \subseteq N_G[a] \subseteq A \cup B$. Since $d_Q(a_r) \leq \dots \leq d_Q(a_1)$, Claim~\ref{lemma-t-frame-claim-Na1Na2-subset} now implies that $A \subseteq N_Q[a_r] \subseteq \dots \subseteq N_Q[a_1] \subseteq A \cup B$. By Claim~\ref{lemma-t-frame-claim-a-comp-BBi}, we know that $a_r$ is complete to all but possibly one of $B_1,\dots,B_t$. Therefore, there exists an index $i^* \in \{1,\dots,t\}$ such that $A \cup (B \setminus B_{i^*}) \subseteq N_Q[a_r]$. It remains to show that $B \subseteq N_Q[a_1]$. Fix any $b \in B$. By the definition of a $t$-frame, $b$ has a neighbor in $A = \{a_1,\dots,a_r\}$. Since $N_Q[a_r] \subseteq \dots \subseteq N_Q[a_1]$, it follows that $b$ is adjacent to $a_1$, that is, $b \in N_Q[a_1]$. This proves that $B \subseteq N_Q[a_1]$, and we are done.~$\blacklozenge$ 

\begin{figure}
\begin{center}
\includegraphics[scale=0.5]{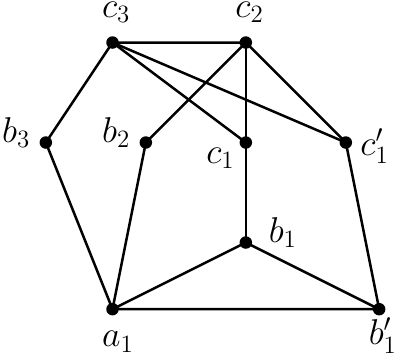}
\end{center} 
\caption{$Q[a_1,b_1,b_1',b_2,b_3,c_1,c_1',c_2,c_3]$ is a $T_0$ (proof of Claim~\ref{lemma-t-frame-claim-Ci-clique}).} \label{fig:lemma-t-frame-claim-Ci-clique} 
\end{figure}

\begin{adjustwidth}{1cm}{1cm} 
\begin{claim} \label{lemma-t-frame-claim-Ci-clique} 
$C_1,\dots,C_t$ are all cliques. 
\end{claim} 
\end{adjustwidth} 
{\em Proof of Claim~\ref{lemma-t-frame-claim-Ci-clique}.} Suppose otherwise. By symmetry, we may assume that $C_1$ is not a clique. Fix distinct, nonadjacent vertices $c_1,c_1' \in C_1$. 

Using the definition of a $t$-frame, we fix $b_1,b_1' \in B_1$ such that $b_1$ is adjacent to $c_1$, and $b_1'$ is adjacent to $c_1'$. If $b_1$ is adjacent to $c_1'$, then we fix an arbitrary $c_2 \in C_2$, and we observe that $b_1,c_1,c_2,c_1',b_1$ is a 4-hole in $Q$, a contradiction. Therefore, $b_1$ is nonadjacent to $c_1'$, and similarly, $b_1'$ is nonadjacent to $c_1$. In particular, $b_1 \neq b_1'$, and so by Claim~\ref{lemma-t-frame-claim-Bi-clique}, $b_1,b_1'$ are adjacent. 

By Claim~\ref{lemma-t-frame-claim-Na-nested}, $a_1$ is complete to $B$. Further, fix arbitrary $b_2 \in B_2$ and $b_3 \in B_3$. For each $i \in \{2,3\}$, fix a neighbor $c_i \in C_i$ of $b_i$. But then adjacency between vertices $a_1,b_1,b_1',b_2,b_3,c_1,c_1',c_2,c_3$ in $Q$ is as shown in Figure~\ref{fig:lemma-t-frame-claim-Ci-clique}, and in particular, these vertices induce a $T_0$ in $Q$, contrary to the fact that $Q$ is $T_0$-free.~$\blacklozenge$

\begin{adjustwidth}{1cm}{1cm} 
\begin{claim} \label{lemma-t-frame-claim-Nbi-nested} 
For all $i \in \{1,\dots,t\}$, and all $b,b' \in B_i$, one of the sets $N_Q[b]$ and $N_Q[b']$ is included in the other. 
\end{claim}  
\end{adjustwidth} 
{\em Proof of Claim~\ref{lemma-t-frame-claim-Nbi-nested}.} By symmetry, it suffices to prove the claim for $i = 1$. Fix $b_1,b_1' \in B_1$. By the definition of a $t$-frame, we know that $N_Q[b_1] \subseteq A \cup B_1 \cup C_1$ and $N_Q[b_1'] \subseteq A \cup B_1 \cup C_1$. Moreover, by Claim~\ref{lemma-t-frame-claim-Bi-clique}, we have that $B_1$ is a clique, and consequently, $B_1 \subseteq N_Q[b_1]$ and $B_1 \subseteq N_Q[b_1']$. (Note that this implies that $b_1,b_1'$ are adjacent.) So, it suffices to show that one of $N_Q(b_1) \cap (A \cup C_1)$ and $N_Q(b_1') \cap (A \cup C_1)$ is included in the other. Suppose otherwise. Fix $x,x' \in A \cup C_1$ such that $x$ is adjacent to $b_1$ and nonadjacent to $b_1'$, whereas $x'$ is adjacent to $b_1'$ and nonadjacent to $b_1$. (Clearly, $x \neq x'$.) If $x,x'$ are adjacent, then $b_1,x,x',b_1',b_1$ is a 4-hole in $Q$, a contradiction. Therefore, $x,x'$ are nonadjacent. By Claims~\ref{lemma-t-frame-claim-A-clique} and~\ref{lemma-t-frame-claim-Ci-clique}, $A$ and $C_1$ are both cliques. Therefore, one of $x,x'$ belongs to $A$ and the other one belongs to $C_1$. By symmetry, we may assume that $x \in A$ and $x' \in C_1$. By the definition of a $t$-frame, $x$ has a neighbor in at least two of $B_1,\dots,B_t$; consequently, $x$ has a neighbor in at least one of $B_2,\dots,B_t$, and by symmetry, we may assume that $x$ is adjacent to some $b_2 \in B_2$. Now, using the definition of a $t$-frame, we fix a neighbor $c_2 \in C_2$ of $b_2$, and we observe that $b_1,x,b_2,c_2,x',b_1',b_1$ is a 6-hole in $Q$, a contradiction.~$\blacklozenge$

\begin{adjustwidth}{1cm}{1cm} 
\begin{claim} \label{lemma-t-frame-claim-Nci-nested} 
For all $i \in \{1,\dots,t\}$, and all $c,c' \in C_i$, one of the sets $N_Q[c]$ and $N_Q[c']$ is included in the other. 
\end{claim}  
\end{adjustwidth} 
{\em Proof of Claim~\ref{lemma-t-frame-claim-Nci-nested}.} By symmetry, it suffices to prove the claim for $i = 1$. Fix $c_1,c_1' \in C_1$. By the definition of a $t$-frame, $N_Q[c_1] \subseteq B_1 \cup C$ and $N_Q[c_1'] \subseteq B_1 \cup C$. Next, by Claim~\ref{lemma-t-frame-claim-Ci-clique}, $C_1$ is a clique, and by the definition of a $t$-frame, $C_1$ is complete to $C \setminus C_1$. Therefore, $C \subseteq N_Q[c_1]$ and $C \subseteq N_Q[c_1']$, and it is enough to show that one of $N_Q(c_1) \cap B_1$ and $N_Q(c_1') \cap B_1$ is included in the other. But since $B_1$ and $C_1$ are both cliques (by Claims~\ref{lemma-t-frame-claim-Bi-clique} and~\ref{lemma-t-frame-claim-Ci-clique}), this follows immediately from Proposition~\ref{prop-C4Free-CoBip}.~$\blacklozenge$

\begin{adjustwidth}{1cm}{1cm} 
\begin{claim} \label{lemma-t-frame-claim-Bi-ordering-closed} 
For all $i \in \{1,\dots,t\}$, both the following hold: 
\begin{itemize} 
\item $N_Q[b_{r_i}^i] \subseteq \dots \subseteq N_Q[b_1^i]$; 
\item $N_Q[c_{s_i}^i] \subseteq \dots \subseteq N_Q[c_1^i]$. 
\end{itemize} 
\end{claim}  
\end{adjustwidth} 
{\em Proof of Claim~\ref{lemma-t-frame-claim-Bi-ordering-closed}.} Fix $i \in \{1,\dots,t\}$. Recall that $d_Q(b_{r_i}^i) \leq \dots \leq d_Q(b_1^i)$ and $d_Q(c_{s_i}^i) \leq \dots \leq d_Q(c_1^i)$. So, by Claims~\ref{lemma-t-frame-claim-Nbi-nested} and~\ref{lemma-t-frame-claim-Nci-nested}, we have that $N_Q[b_{r_i}^i] \subseteq \dots \subseteq N_Q[b_1^i]$ and $N_Q[c_{s_i}^i] \subseteq \dots \subseteq N_Q[c_1^i]$.~$\blacklozenge$


\begin{adjustwidth}{1cm}{1cm} 
\begin{claim} \label{lemma-t-frame-claim-Bi-ordering} 
For all $i \in \{1,\dots,t\}$, both the following hold: 
\begin{itemize} 
\item $\{c_1^i\} \subseteq N_Q(b_{r_i}^i) \cap C_i \subseteq \dots \subseteq N_Q(b_1^i) \cap C_i = C_i$; 
\item $\{b_1^i\} \subseteq N_Q(c_{s_i}^i) \cap B_i \subseteq \dots \subseteq N_Q(c_1^i) \cap B_i = B_i$. 
\end{itemize} 
\end{claim}  
\end{adjustwidth} 
{\em Proof of Claim~\ref{lemma-t-frame-claim-Bi-ordering}.} Fix $i \in \{1,\dots,t\}$. In view of Claim~\ref{lemma-t-frame-claim-Bi-ordering-closed}, it suffices to show that $b_1^i$ is complete to $C_i$, and that $c_1^i$ complete to $B_i$. But this follows immediately from Claim~\ref{lemma-t-frame-claim-Bi-ordering-closed}, and from the fact that, by the definition of a $t$-frame, every vertex in $B_i$ has a neighbor in $C_i$, and every vertex in $C_i$ has a neighbor in $B_i$.~$\blacklozenge$

\begin{adjustwidth}{1cm}{1cm} 
\begin{claim} \label{lemma-t-frame-claim-AB-not-comp} 
If $A$ is not complete to $B$, then both the following hold: 
\begin{itemize} 
\item $t = 3$; 
\item for all $i \in \{1,2,3\}$, $B_i$ is complete to $C_i$. 
\end{itemize} 
\end{claim} 
\end{adjustwidth} 
{\em Proof of Claim~\ref{lemma-t-frame-claim-AB-not-comp}.} Assume that $A$ is not complete to $B$. Recall that $A = \{a_1,\dots,a_r\}$, and that by Claim~\ref{lemma-t-frame-claim-Na-nested}, there exists an index $i^* \in \{1,\dots,t\}$ such that $A \cup (B \setminus B_{i^*}) \subseteq N_Q[a_r] \subseteq \dots \subseteq N_Q[a_1] = A \cup B$. By symmetry, we may assume that $i^* = 1$, so that $A \cup (B_2 \cup \dots \cup B_t) \subseteq N_Q[a_r] \subseteq \dots \subseteq N_Q[a_1] = A \cup B$. Since $A$ is not complete to $B$, we see that $a_r$ has a nonneighbor in $B_1$. On the other hand, $A$ is complete to $B_2 \cup \dots \cup B_t$. 

We first show that $t = 3$. Suppose otherwise, that is, suppose that $t \geq 4$. Fix a nonneighbor $b_1 \in B_1$ of $a_r$, fix a neighbor $c_1 \in C_1$ of $b_1$, and fix any $c_2 \in C_2$, $b_3 \in B_3$, and $b_4 \in B_4$. Since $B_2 \cup \dots \cup B_t \subseteq N_Q[a_r]$, we see that $a_r$ is complete to $\{b_3,b_4\}$. But now $Q[b_1,c_1,c_2,b_3,a_r,b_4]$ is a $2P_3$, a contradiction. This proves that $t = 3$. 

It remains to show that for all $i \in \{1,2,3\}$, $B_i$ is complete to $C_i$. 

We first show that $B_2$ is complete to $C_2$. Suppose otherwise, and fix nonadjacent vertices $b_2 \in B_2$ and $c_2 \in C_2$. Since $A$ is not complete to $B_1$, our ordering of $A$ guarantees that $a_r$ has a nonneighbor $b_1 \in B_1$. Fix a neighbor $c_1 \in C_1$ of $b_1$. Fix any $b_3 \in B_3$, and recall that $a_r$ is complete to $B_2 \cup B_3$; in particular, $a_r$ is complete to $\{b_2,b_3\}$. But now $Q[b_1,c_1,c_2,b_2,a_r,b_3]$ is a $2P_3$, a contradiction. This proves that $B_2$ is complete to $C_2$. Analogously, $B_3$ is complete to $C_3$. 

Next, we show that every vertex of $B_1$ is complete to at least one of $A,C_1$. Suppose otherwise, and fix a vertex $b_1 \in B_1$ that has a nonneighbor both in $A$ and in $C_1$. It then follows from our ordering of $A$ that $b_1$ is nonadjacent to $a_r$. Moreover, recall that $a_r$ is complete to $B_2 \cup B_3$. Fix any $b_2 \in B_2$ and $b_3 \in B_3$. Next, fix a nonneighbor $c_1' \in C_1$ of $b_1$. By the definition of a $t$-frame, $b_1$ has a neighbor $c_1 \in C_1$, and by Claim~\ref{lemma-t-frame-claim-Ci-clique}, $c_1$ and $c_1'$ are adjacent. But now $Q[b_1,c_1,c_1',b_2,a_r,b_3]$ is a $2P_3$, a contradiction. This proves that every vertex of $B_1$ is complete to at least one of $A,C_1$. 

We can now show that $B_1$ is complete to $C_1$. Suppose otherwise, and fix nonadjacent vertices $b_1 \in B_1$ and $c_1 \in C_1$. By what we just showed, $b_1$ is complete to $A$, and in particular, $b_1$ is adjacent to $a_r$. Now, recall that $a_r$ has a nonneighbor in $B_1$; so, fix some $b_1' \in B_1$ that is nonadjacent to $a_r$. Then $b_1'$ is adjacent to $c_1$, since otherwise, $b_1'$ would have a nonneighbor both in $A$ and in $C_1$, contrary to what we showed above. By Claim~\ref{lemma-t-frame-claim-Bi-clique}, $b_1,b_1'$ are adjacent. Next, fix any $b_2 \in B_2$, and fix a neighbor $c_2 \in C_2$ of $b_2$. Recall that $A$ is complete to $B_2$, and in particular, $a_r$ is adjacent to $b_2$. But now $a_r,b_1,b_1',c_1,c_2,b_2,a_r$ is a 6-hole in $Q$, a contradiction. This proves that $B_1$ is complete to $C_1$, and we are done.~$\blacklozenge$

\medskip 

Our proof is now complete: (\ref{ref-lemma-t-frame-ABiCi-cliques}) follows from Claims~\ref{lemma-t-frame-claim-A-clique}, \ref{lemma-t-frame-claim-Bi-clique}, and~\ref{lemma-t-frame-claim-Ci-clique}; (\ref{ref-lemma-t-frame-AB}) follows from Claims~\ref{lemma-t-frame-claim-Na-nested} and~\ref{lemma-t-frame-claim-AB-not-comp}; and (\ref{ref-lemma-t-frame-BiCi-order}) follows from Claim~\ref{lemma-t-frame-claim-Bi-ordering}. 
\end{proof}

\subsection{The structure of $\boldsymbol{(2P_3,C_4,C_6,C_7,T_0)}$-free graphs that contain an induced 3-pentagon and contain no simplicial vertices} \label{subsec:with3pentagon-main} 

In this subsection, we prove Lemma~\ref{lemma-T0free-pyramid} (a technical lemma) and Theorem~\ref{thm-T0free-pyramid-iff} (the main theorem of this section). We note that Lemma~\ref{lemma-T0free-pyramid} is the main ingredient of the proof of Theorem~\ref{thm-T0free-pyramid-iff}.

\begin{lemma} \label{lemma-T0free-pyramid} Let $G$ be a $(2P_3,C_4,C_6,C_7,T_0)$-free graph that contains an induced 3-pentagon. Then one of the following holds: 
\begin{itemize} 
\item $G$ is a 5-basket, a villa, or a mansion; 
\item $G$ contains a simplicial or a universal vertex. 
\end{itemize} 
\end{lemma} 
\begin{proof} 
By hypothesis, $G$ contains an induced 3-pentagon; let $t \geq 3$ be the largest integer such that $G$ contains an induced $t$-pentagon. Then $G$ is $(t+1)$-pentagon-free. Now, note that the $t$-pentagon is, in particular, a $t$-frame. Let $Q$ be a maximal induced $t$-frame in $G$,\footnote{As usual, this means that for all induced subgraphs $Q'$ of $G$ such that $Q$ is a proper induced subgraph of $Q'$, the graph $Q'$ is not a $t$-frame.} and let $(A;B_1,\dots,B_t;C_1,\dots,C_t)$ be a $t$-frame partition of $Q$. To simplify notation, set $B := B_1 \cup \dots \cup B_t$ and $C := C_1 \cup \dots \cup C_t$. Let $A = \{a_1,\dots,a_r\}$ be an ordering of $A$ such that $d_Q(a_r) \leq \dots \leq d_Q(a_1)$. Next, for all $i \in \{1,\dots,t\}$, let $B_i = \{b_1^i,\dots,b_{r_i}^i\}$ be an ordering of $B_i$ such that $d_Q(b_{r_i}^i) \leq \dots \leq d_Q(b_1^i)$, and let $C_i = \{c_1^i,\dots,c_{s_i}^i\}$ be an ordering of $C_i$ such that $d_Q(c_{s_i}^i) \leq \dots \leq d_Q(c_1^i)$. Since $G$ is $(2P_3,C_4,C_6,C_7,T_0,\text{$(t+1)$-pentagon})$-free, so is its induced subgraph $Q$. Thus, Lemma~\ref{lemma-t-frame} applies to the $t$-frame $Q$, to the $t$-frame partition $(A;B_1,\dots,B_t;C_1,\dots,C_t)$ of $Q$, and to our orderings of the sets $A,B_1,\dots,B_t,C_1,\dots,C_t$. In particular, we know that either $Q$ is a $t$-villa with an associated $t$-villa partition $(A;B_1,\dots,B_t;C_1,\dots,C_t)$, or $t = 3$ and $Q$ is a 5-basket with an associated 5-basket partition $(A;B_1,B_2,B_3;C_1,C_2,C_3;\emptyset)$. Furthermore (and still by Lemma~\ref{lemma-t-frame}), we know that $A,B_1,\dots,B_t,C_1,\dots,C_t$ are all nonempty cliques, and that the following hold: 
\begin{itemize} 
\item either $A$ is complete to $B$, or all the following hold: 
\begin{itemize} 
\item $t = 3$, 
\item there exists some $i^* \in \{1,2,3\}$ such that $B \setminus B_{i^*} \subseteq N_Q(a_r) \cap B \subseteq \dots \subseteq N_Q(a_1) \cap B = B$, and in particular, $A$ is complete to $B \setminus B_{i^*}$, 
\item for all $i \in \{1,2,3\}$, $B_i$ is complete to $C_i$; 
\end{itemize} 
\item for all $i \in \{1,\dots,t\}$, both the following hold: 
\begin{itemize} 
\item $\{c_1^i\} \subseteq N_Q(b_{r_i}^i) \cap C_i \subseteq \dots \subseteq N_Q(b_1^i) \cap C_i = C_i$,
\item $\{b_1^i\} \subseteq N_Q(c_{s_i}^i) \cap B_i \subseteq \dots \subseteq N_Q(c_1^i) \cap B_i = B_i$. 
\end{itemize} 
\end{itemize} 
Note that the first bullet point above in particular implies that $A$ is complete to all but possibly one of the cliques $B_1,\dots,B_t$, and moreover, $a_1$ is in fact complete to $B$. On the other hand, the second bullet point implies that for all indices $i \in \{1,\dots,t\}$, $b_1^i$ is complete to $C_i$, and $c_1^i$ is complete to $B_i$. Moreover, since $C_1,\dots,C_t$ are cliques, and since they are pairwise complete to each other (by the definition of a $t$-frame), we see that $C$ is a clique. 

\medskip 

We now define sets $D_1,\dots,D_t,F_1,\dots,F_t,X_1,\dots,X_t,Y,Z,W$ as follows: 
\begin{itemize} 
\item for all $i \in \{1,\dots,t\}$, set $D_i := \big\{v \in V(G) \setminus V(Q) \mid N_G(v) \cap V(Q) = B_i \cup C_i\big\}$; 
\item for all $i \in \{1,\dots,t\}$, set $F_i := \big\{v \in V(G) \setminus V(Q) \mid N_G(v) \cap V(Q) = A \cup (B \setminus B_i) \cup (C \setminus C_i)\big\}$; 
\item for all $i \in \{1,\dots,t\}$, set $X_i := \big\{v \in V(G) \setminus V(Q) \mid B_i \subseteq N_G(v) \subseteq A \cup B_i\big\}$; 
\item set $Y := \big\{v \in V(G) \setminus V(Q) \mid \emptyset \neq N_G(v) \cap V(Q) \subseteq C\big\}$; 
\item set $Z := \big\{v \in V(G) \setminus V(Q) \mid N_G(v) \cap V(Q) \subseteq A\big\}$; 
\item set $W := \big\{v \in V(G) \setminus V(Q) \mid N_G(v) \cap V(Q) = V(Q)\big\}$. 
\end{itemize} 
To simplify notation, we further set 
\begin{itemize} 
\item $D := D_1 \cup \dots \cup D_t$; 
\item $F := F_1 \cup \dots \cup F_t$; 
\item $X := X_1 \cup \dots \cup X_t$. 
\end{itemize}

In what follows, we will prove a lengthy sequence of claims about the sets that we just defined, that is, sets $D_1,\dots,D_t,F_1,\dots,F_t,X_1,\dots,X_t,Y,Z,W$. Before going into technical details, let us give a brief outline of the entire proof. Our first goal is to prove that these sets form a partition of $V(G) \setminus V(Q)$ (see Claim~\ref{lemma-T0free-pyramid-claim-sets-partition-VG-VQ}). Then, we examine the relationship between these sets. In particular, it will turn out that not all of these sets may simultaneously be nonempty, most of them are cliques, and adjacency between them is not arbitrary. If all of these sets are empty, then $G = Q$ (in which case $G$ is a 5-basket or a villa). Otherwise, we will see that either $G$ is a 5-basket or a mansion (this can only happen when almost all of our sets are empty), or $G$ contains a simplicial or a universal vertex. We now continue our formal proof. 

\begin{adjustwidth}{1cm}{1cm} 
\begin{claim} \label{lemma-T0free-pyramid-claim-anticomp-B} 
For all $v \in V(G) \setminus V(Q)$, if $v$ is anticomplete to $B$, then $v \in Y \cup Z$. 
\end{claim} 
\end{adjustwidth} 
{\em Proof of Claim~\ref{lemma-T0free-pyramid-claim-anticomp-B}.} Fix $v \in V(G) \setminus V(Q)$ that is anticomplete to $B$, i.e.\ $N_G(v) \cap V(Q) \subseteq A \cup C$. Clearly, it is enough to show that $v$ is anticomplete to $A$ or $C$, for it will then immediately follow that $v \in Y \cup Z$. Suppose otherwise. By symmetry, we may assume that $v$ is adjacent to some $a \in A$ and $c_1 \in C_1$. Using the definition of a $t$-frame, fix a neighbor $b_1 \in B_1$ of $c_1$. Then $a$ is nonadjacent to $b_1$, for otherwise, $v,a,b_1,c_1,v$ would be a 4-hole in $G$, a contradiction. Similarly, $v$ is nonadjacent to $a_1$, for otherwise, $v,a_1,b_1,c_1,v$ would be a 4-hole in $G$, a contradiction.\footnote{We are using the fact that $a_1$ is complete to $B$.} Note that this, in particular, implies that $a \neq a_1$; since $A$ is a clique, it follows that $a,a_1$ are adjacent. 

Since $a$ is nonadjacent to $b_1$, we see that $A$ is not complete to $B_1$. So, by Lemma~\ref{lemma-t-frame}(\ref{ref-lemma-t-frame-AB}), $A$ is complete to $B \setminus B_1 = B_2 \cup \dots \cup B_t$. Now, fix any $b_2 \in B_2$ and $b_3 \in B_3$. By the definition of a $t$-frame, $b_2$ has a neighbor $c_2 \in C_2$, and $b_3$ has a neighbor $c_3 \in C_3$. Then $v$ is nonadjacent to $c_2$, for otherwise, $v,a,b_2,c_2,v$ would be a 4-hole in $G$, a contradiction. Similarly, $v$ is nonadjacent to $c_3$. Moreover, by the definition of a $t$-frame, $\{b_1,b_2,b_3\}$ is a stable set, whereas $\{c_1,c_2,c_3\}$ is a clique. But now adjacency between vertices $a,a_1,v,b_1,b_2,b_3,c_1,c_2,c_3$ is as in Figure~\ref{fig:thm-T0free-pyramid-claim-anticomp-B}, and so $G[a,a_1,v,b_1,b_2,b_3,c_1,c_2,c_3]$ is a $T_0$, contrary to the fact that $G$ is $T_0$-free.~$\blacklozenge$ 

\begin{figure}
\begin{center}
\includegraphics[scale=0.5]{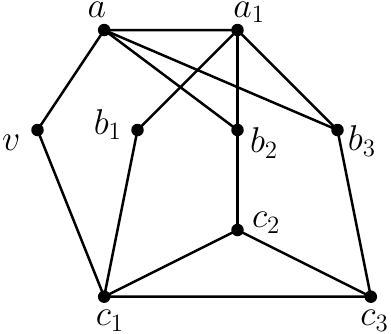}
\end{center} 
\caption{$G[a,a_1,v,b_1,b_2,b_3,c_1,c_2,c_3]$ is a $T_0$ (proof of Claim~\ref{lemma-T0free-pyramid-claim-anticomp-B}).} \label{fig:thm-T0free-pyramid-claim-anticomp-B} 
\end{figure}

\begin{adjustwidth}{1cm}{1cm} 
\begin{claim} \label{lemma-T0free-pyramid-claim-BiBj} 
If a vertex $v \in V(G) \setminus V(Q)$ has a neighbor in at least two of $B_1,\dots,B_t$, then $v$ is complete to $A$. 
\end{claim} 
\end{adjustwidth} 
{\em Proof of Claim~\ref{lemma-T0free-pyramid-claim-BiBj}.} Suppose otherwise. By symmetry, we may assume that there exists a vertex $v \in V(G) \setminus V(Q)$ that is adjacent to some $b_1 \in B_1$ and $b_2 \in B_2$, but is nonadjacent to some $a \in A$. Then $a$ is not complete to $\{b_1,b_2\}$, for otherwise, $v,b_1,a,b_2,v$ would be a 4-hole in $G$, a contradiction. By symmetry, we may assume that $a$ is nonadjacent to $b_1$. On the other hand, by Lemma~\ref{lemma-t-frame}(\ref{ref-lemma-t-frame-AB}), $a_1$ is complete to $B$, and so $v$ is adjacent to $a_1$, for otherwise, $a_1,b_1,v,b_2,a_1$ would be a 4-hole in $G$, a contradiction. 

Now, since $a$ is nonadjacent to $b_1$, we see, in particular, that $A$ is not complete to $B_1$. So, by Lemma~\ref{lemma-t-frame}(\ref{ref-lemma-t-frame-AB}), we have that $t = 3$, and that $A$ is complete to $B \setminus B_1 = B_2 \cup B_3$. Now, fix any $b_3 \in B_3$, and using the definition of a $t$-frame, fix any neighbor $c_3 \in C_3$ of $b_3$. First of all, $v$ is nonadjacent to $b_3$, for otherwise, $v,b_2,a,b_3,v$ would be a 4-hole in $G$, a contradiction. Then $v$ is also nonadjacent to $c_3$, for otherwise, $v,a_1,b_3,c_3,v$ would be a 4-hole in $G$, again a contradiction. Now, using the definition of a $t$-frame, we fix any neighbor $c_1 \in C_1$ of $b_1$. Then $v$ is nonadjacent to $c_1$, for otherwise, $v,c_1,c_3,b_3,a,b_2,v$ would be a 6-hole in $G$, a contradiction. But now $v,b_1,c_1,c_3,b_3,a,b_2,v$ is a 7-hole in $G$, again a contradiction.~$\blacklozenge$

\begin{adjustwidth}{1cm}{1cm} 
\begin{claim} \label{lemma-T0free-pyramid-claim-BiCj} 
For all $v \in V(G) \setminus V(Q)$ and $i \in \{1,\dots,t\}$, if $v$ has a neighbor both in $B_i$ and in $C \setminus C_i$, then $v$ is complete to $A$. 
\end{claim} 
\end{adjustwidth} 
{\em Proof of Claim~\ref{lemma-T0free-pyramid-claim-BiCj}.} Suppose otherwise. By symmetry, we may assume that some vertex $v \in V(G) \setminus V(Q)$ has a neighbor both in $B_1$ and in $C \setminus C_1 = C_2 \cup \dots \cup C_t$, but is not complete to $A$. By Claim~\ref{lemma-T0free-pyramid-claim-BiBj}, $v$ is anticomplete to $B \setminus B_1 = B_2 \cup \dots \cup B_t$. 

We first show that $v$ is nonadjacent to $a_1$. Suppose otherwise. Since $v$ has a neighbor in $C_2 \cup \dots \cup C_t$, we may assume by symmetry that $v$ is adjacent to some $c_2 \in C_2$. But then since $v$ is anticomplete to $B_2$, we see that $v,a_1,b_1^2,c_2,v$ is a 4-hole in $G$, a contradiction. Thus, $v$ is indeed nonadjacent to $a_1$. 

Next, we claim that $v$ is complete to $C \setminus C_1 = C_2 \cup \dots \cup C_t$. Suppose otherwise. Since $v$ has a neighbor in $C_2 \cup \dots \cup C_t$, we see that $v$ is mixed on $C_2 \cup \dots \cup C_t$. By Proposition~\ref{prop-mixed-on-Yi}, and by symmetry, we may now assume that $v$ is adjacent to some $c_2 \in C_2$ and nonadjacent to some $c_3 \in C_3$. We now fix any neighbor $b_1 \in B_1$ of $v$, and we observe that  $v,b_1,a_1,b_1^3,c_3,c_2,v$ is a 6-hole in $G$, a contradiction. This proves that $v$ is indeed complete to $C \setminus C_1 = C_2 \cup \dots \cup C_t$. 

Now, if $v$ is anticomplete to $A$, then $G[V(Q) \cup \{v\}]$ is a $t$-frame, with an associated $t$-frame partition $(A;B_1,B_2,\dots,B_t;C_1 \cup \{v\},C_2,\dots,C_t)$, contrary to the maximality of the $t$-frame $Q$. So, $v$ has a neighbor $a \in A$. By the definition of a $t$-frame, $a$ has a neighbor in at least two of $B_1,\dots,B_t$, and consequently, $a$ has a neighbor in at least one of $B_2,\dots,B_t$. By symmetry, we may assume that $a$ has a neighbor $b_2 \in B_2$. But then $v,a,b_2,c_1^2,v$ is a 4-hole in $G$, a contradiction.~$\blacklozenge$

\begin{adjustwidth}{1cm}{1cm} 
\begin{claim} \label{lemma-T0free-pyramid-claim-not-mixed-Bi}
No vertex in $V(G) \setminus V(Q)$ is mixed on any one of $B_1,\dots,B_t$. 
\end{claim} 
\end{adjustwidth} 
{\em Proof of Claim~\ref{lemma-T0free-pyramid-claim-not-mixed-Bi}.} Suppose otherwise. By symmetry, we may assume that some vertex $v \in V(G) \setminus V(Q)$ is mixed on $B_1$. Fix $b_1,b_1' \in B_1$ such that $v$ is adjacent to $b_1$ and nonadjacent to $b_1'$.\footnote{Recall that $A,B_1,\dots,B_t,C_1,\dots,C_t$ are all cliques. In particular, $b_1,b_1'$ are adjacent.} First of all, note that $v$ has a neighbor in $(B \setminus B_1) \cup (C \setminus C_1)$, for otherwise, $G[v,b_1,b_1',b_1^2,c_1^2,c_1^3]$ would be a $2P_3$, a contradiction. Since $v$ has a neighbor (namely $b_1$) in $B_1$, Claims~\ref{lemma-T0free-pyramid-claim-BiBj} and~\ref{lemma-T0free-pyramid-claim-BiCj} now guarantee that $v$ is complete to $A$, and in particular, $v$ is adjacent to $a_1$. Then $v$ is nonadjacent to $c_1^1$, for otherwise, $v,a_1,b_1',c_1^1,v$ would be a 4-hole in $G$, a contradiction. It now follows that $v$ is anticomplete to $C \setminus C_1 = C_2 \cup \dots \cup C_t$, for otherwise, we fix a neighbor $c \in C \setminus C_1$ of $v$, and we observe that $v,c,c_1^1,b_1,v$ is a 4-hole in $G$, a contradiction. So, $v$ has a neighbor in $B \setminus B_1$,\footnote{Indeed, we showed above that $v$ has a neighbor in $(B \setminus B_1) \cup (C \setminus C_1)$, but is anticomplete to $C \setminus C_1$. So, $v$ has a neighbor in $B \setminus B_1$} and by symmetry, we may assume that $v$ is adjacent to some $b_2 \in B_2$. If $v$ is anticomplete to $C_1$, then $G[V(Q) \cup \{v\}]$ is a $t$-frame, with an associated $t$-frame partition $(A \cup \{v\};B_1,\dots,B_t;C_1,\dots,C_t)$,\footnote{This is because $v$ is adjacent to $b_1 \in B_1$ and to $b_2 \in B_2$ (and in particular, $v$ has a neighbor in at least two of $B_1,\dots,B_t$), and is anticomplete to $C$.} contrary to the maximality of the $t$-frame $Q$. So, $v$ has a neighbor $c_1 \in C_1$. But now $v,c_1,c_1^2,b_2,v$ is a 4-hole in $G$, a contradiction.~$\blacklozenge$

\begin{adjustwidth}{1cm}{1cm} 
\begin{claim} \label{lemma-T0free-pyramid-claim-not-mixed-Ci}
No vertex in $V(G) \setminus V(Q)$ is mixed on any one of $C_1,\dots,C_t$. 
\end{claim} 
\end{adjustwidth} 
{\em Proof of Claim~\ref{lemma-T0free-pyramid-claim-not-mixed-Ci}.} Suppose otherwise. By symmetry, we may assume that some vertex $v \in V(G) \setminus V(Q)$ is mixed on $C_1$. Fix $c_1,c_1' \in C_1$ such that $v$ is adjacent to $c_1$ and nonadjacent to $c_1'$. 

We first show that $v$ is adjacent to $a_1$. Suppose otherwise. Then $v$ has a neighbor (namely $c_1$) in $C_1$ and is not complete to $A$, and so by Claim~\ref{lemma-T0free-pyramid-claim-BiCj}, $v$ is anticomplete to $B \setminus B_1 = B_2 \cup \dots \cup B_t$. But now $G[v,c_1,c_1',b_1^2,a_1,b_1^3]$ is a $2P_3$, a contradiction. This proves that $v$ is indeed adjacent to $a_1$. Note that this implies that $v$ is adjacent to $b_1^1$, for otherwise, $v,a_1,b_1^1,c_1,v$ would be a 4-hole in $G$, a contradiction. 

Next, we claim that $v$ is anticomplete to $C \setminus C_1 = C_2 \cup \dots \cup C_t$. Suppose otherwise. By symmetry, we may assume that $v$ is adjacent to some $c_2 \in C_2$. But then $v,c_2,c_1',b_1^1,v$ is a 4-hole in $G$, a contradiction. This proves that $v$ is indeed anticomplete to $C \setminus C_1$. 

We now show that $v$ is anticomplete to $B \setminus B_1 = B_2 \cup \dots \cup B_t$. Suppose otherwise. By symmetry, we may assume that $v$ is adjacent to some $b_2 \in B_2$. Since $v$ is anticomplete to $C \setminus C_1$, we know that $v$ is nonadjacent to $c_1^2$. But then $v,b_2,c_1^2,c_1,v$ is a 4-hole in $G$, a contradiction. This proves that $v$ is indeed anticomplete to $B \setminus B_1$. 

We now know that $v$ has a neighbor (namely $a$) in $A$ and a neighbor (namely $c_1$) in $C_1$, and we also know that $v$ is anticomplete to $(B \setminus B_1) \cup (C \setminus C_1)$. But then $G[V(Q) \cup \{v\}]$ is a $t$-frame, with an associated $t$-frame partition $(A;B_1 \cup \{v\},B_2,\dots,B_t;C_1,C_2,\dots,C_t)$, contrary to the maximality of the $t$-frame $Q$.~$\blacklozenge$

\begin{adjustwidth}{1cm}{1cm} 
\begin{claim} \label{lemma-T0free-pyramid-claim-BiBj-FiW}
For all $v \in V(G) \setminus V(Q)$, if $v$ has a neighbor in at least two of $B_1,\dots,B_t$, then $v \in F \cup W$. 
\end{claim} 
\end{adjustwidth} 
{\em Proof of Claim~\ref{lemma-T0free-pyramid-claim-BiBj-FiW}.} Fix $v \in V(G) \setminus V(Q)$, and assume that $v$ has a neighbor in at least two of $B_1,\dots,B_t$. By Claim~\ref{lemma-T0free-pyramid-claim-not-mixed-Bi}, $v$ is not mixed on any one of $B_1,\dots,B_t$. So, by symmetry, we may assume that there exists some $\ell \in \{2,\dots,t\}$ such that $v$ is complete to $B_1 \cup \dots \cup B_{\ell}$ and anticomplete to $B_{\ell+1} \cup \dots \cup B_t$. By Claim~\ref{lemma-T0free-pyramid-claim-BiBj}, $v$ is complete to $A$. 

Next, we claim that $v$ is anticomplete to $C_{\ell+1} \cup \dots \cup C_t$. Suppose otherwise. (In particular, $\ell \leq t-1$.) By symmetry, we may assume that $v$ is adjacent to some $c_{\ell+1} \in C_{\ell+1}$. But then $v,c_{\ell+1},b_1^{\ell+1},a_1,v$ is a 4-hole in $G$, a contradiction. This proves that $v$ is indeed anticomplete to $C_{\ell+1} \cup \dots \cup C_t$. Now, if $v$ is also anticomplete to $C_1 \cup \dots \cup C_{\ell}$, then $G[V(Q) \cup \{v\}]$ is a $t$-frame, with an associated $t$-frame partition $(A \cup \{v\};B_1,\dots,B_t;C_1,\dots,C_t)$, contrary to the maximality of the $t$-frame $Q$. So, $v$ has a neighbor in $C_1 \cup \dots \cup C_{\ell}$. We claim that $v$ is in fact complete to $C_1 \cup \dots \cup C_{\ell}$. Suppose otherwise. Then $v$ is mixed on $C_1 \cup \dots \cup C_{\ell}$, and so by Proposition~\ref{prop-mixed-on-Yi}, and by symmetry, we may assume that $v$ has a neighbor $c_1 \in C_1$ and a nonneighbor $c_2 \in C_2$. But then $v,c_1,c_2,b_1^2,v$ is a 4-hole in $G$, a contradiction. This proves that $v$ is complete to $C_1 \cup \dots \cup C_{\ell}$. 

So far, we have shown that $v$ is complete to $A \cup (B_1 \cup \dots \cup B_{\ell}) \cup (C_1 \cup \dots \cup C_{\ell})$ and anticomplete to $(B_{\ell+1} \cup \dots \cup C_t) \cup (C_{\ell+1} \cup \dots \cup C_t)$. If $\ell \geq t-1$, then $v \in F_t \cup W$, and we are done. We may therefore assume that $\ell \leq t-2$. So, $2 \leq \ell \leq t-2$, and in particular, $t \geq 4$. But now $G[b_1^1,v,b_1^2,b_1^{\ell+1},c_1^{\ell+1},c_1^{\ell+2}]$ is a $2P_3$, a contradiction.~$\blacklozenge$

\begin{adjustwidth}{1cm}{1cm} 
\begin{claim} \label{lemma-T0free-pyramid-claim-sets-partition-VG-VQ}
Sets $D_1,\dots,D_t,F_1,\dots,F_t,X_1,\dots,X_t,Y,Z,W$ form a partition of $V(G) \setminus V(Q)$.\footnote{As usual, some (or all) of the sets $D_1,\dots,D_t,F_1,\dots,F_t,X_1,\dots,X_t,Y,Z,W$ may possibly be empty.} Consequently, sets $D,F,X,Y,Z,W$ form a partition of $V(G) \setminus V(Q)$. 
\end{claim} 
\end{adjustwidth} 
{\em Proof of Claim~\ref{lemma-T0free-pyramid-claim-sets-partition-VG-VQ}.} It suffices to prove the first statement, for the second statement follows immediately from the first and from the definition of the sets $D,F,X$. By definition, sets $D_1,\dots,D_t,F_1,\dots,F_t,X_1,\dots,X_t,Y,Z,W$ are pairwise disjoint, and they are all subsets of $V(G) \setminus V(Q)$. So, it suffices to show that every vertex of $V(G) \setminus V(Q)$ belongs to one of $D_1,\dots,D_t,F_1,\dots,F_t,X_1,\dots,X_t,Y,Z,W$. 

Fix $v \in V(G) \setminus V(Q)$. If $v$ is anticomplete to $B$, then Claim~\ref{lemma-T0free-pyramid-claim-anticomp-B} guarantees that $v \in Y \cup Z$. On the other hand, if $v$ has a neighbor in at least two of $B_1,\dots,B_t$, then Claim~\ref{lemma-T0free-pyramid-claim-BiBj-FiW} guarantees that $v \in F_1 \cup \dots \cup F_t \cup W$. So from now on, we may assume that $v$ has a neighbor in exactly one of $B_1,\dots,B_t$. By symmetry, we may assume that $v$ has a neighbor in $B_1$ and is anticomplete to $B \setminus B_1 = B_2 \cup \dots \cup B_t$. Claim~\ref{lemma-T0free-pyramid-claim-not-mixed-Bi} then guarantees that $v$ is complete to $B_1$. 

First, we claim that $v$ is anticomplete to $C \setminus C_1 = C_2 \cup \dots \cup C_t$. Suppose otherwise. By symmetry, we may assume that $v$ has a neighbor $c_2 \in C_2$. Since $v$ has a neighbor both in $B_1$ and in $C \setminus C_1$, Claim~\ref{lemma-T0free-pyramid-claim-BiCj} guarantees that $v$ is complete to $A$. But then $v,a_1,b_1^2,c_2,v$ is a 4-hole in $G$, a contradiction. This proves that $v$ is indeed anticomplete to $C \setminus C_1$. 

So far, we have established that $B_1 \subseteq N_G(v) \cap V(Q) \subseteq A \cup B_1 \cup C_1$. We may now assume that $v$ has a neighbor in $C_1$, for otherwise, we have that $B_1 \subseteq N_G(v) \cap V(Q) \subseteq A \cup B_1$, and consequently, $v \in X_1$. By Claim~\ref{lemma-T0free-pyramid-claim-not-mixed-Ci}, $v$ is complete to $C_1$. If $v$ is anticomplete to $A$, then $N_G(v) \cap V(Q) = B_1 \cup C_1$, and consequently, $v \in D_1$. We may therefore assume that $v$ has a neighbor in $A$. But now $G[V(Q) \cup \{v\}]$ is a $t$-frame, with an associated $t$-frame partition $(A;B_1 \cup \{v\},B_2,\dots,B_t;C_1,C_2,\dots,C_t)$, contrary to the maximality of the $t$-frame $Q$.~$\blacklozenge$

\begin{adjustwidth}{1cm}{1cm}  
\begin{claim} \label{lemma-T0free-pyramid-claim-NGS-clique}
For all nonempty sets $S \subseteq V(G) \setminus V(Q)$, if $N_G(S)$ is a clique, then $G$ contains a simplicial vertex. 
\end{claim} 
\end{adjustwidth} 
{\em Proof of Claim~\ref{lemma-T0free-pyramid-claim-NGS-clique}.} Fix a nonempty set $S \subseteq V(G) \setminus V(Q)$, and assume that $N_G(S)$ is a clique. In view of Proposition~\ref{prop-2P3C4-free-clique-cut-simplicial}, it suffices to show that $N_G(S)$ is a clique-cutset of $G$; since $N_G(S)$ is a clique, it is in fact enough to show that $N_G(S)$ is a cutset of $G$. Clearly, $\big(S,N_G(S),V(G) \setminus N_G[S]\big)$ is a partition of $V(G)$, there are no edges between $S$ and $V(G) \setminus N_G[S]$, and we know that $S \neq \emptyset$. Thus, we just need to show that $V(G) \setminus N_G[S] \neq \emptyset$. But note that $V(Q) \setminus N_G(S) \subseteq (V(Q) \cup S) \setminus N_G[S] \subseteq V(G) \setminus N_G[S]$. So, we just need to show that $V(Q) \setminus N_G(S) \neq \emptyset$. But this follows immediately from the fact that $N_G(S)$ is a clique, whereas $V(Q)$ is not.\footnote{The fact that $V(Q)$ is not a clique follows, for example, from the fact that $t \geq 3$, and the fact that $B_1,\dots,B_t$ are nonempty, pairwise disjoint sets, pairwise anticomplete to each other.}~$\blacklozenge$

\begin{adjustwidth}{1cm}{1cm} 
\begin{claim} \label{lemma-T0free-pyramid-claim-Fi-clique} 
$F \cup W$ is a clique, and at most one of $F_1,\dots,F_t$ is nonempty. Moreover, for all $i \in \{1,\dots,t\}$, if $F_i \neq \emptyset$, then $B_i$ is complete to $C_i$. 
\end{claim}  
\end{adjustwidth} 
{\em Proof of Claim~\ref{lemma-T0free-pyramid-claim-Fi-clique}.} We first show that at most one of $F_1,\dots,F_t$ is nonempty. Suppose otherwise. By symmetry, we may assume that $F_1,F_2$ are both nonempty. Fix $f_1 \in F_1$ and $f_2 \in F_2$. If $f_1,f_2$ are adjacent, then $f_1,f_2,c_1^1,c_1^2,f_1$ is a 4-hole in $G$, a contradiction. So, $f_1,f_2$ are nonadjacent. But then $a_1,f_1,c_1^3,f_2,a_1$ is a 4-hole in $G$, again a contradiction.

Next, we show that $F \cup W$ is a clique. By what we just showed, and by symmetry, we may assume that $F_2 = \dots = F_t = \emptyset$, so that $F = F_1$ and $F \cup W = F_1 \cup W$. But note that $a_1,c_1^2$ are distinct, nonadjacent vertices that are both complete to $F_1 \cup W$, and so Proposition~\ref{prop-non-adj-comp-clique} guarantees that $F_1 \cup W$ is a clique. 

It remains to show that for all $i \in \{1,\dots,t\}$, if $F_i \neq \emptyset$, then $B_i$ is complete to $C_i$. Suppose otherwise. By symmetry, we may assume that $F_1 \neq \emptyset$, and that $B_1$ is not complete to $C_1$. It then follows from our orderings of $B_1$ and $C_1$ that $b_{r_1}^1$ and $c_{s_1}^1$ are nonadjacent, that $r_1 \geq 2$, and that $b_1^1$ is complete to $\{b_{r_1}^1,c_{s_1}^1\}$. Fix any $f_1 \in F_1$. But then $G[b_{r_1}^1,b_1^1,c_{s_1}^1,b_1^2,f_1,b_1^3]$ is a $2P_3$, a contradiction.~$\blacklozenge$

\begin{adjustwidth}{1cm}{1cm} 
\begin{claim} \label{lemma-T0free-pyramid-claim-YZ-anticomp} 
$Y$ is anticomplete to $Z$. 
\end{claim} 
\end{adjustwidth} 
{\em Proof of Claim~\ref{lemma-T0free-pyramid-claim-YZ-anticomp}.} Suppose otherwise, and fix adjacent vertices $y \in Y$ and $z \in Z$. 

Suppose first that $z$ is nonadjacent to $a_1$. By the definition of $Y$, we know that $y$ has a neighbor in $C$; by symmetry, we may assume that $y$ is adjacent to some $c_1 \in C_1$. But then $G[z,y,c_1,b_1^2,a_1,b_1^3]$ is a $2P_3$, a contradiction. 

We have now shown that $z$ is adjacent to $a_1$. Then $y$ is not complete to $C$, for otherwise, $G[a_1,z,b_1^1,\dots,b_1^t,y,c_1^1,\dots,c_1^t]$ would be a $(t+1)$-pentagon, a contradiction. Since $y$ has a neighbor in $C$ (by the definition of $Y$), we see that $y$ is mixed on $C = C_1 \cup \dots \cup C_t$. By Proposition~\ref{prop-mixed-on-Yi}, and by symmetry, we may now assume that $y$ is adjacent to some $c_1 \in C_1$ and nonadjacent to some $c_2 \in C_2$. But then $z,y,c_1,c_2,b_1^2,a_1,z$ is a 6-hole in $G$, a contradiction.~$\blacklozenge$ 

\begin{adjustwidth}{1cm}{1cm} 
\begin{claim} \label{lemma-T0free-pyramid-claim-DiXiZ-anticomp} 
$Z$ is anticomplete to $D \cup X$. 
\end{claim}  
\end{adjustwidth} 
{\em Proof of Claim~\ref{lemma-T0free-pyramid-claim-DiXiZ-anticomp}.} Suppose otherwise. By symmetry, we may assume that some $z \in Z$ and $x \in D_1 \cup X_1$ are adjacent. Now, we know that $Z$ is anticomplete to $B \cup C$, whereas $D_1 \cup X_1$ is complete to $B_1$ and anticomplete to $(B \setminus B_1) \cup (C \setminus C_1)$. So, $G[z,x,b_1^1,b_1^2,c_1^2,c_1^3]$ is a $2P_3$, a contradiction.~$\blacklozenge$

\begin{adjustwidth}{1cm}{1cm} 
\begin{claim} \label{lemma-T0free-pyramid-claim-NGZ} 
$N_G(Z) \subseteq A \cup F \cup W$, and consequently, $N_G(Z)$ is a clique. Therefore, if $Z \neq \emptyset$, then $G$ contains a simplicial vertex. 
\end{claim} 
\end{adjustwidth} 
{\em Proof of Claim~\ref{lemma-T0free-pyramid-claim-NGZ}.} By the definition of $Z$, we have that $N_G(Z) \cap V(Q) \subseteq A$, and by Claims~\ref{lemma-T0free-pyramid-claim-sets-partition-VG-VQ},~\ref{lemma-T0free-pyramid-claim-YZ-anticomp}, and~\ref{lemma-T0free-pyramid-claim-DiXiZ-anticomp}, we have that $N_G(Z) \setminus V(Q) \subseteq F \cup W$. Therefore, $N_G(Z) \subseteq A \cup F \cup W$. Now, by Lemma~\ref{lemma-t-frame}(\ref{ref-lemma-t-frame-ABiCi-cliques}), $A$ is a clique, and by Claim~\ref{lemma-T0free-pyramid-claim-Fi-clique}, $F \cup W$ is a clique. Moreover, by the definition of $F$ and $W$, we know that $A$ is complete to $F \cup W$. So, $A \cup F \cup W$ is a clique, and consequently, its subset $N_G(Z)$ is also a clique. Therefore, if $Z \neq \emptyset$, then Claim~\ref{lemma-T0free-pyramid-claim-NGS-clique} guarantees that $G$ contains a simplicial vertex.~$\blacklozenge$

\begin{adjustwidth}{1cm}{1cm}
\begin{claim} \label{lemma-T0free-pyramid-claim-DiXiY-anticomp} 
$Y$ is anticomplete to $D \cup X$. 
\end{claim} 
\end{adjustwidth} 
{\em Proof of Claim~\ref{lemma-T0free-pyramid-claim-DiXiY-anticomp}.} We first show that $Y$ is anticomplete to $D$. Suppose otherwise. By symmetry, we may assume that some $y \in Y$ and $d_1 \in D_1$ are adjacent. Suppose first that $y$ has a neighbor in $C_2 \cup \dots \cup C_t$; by symmetry, we may assume that $y$ is adjacent to some $c_2 \in C_2$. But then $d_1,y,c_2,b_1^2,a_1,b_1^1,d_1$ is a 6-hole in $G$, a contradiction. Thus, $y$ is anticomplete to $C_2 \cup \dots \cup C_t$. But now $G[y,d_1,b_1^1,b_1^2,c_1^2,c_1^3]$ is a $2P_3$, a contradiction. This proves that $Y$ is anticomplete to $D$. 

It remains to show that $Y$ is anticomplete to $X$. Suppose otherwise. By symmetry, we may assume that some $y \in Y$ is adjacent to some $x_1 \in X_1$. If $y$ is adjacent to some $c_1 \in C_1$, then $x_1,y,c_1,b_1^1,x_1$ is a 4-hole in $G$, a contradiction. Therefore, $y$ is anticomplete to $C_1$. Now, by the definition of $Y$, we know that $y$ has a neighbor in $C$, and so by symmetry, we may assume that $y$ has a neighbor in $C_2$; by Claim~\ref{lemma-T0free-pyramid-claim-not-mixed-Ci}, it follows that $y$ is complete to $C_2$. Then $x_1$ is adjacent to $a_1$, for otherwise, $x_1,b_1^1,a_1,b_1^2,c_1^2,y,x_1$ would be a 6-hole in $G$, a contradiction. Next, if $y$ had a nonneighbor $c_3 \in C_3$, then $y,c_1^2,c_3,b_1^3,a_1,x_1,y$ would be a 6-hole in $G$, a contradiction. Therefore, $y$ is complete to $C_3$. But now we see that adjacency between vertices $x_1,y,a_1,b_1^1,b_1^2,b_1^3,c_1^1,c_1^2,c_1^3$ is as in Figure~\ref{fig:thm-T0free-pyramid-claim-DiXiY-anticomp}, and we deduce that these vertices induce a $T_0$ in $G$, a contradiction. This proves that $Y$ is indeed anticomplete to $X$, and we are done.~$\blacklozenge$

\begin{figure}
\begin{center}
\includegraphics[scale=0.5]{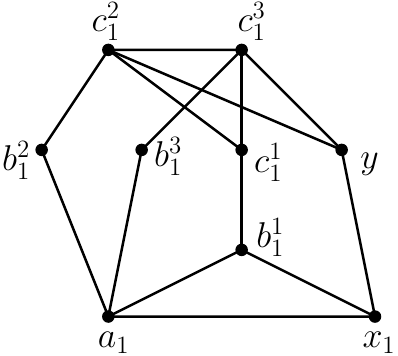}
\end{center} 
\caption{$G[x_1,y,a_1,b_1^1,b_1^2,b_1^3,c_1^1,c_1^2,c_1^3]$ is a $T_0$ (proof of Claim~\ref{lemma-T0free-pyramid-claim-DiXiY-anticomp}).} \label{fig:thm-T0free-pyramid-claim-DiXiY-anticomp} 
\end{figure} 

\begin{adjustwidth}{1cm}{1cm}
\begin{claim} \label{lemma-T0free-pyramid-claim-NGY} 
$N_G(Y) \subseteq C \cup F \cup W$. 
\end{claim} 
\end{adjustwidth} 
{\em Proof of Claim~\ref{lemma-T0free-pyramid-claim-NGY}.} By the definition of $Y$, we have that $N_G(Y) \cap V(Q) \subseteq C$. It remains to show that $N_G(Y) \setminus V(Q) \subseteq F \cup W$. But this follows immediately from Claims~\ref{lemma-T0free-pyramid-claim-sets-partition-VG-VQ},~\ref{lemma-T0free-pyramid-claim-YZ-anticomp}, and~\ref{lemma-T0free-pyramid-claim-DiXiY-anticomp}.~$\blacklozenge$

\begin{adjustwidth}{1cm}{1cm}
\begin{claim} \label{lemma-T0free-pyramid-claim-YC-AB} 
For all $i \in \{1,\dots,t\}$, if some vertex of $Y$ has a neighbor in $C_i$, then $A$ is complete to $B_i$. Consequently, if $Y$ is nonempty and complete to $C$, then $A$ is complete to $B$. 
\end{claim} 
\end{adjustwidth} 
{\em Proof of Claim~\ref{lemma-T0free-pyramid-claim-YC-AB}.} Clearly, the first statement implies the second. So, it suffices to prove the first statement, and by symmetry, it suffices to prove it for $i = 1$. Suppose otherwise, that is, suppose that some vertex $y \in Y$ has a neighbor in $C_1$, but $A$ is not complete to $B_1$. By Claim~\ref{lemma-T0free-pyramid-claim-not-mixed-Ci}, $y$ is complete to $C_1$, and by Lemma~\ref{lemma-t-frame}(\ref{ref-lemma-t-frame-AB}), $A$ is complete to $B \setminus B_1$. Now, fix nonadjacent vertices $a \in A$ and $b_1 \in B_1$. Then $G[y,c_1^1,b_1,b_1^2,a,b_1^3]$ is a $2P_3$, a contradiction.~$\blacklozenge$

\begin{adjustwidth}{1cm}{1cm}
\begin{claim} \label{lemma-T0free-pyramid-claim-Y-outcomes} 
If $Y \neq \emptyset$, then one of the following holds: 
\begin{itemize} 
\item $F \neq \emptyset$, and moreover, the set $Y$ is a clique, complete to $C \cup F \cup W$ and anticomplete to $V(G) \setminus (C \cup F \cup W \cup Y)$; 
\item $N_G(Y)$ is a clique, and $G$ contains a simplicial vertex. 
\end{itemize} 
\end{claim} 
\end{adjustwidth} 
{\em Proof of Claim~\ref{lemma-T0free-pyramid-claim-Y-outcomes}.} Assume that $Y \neq \emptyset$. If $N_G(Y)$ is a clique, then Claim~\ref{lemma-T0free-pyramid-claim-NGS-clique} guarantees that $G$ contains a simplicial vertex, and so the second outcome holds. Thus, it is enough to show that either the first outcome holds, or $N_G(Y)$ is a clique. 

Suppose first that $F = \emptyset$. Then Claim~\ref{lemma-T0free-pyramid-claim-NGY} guarantees that $N_G(Y) \subseteq C \cup W$. By Lemma~\ref{lemma-t-frame}(\ref{ref-lemma-t-frame-ABiCi-cliques}) and Claim~\ref{lemma-T0free-pyramid-claim-Fi-clique}, we know that $C$ and $W$ are cliques, and by the definition of $W$, we know that $W$ is complete to $C$. So, $C \cup W$ is a clique, and consequently, its subset $N_G(Y)$ is also a clique. 

From now on, we assume that $F \neq \emptyset$. By Claim~\ref{lemma-T0free-pyramid-claim-Fi-clique} and by symmetry, we may assume that $F_2 = \dots = F_t = \emptyset$ and $F = F_1 \neq \emptyset$. 

Suppose first that $Y$ is anticomplete to $C_1$. Then Claim~\ref{lemma-T0free-pyramid-claim-NGY} guarantees that $N_G(Y) \subseteq (C \setminus C_1) \cup F_1 \cup W$. By the definition of $F_1$ and $W$, we know that $F_1 \cup W$ is complete to $C \setminus C_1$. On the other hand, by Lemma~\ref{lemma-t-frame}(\ref{ref-lemma-t-frame-ABiCi-cliques}) and Claim~\ref{lemma-T0free-pyramid-claim-Fi-clique}, both $C \setminus C_1$ and $F_1 \cup W$ are cliques. Thus, $(C \setminus C_1) \cup F_1 \cup W$ is a clique, and consequently, its subset $N_G(Y)$ is also a clique. 

From now on, we assume that $Y$ is not anticomplete to $C_1$. Our goal is to show that the first outcome holds. By Claim~\ref{lemma-T0free-pyramid-claim-NGY}, we have that $N_G(Y) \subseteq C \cup F \cup W$, and in particular, $Y$ is anticomplete to $V(G) \setminus (C \cup F \cup W \cup Y)$. Thus, we just need to show that $Y$ is a clique, complete to $C \cup F \cup W = C \cup F_1 \cup W$. 

Let $Y'$ be the set of all vertices in $Y$ that have a neighbor in $C_1$; since $Y$ is not anticomplete to $C_1$, we know that $Y' \neq \emptyset$. By Claim~\ref{lemma-T0free-pyramid-claim-not-mixed-Ci}, $Y'$ is in fact complete to $C_1$. Then $Y'$ is complete to $F_1$, for otherwise, we fix nonadjacent vertices $y \in Y'$ and $f_1 \in F_1$, and we observe that $G[y,c_1^1,b_1^1,b_1^2,f_1,b_1^3]$ is a $2P_3$, a contradiction. Further, by Claim~\ref{lemma-T0free-pyramid-claim-Fi-clique}, $F_1$ is a clique, complete to $W$. But now for any $f_1 \in F_1$, we have that $f_1,c_1^1$ are distinct, nonadjacent vertices that are both complete to $(C \setminus C_1) \cup W \cup Y'$, and so by Proposition~\ref{prop-non-adj-comp-clique}, $(C \setminus C_1) \cup W \cup Y'$ is a clique. In particular, $Y'$ is a clique, complete to $(C \setminus C_1) \cup W$. We have already seen that $Y'$ is complete to $C_1$ and to $F_1$, and we deduce that $Y'$ is a clique, complete to $C \cup F_1 \cup W$. 


It remains to show that $Y' = Y$. Suppose otherwise, and fix some $y \in Y \setminus Y'$. By the definition of $Y$, $y$ has a neighbor in $C$; since $y$ is anticomplete to $C_1$ (because $y \notin Y'$), it follows that $y$ in fact has a neighbor in $C \setminus C_1 = C_2 \cup \dots \cup C_t$. By symmetry, we may assume that $y$ is adjacent to some $c_2 \in C_2$. Fix any $y' \in Y'$. But now if $y,y'$ are adjacent, then $G[y,y',c_1^1,b_1^2,a_1,b_1^3]$ is a $2P_3$, a contradiction; and if $y,y'$ are nonadjacent, then $G[y,c_2,y',b_1^1,a_1,b_1^3]$ is a $2P_3$, again a contradiction. This proves that $Y' = Y$, and we are done.~$\blacklozenge$

\begin{adjustwidth}{1cm}{1cm}
\begin{claim} \label{lemma-T0free-pyramid-claim-DiXi-at-most-one-nonempty} 
At most one of $D$ and $X$ is nonempty. 
\end{claim} 
\end{adjustwidth} 
{\em Proof of Claim~\ref{lemma-T0free-pyramid-claim-DiXi-at-most-one-nonempty}.} We assume that $D \neq \emptyset$, and we show that $X = \emptyset$. By symmetry, we may assume that $D_1 \neq \emptyset$, and we fix some $d_1 \in D_1$. 

We first show that $X_1 = \emptyset$. Suppose otherwise, and fix some $x_1 \in X_1$. Then $d_1,x_1$ are adjacent for otherwise, $G[d_1,b_1^1,x_1,b_1^2,c_1^2,c_1^3]$ would be a $2P_3$, a contradiction. Now $x_1$ is nonadjacent to $a_1$, for otherwise, $d_1,x_1,a_1,b_1^2,c_1^2,c_1^1,d_1$ would be a 6-hole in $G$, a contradiction. But now $G[x_1,d_1,c_1^1,b_1^2,a_1,b_1^3]$ is a $2P_3$, a contradiction. This proves that $X_1 = \emptyset$. 

It remains to show that $X_2 \cup \dots \cup X_t = \emptyset$. Suppose otherwise. By symmetry, we may assume that $X_2 \neq \emptyset$, and we fix some $x_2 \in X_2$. Suppose first that $d_1,x_2$ are adjacent. Then $x_2$ is nonadjacent to $a_1$, for otherwise, $x_2,a_1,b_1^1,d_1,x_2$ would be a 4-hole in $G$, a contradiction. But then $d_1,x_2,b_1^2,a_1,b_1^3,c_1^3,c_1^1,d_1$ is a 7-hole in $G$, again a contradiction. This proves that $d_1,x_2$ are nonadjacent. But now if $x_2$ is adjacent to $a_1$, then $G[x_2,a_1,b_1^3,d_1,c_1^1,c_1^2]$ is a $2P_3$, a contradiction; and if $x_2$ is nonadjacent to $a_1$, then $G[x_2,b_1^2,a_1,d_1,c_1^1,c_1^3]$ is a $2P_3$, again a contradiction. This proves that $X_2 \cup \dots \cup X_t = \emptyset$, and we are done.~$\blacklozenge$

\begin{adjustwidth}{1cm}{1cm} 
\begin{claim} \label{lemma-T0free-pyramid-claim-DiXi-clique} 
Sets $D_1,\dots,D_t$ are cliques, and at most one of them is nonempty. Sets $X_1,\dots,X_t$ are cliques, pairwise anticomplete to each other. 
\end{claim}  
\end{adjustwidth} 
{\em Proof of Claim~\ref{lemma-T0free-pyramid-claim-DiXi-clique}.} We first show that $D_1 \cup X_1,\dots,D_t \cup X_t$ are all cliques. Suppose otherwise. By symmetry, we may assume that $D_1 \cup X_1$ is not a clique. Fix distinct, nonadjacent vertices $x_1,x_1' \in D_1 \cup X_1$. But $D_1 \cup X_1$ is complete to $B_1$ and anticomplete to $(B \setminus B_1) \cup (C \setminus C_1)$, and we deduce that $G[x_1,b_1^1,x_1',b_1^2,c_1^2,c_1^3]$ is a $2P_3$, a contradiction. This proves that $D_1 \cup X_1,\dots,D_t \cup X_t$ are all cliques. Consequently, $D_1,\dots,D_t,X_1,\dots,X_t$ are all cliques. 

Next, we show that at most one of $D_1,\dots,D_t$ is nonempty. Suppose otherwise. By symmetry, we may assume that $D_1,D_2$ are both nonempty, and we fix $d_1 \in D_1$ and $d_2 \in D_2$. But if $d_1,d_2$ are adjacent then $d_1,d_2,c_1^2,c_1^1,d_1$ is a 4-hole in $G$, whereas if $d_1,d_2$ are nonadjacent, then $G[d_1,b_1^1,a_1,d_2,c_1^2,c_1^3]$ is a $2P_3$, a contradiction in either case. This proves that at most one of $D_1,\dots,D_t$ is nonempty.

It remains to show that $X_1,\dots,X_t$ are pairwise anticomplete to each other. Suppose otherwise. By symmetry, we may assume that some $x_1 \in X_1$ and $x_2 \in X_2$ are adjacent. Then $x_1,b_1^1,c_1^1,c_1^2,b_1^2,x_2,x_1$ is a 6-hole in $G$, a contradiction. Thus, $X_1,\dots,X_t$ are indeed pairwise anticomplete to each other, and we are done.~$\blacklozenge$

\begin{adjustwidth}{1cm}{1cm}
\begin{claim} \label{lemma-T0free-pyramid-claim-Di-t=3} 
If $D \neq \emptyset$, then $t = 3$, and for all $i \in \{1,2,3\}$, $B_i$ is complete to $C_i$. 
\end{claim} 
\end{adjustwidth} 
{\em Proof of Claim~\ref{lemma-T0free-pyramid-claim-Di-t=3}.} Assume that $D \neq \emptyset$. By symmetry, we may assume that $D_1 \neq \emptyset$, and we fix an arbitrary $d_1 \in D_1$. 

First, if $t \geq 4$, then $G[d_1,c_1^1,c_1^2,b_1^3,a_1,b_1^4]$ is a $2P_3$, a contradiction. Therefore, $t = 3$. 

Next, if some $b_1 \in B_1$ and $c_1 \in C_1$ are nonadjacent, then $d_1,b_1,a_1,b_1^2,c_1^2,c_1,d_1$ is a 6-hole in $G$, a contradiction. Therefore, $B_1$ is complete to $C_1$. 

Finally, if some $b_2 \in B_2$ and $c_2 \in C_2$ are nonadjacent, then $G[d_1,c_1^1,c_2,b_2,a_1,b_1^3]$ is a $2P_3$, a contradiction. Therefore, $B_2$ is complete to $C_2$, and similarly, $B_3$ is complete to $C_3$.~$\blacklozenge$

\begin{adjustwidth}{1cm}{1cm}
\begin{claim} \label{lemma-T0free-pyramid-claim-NGDi} 
For all $i \in \{1,\dots,t\}$, $N_G(D_i) \subseteq B_i \cup C_i \cup (F \setminus F_i) \cup W$. 
\end{claim} 
\end{adjustwidth} 
{\em Proof of Claim~\ref{lemma-T0free-pyramid-claim-NGDi}.} By symmetry, it suffices to prove the claim for $i = 1$. We may assume that $D_1 \neq \emptyset$, for otherwise, $N_G(D_i) = \emptyset$, and the result is immediate. 

By the definition of $D_1$, we know that $N_G(D_1) \cap V(Q) = B_1 \cup C_1$. So, we just need to show that $N_G(D_1) \setminus V(Q) \subseteq (F \setminus F_1) \cup W$. Now, since $D_1 \neq \emptyset$, Claims~\ref{lemma-T0free-pyramid-claim-sets-partition-VG-VQ},~\ref{lemma-T0free-pyramid-claim-DiXi-at-most-one-nonempty}, and~\ref{lemma-T0free-pyramid-claim-DiXi-clique} together guarantee that $V(G) \setminus V(Q) = D_1 \cup F \cup Y \cup Z \cup W$. Further, by Claims~\ref{lemma-T0free-pyramid-claim-DiXiZ-anticomp} and~\ref{lemma-T0free-pyramid-claim-DiXiY-anticomp}, we have that $D$ is anticomplete to $Y \cup Z$, and we deduce that $N_G(D_1) \setminus V(Q) \subseteq F \cup W$. 

It now remains to show that $D_1$ is anticomplete to $F_1$. Suppose otherwise. By symmetry, we may assume that some $d_1 \in D_1$ and $f_1 \in F_1$ are adjacent. But then $d_1,f_1,c_1^2,c_1^1,d_1$ is a 4-hole in $G$, a contradiction.~$\blacklozenge$

\begin{adjustwidth}{1cm}{1cm}
\begin{claim} \label{lemma-T0free-pyramid-claim-Di-simplicial} 
If $D \neq \emptyset$, then $G$ contains a simplicial vertex. 
\end{claim} 
\end{adjustwidth} 
{\em Proof of Claim~\ref{lemma-T0free-pyramid-claim-Di-simplicial}.} Assume that $D \neq \emptyset$. By symmetry, we may assume that $D_1 \neq \emptyset$. By Claim~\ref{lemma-T0free-pyramid-claim-NGS-clique}, it suffices to show that $N_G(D_1)$ is a clique. By Claim~\ref{lemma-T0free-pyramid-claim-NGDi}, we have that $N_G(D_1) \subseteq B_1 \cup C_1 \cup (F \setminus F_1) \cup W$, and so we just need to show that $B_1 \cup C_1 \cup (F \setminus F_1) \cup W$ is a clique. By the definition of sets $F$, $F_1$, and $W$, we know that $(F \setminus F_1) \cup W$ is complete to $B_1 \cup C_1$. By Claim~\ref{lemma-T0free-pyramid-claim-Fi-clique}, $(F \setminus F_1) \cup W$ is a clique, and by Lemma~\ref{lemma-t-frame}(\ref{ref-lemma-t-frame-ABiCi-cliques}), $B_1$ and $C_1$ are cliques. Finally, since $D_1 \neq \emptyset$, Claim~\ref{lemma-T0free-pyramid-claim-Di-t=3} guarantees that $B_1$ is complete to $C_1$. So, $B_1 \cup C_1 \cup (F \setminus F_1) \cup W$ is indeed a clique, and we are done.~$\blacklozenge$ 

%

\begin{adjustwidth}{1cm}{1cm}
\begin{claim} \label{lemma-T0free-pyramid-claim-Xi-nonempty-AB-complete} 
If $X \neq \emptyset$, then $A$ is complete to $B$. 
\end{claim} 
\end{adjustwidth} 
{\em Proof of Claim~\ref{lemma-T0free-pyramid-claim-Xi-nonempty-AB-complete}.} Assume that $X \neq \emptyset$. By symmetry, we may assume that $X_1 \neq \emptyset$, and we fix some $x_1 \in X_1$. 

We first show that $A$ is complete to $B_1$. Suppose otherwise, and fix nonadjacent vertices $a \in A$ and $b_1 \in B_1$. Then $x_1$ is nonadjacent to $a$, for otherwise, $x_1,a,b_1^2,c_1^2,c_1^1,b_1,x_1$ would be a 6-hole in $G$, a contradiction. Now, since $A$ is not complete to $B_1$, Lemma~\ref{lemma-t-frame}(\ref{ref-lemma-t-frame-AB}) guarantees that $t = 3$ and that $A$ is complete to $B \setminus B_1 = B_2 \cup B_3$. But then $G[x_1,b_1^1,c_1^1,b_1^2,a,b_1^3]$ is a $2P_3$, a contradiction. This proves that $A$ is indeed complete to $B_1$. 

It remains to show that $A$ is complete to $B \setminus B_1 = B_2 \cup \dots \cup B_t$. Suppose otherwise. By symmetry, we may assume that some $a \in A$ and $b_2 \in B_2$ are nonadjacent. In particular, $A$ is not complete to $B_2$, and so by Lemma~\ref{lemma-t-frame}(\ref{ref-lemma-t-frame-AB}), we have that $t = 3$ and that $A$ is complete to $B \setminus B_2 = B_1 \cup B_3$. But now if $x_1$ is adjacent to $a$, then $G[x_1,a,b_1^3,b_2,c_1^2,c_1^1]$ is a $2P_3$, a contradiction; and if $x_1$ is nonadjacent to $a$, then $G[x_1,b_1^1,a,b_2,c_1^2,c_1^3]$ is a $2P_3$, again a contradiction. This proves that $A$ is indeed complete to $B \setminus B_1$, and we are done.~$\blacklozenge$

\begin{adjustwidth}{1cm}{1cm}
\begin{claim} \label{lemma-T0free-pyramid-claim-FiXi} 
For all $i \in \{1,\dots,t\}$, $X_i$ is complete to $B_i \cup F_i$ and anticomplete to $V(G) \setminus (A \cup B_i \cup F \cup W \cup X_i)$, and consequently, $N_G(X_i) \subseteq A \cup B_i \cup F \cup W$.  
\end{claim} 
\end{adjustwidth} 
{\em Proof of Claim~\ref{lemma-T0free-pyramid-claim-FiXi}.} By symmetry, it suffices to prove the claim for $i = 1$. We may assume that $X_1 \neq \emptyset$, for otherwise, the result is immediate. 

By definition, $X_1$ is complete to $B_1$. Moreover, $X_1$ is complete to $F_1$, for if some $x_1 \in X_1$ and $f_1 \in F_1$ were nonadjacent, then $G[x_1,b_1^1,c_1^1,b_1^2,f_1,b_1^3]$ would be a $2P_3$, a contradiction. This proves that $X_1$ is complete to $B_1 \cup F_1$. 

It now remains to show that $X_1$ is anticomplete to $V(G) \setminus (A \cup B_1 \cup F \cup W \cup X_1)$. But note that $V(G) \setminus (A \cup B_1 \cup F \cup W \cup X_1) = \big(V(Q) \setminus (A \cup B_1)\big) \cup \big(V(G) \setminus (V(Q) \cup F \cup W \cup X_1)\big)$. By definition, $X_1$ is anticomplete to $V(Q) \setminus (A \cup B_1)$. So, in fact, we just need to prove that $X_1$ is anticomplete to $V(G) \setminus (V(Q) \cup F \cup W \cup X_1)$. Since $X_1 \neq \emptyset$, Claim~\ref{lemma-T0free-pyramid-claim-DiXi-at-most-one-nonempty} guarantees that $D = \emptyset$. So, by Claim~\ref{lemma-T0free-pyramid-claim-sets-partition-VG-VQ}, we have that $V(G) \setminus (V(Q) \cup F \cup W \cup X_1) = (X_2 \cup \dots \cup X_t) \cup Y \cup Z$. But we are now done, since Claims~\ref{lemma-T0free-pyramid-claim-DiXiZ-anticomp},~\ref{lemma-T0free-pyramid-claim-DiXiY-anticomp}, and~\ref{lemma-T0free-pyramid-claim-DiXi-clique} together guarantee that $X_1$ is anticomplete to $(X_2 \cup \dots \cup X_t) \cup Y \cup Z$.~$\blacklozenge$

\begin{adjustwidth}{1cm}{1cm}
\begin{claim} \label{lemma-T0free-pyramid-claim-Fi-empty-NGXi-clique} 
For all $i \in \{1,\dots,t\}$, if $F_i = \emptyset$, then $N_G(X_i)$ is a clique. Consequently, if there exists some $i \in \{1,\dots,t\}$ such that $F_i = \emptyset$ and $X_i \neq \emptyset$, then $G$ contains a simplicial vertex. 
\end{claim} 
\end{adjustwidth} 
{\em Proof of Claim~\ref{lemma-T0free-pyramid-claim-Fi-empty-NGXi-clique}.} In view of Claim~\ref{lemma-T0free-pyramid-claim-NGS-clique}, we see that the second statement follows immediately from the first. Thus, it suffices to prove the first statement. So, we fix $i \in \{1,\dots,t\}$, we assume that $F_i = \emptyset$, and we show that $N_G(X_i)$ is a clique. We may assume that $X_i \neq \emptyset$, for otherwise, $N_G(X_i) = \emptyset$, and we are done. By Claim~\ref{lemma-T0free-pyramid-claim-FiXi}, we have that $N_G(X_i) \subseteq A \cup B_i \cup F \cup W$. Thus, it is enough to show that $A \cup B_i \cup F \cup W$ is  clique. By Lemma~\ref{lemma-t-frame}(\ref{ref-lemma-t-frame-ABiCi-cliques}), $A$ and $B_i$ are cliques, and since $X_i \neq \emptyset$, Claim~\ref{lemma-T0free-pyramid-claim-Xi-nonempty-AB-complete} guarantees that $A$ is complete to $B$ (and consequently, to $B_i$); so $A \cup B_i$ is a clique. Moreover, since $F_i = \emptyset$, the definition of $F$ and $W$ guarantees that $F \cup W$ is complete to $A \cup B_i$. Finally, Claim~\ref{lemma-T0free-pyramid-claim-Fi-clique} guarantees that $F \cup W$ is a clique. Thus, $A \cup B_i \cup F \cup W$ is a clique, and we are done.~$\blacklozenge$

\begin{adjustwidth}{1cm}{1cm}
\begin{claim} \label{lemma-T0free-pyramid-claim-Fi-nonempty-Xi-complete-ABiFi} 
For all $i \in \{1,\dots,t\}$, if $F_i \neq \emptyset$, then $X_i$ is complete to $A \cup B_i \cup F_i \cup W$ and anticomplete to $V(G) \setminus (A \cup B_i \cup F_i \cup W \cup X_i)$. 
\end{claim} 
\end{adjustwidth} 
{\em Proof of Claim~\ref{lemma-T0free-pyramid-claim-Fi-nonempty-Xi-complete-ABiFi}.} By symmetry, it suffices to prove the claim for $i = 1$. So, assume that $F_1 \neq \emptyset$. By Claim~\ref{lemma-T0free-pyramid-claim-Fi-clique}, it follows that $F_2 = \dots = F_t = \emptyset$, and consequently, $F = F_1$. Now, we must show that $X_1$ is complete to $A \cup B_1 \cup F_1 \cup W$ and anticomplete to $V(G) \setminus (A \cup B_1 \cup F_1 \cup W \cup X_1)$. Since $F = F_1$, Claim~\ref{lemma-T0free-pyramid-claim-FiXi} guarantees that $X_1$ is complete to $B_1 \cup F_1$ and anticomplete to $V(G) \setminus (A \cup B_1 \cup F_1 \cup W \cup X_1)$. So, it only remains to show that $X_1$ is complete to $A \cup W$. Now, since $X_1 \neq \emptyset$, Claim~\ref{lemma-T0free-pyramid-claim-Xi-nonempty-AB-complete} guarantees that $A$ is complete to $B$. On the other hand, by Claim~\ref{lemma-T0free-pyramid-claim-Fi-clique}, $F$ is complete to $W$. But now we fix any $f_1 \in F_1$, and we observe that $b_1^1,f_1$ are distinct, nonadjacent vertices that are both complete to $X_1 \cup A \cup W$; so, by Proposition~\ref{prop-non-adj-comp-clique}, $X_1 \cup A \cup W$ is a clique. In particular, $X_1$ is complete to $A \cup W$, and we are done.~$\blacklozenge$

\begin{adjustwidth}{1cm}{1cm}
\begin{claim} \label{lemma-T0free-pyramid-claim-W-universal} 
If $W \neq \emptyset$, then $G$ contains a simplicial vertex or a universal vertex. 
\end{claim} 
\end{adjustwidth} 
{\em Proof of Claim~\ref{lemma-T0free-pyramid-claim-W-universal}.} Assume that $W \neq \emptyset$. First, we may assume that $D \cup Z = \emptyset$, for otherwise, Claims~\ref{lemma-T0free-pyramid-claim-NGZ} and~\ref{lemma-T0free-pyramid-claim-Di-simplicial} guarantee that $G$ contains a simplicial vertex, and we are done. So, by Claim~\ref{lemma-T0free-pyramid-claim-sets-partition-VG-VQ}, we have that $V(G) = V(Q) \cup F \cup X \cup Y \cup W$. By the definition of $W$, we know that $W$ is complete to $V(Q)$, and by Claim~\ref{lemma-T0free-pyramid-claim-Fi-clique}, $W$ is a clique, complete to $F$. Moreover, we may assume that $W$ is complete to $Y$, for otherwise, Claim~\ref{lemma-T0free-pyramid-claim-Y-outcomes} guarantees that $G$ contains a simplicial vertex, and we are done.\footnote{Indeed, suppose that $W$ is not complete to $Y$. In particular, $Y \neq \emptyset$, and so Claim~\ref{lemma-T0free-pyramid-claim-Y-outcomes} applies. Since $Y$ is not complete to $W$, the first outcome of Claim~\ref{lemma-T0free-pyramid-claim-Y-outcomes} does not hold. Therefore, the second outcome of Claim~\ref{lemma-T0free-pyramid-claim-Y-outcomes} holds, and it follows that $G$ contains a simplicial vertex.} If $W$ is complete to $X$, then every vertex of $W$ is universal in $G$, and consequently (since $W \neq \emptyset$), $G$ contains a universal vertex. So, we may assume that $W$ is not complete to $X$. By symmetry, we may further assume that $W$ is not complete to $X_1$, and in particular, $X_1 \neq \emptyset$. Since $X_1$ is not complete to $W$, Claim~\ref{lemma-T0free-pyramid-claim-Fi-nonempty-Xi-complete-ABiFi} guarantees that $F_1 = \emptyset$. But now Claim~\ref{lemma-T0free-pyramid-claim-Fi-empty-NGXi-clique} guarantees that $G$ contains a simplicial vertex.~$\blacklozenge$ 

\begin{adjustwidth}{1cm}{1cm} 
\begin{claim} \label{lemma-T0free-pyramid-claim-outcomes-prelim}
At least one of the following holds: 
\begin{itemize} 
\item $G = Q$; 
\item there exists some $i \in \{1,\dots,t\}$ such that $F_i \neq \emptyset$ and $V(G) = V(Q) \cup F_i \cup X_i \cup Y$; 
\item $G$ contains a simplicial vertex or a universal vertex. 
\end{itemize} 
\end{claim} 
\end{adjustwidth} 
{\em Proof of Claim~\ref{lemma-T0free-pyramid-claim-outcomes-prelim}.} We may assume that $G$ contains no simplicial vertices and no universal vertices, for otherwise, we are done. So, by Claims~\ref{lemma-T0free-pyramid-claim-NGZ},~\ref{lemma-T0free-pyramid-claim-Di-simplicial}, and~\ref{lemma-T0free-pyramid-claim-W-universal}, $D,Z,W$ are all empty. Therefore, by Claim~\ref{lemma-T0free-pyramid-claim-sets-partition-VG-VQ}, we have that $V(G) = V(Q) \cup F \cup X \cup Y$. By Claim~\ref{lemma-T0free-pyramid-claim-Fi-clique}, and by symmetry, we may assume that $F_2 = \dots = F_t = \emptyset$ and $F = F_1$. So, by Claim~\ref{lemma-T0free-pyramid-claim-Fi-empty-NGXi-clique}, we have that $X_2 = \dots = X_t = \emptyset$. Therefore, $V(G) = V(Q) \cup F_1 \cup X_1 \cup Y$. If $F_1 \neq \emptyset$, then we are done. We may therefore assume that $F_1 = \emptyset$. But now Claim~\ref{lemma-T0free-pyramid-claim-Fi-empty-NGXi-clique} implies that $X_1 = \emptyset$, and Claim~\ref{lemma-T0free-pyramid-claim-Y-outcomes} implies that $Y = \emptyset$. So, $V(G) = V(Q)$. Since $Q$ is an induced subgraph of $G$, it follows that $G = Q$, and we are done.~$\blacklozenge$

\medskip 

We are now ready to complete the proof of the lemma. First of all, by Lemma~\ref{lemma-t-frame}(\ref{ref-lemma-t-frame-t-villa}-\ref{ref-lemma-t-frame-5-basket}), we have the following: 
\begin{enumerate}[(1)] 
\item if $A$ is complete to $B$, then $Q$ is a $t$-villa with an associated $t$-villa partition $(A;B_1,\dots,B_t;C_1,\dots,C_t)$; 
\item if $A$ is not complete to $B$, then $t = 3$, and $Q$ is a 5-basket with an associated 5-basket partition $(A;B_1,B_2,B_3;C_1,C_2,C_3;\emptyset)$. 
\end{enumerate} 
In particular, $Q$ is a villa or a 5-basket, and so if $G = Q$, then we are done. We are also done if $G$ contains a simplicial vertex or a universal vertex. Thus, by Claim~\ref{lemma-T0free-pyramid-claim-outcomes-prelim}, and by symmetry, we may assume that $F_1 \neq \emptyset$, and that $V(G) = V(Q) \cup F_1 \cup X_1 \cup Y$. By definition, $F_1$ is complete to $V(Q) \setminus (B_1 \cup C_1)$ and anticomplete to $B_1 \cup C_1$, and by Claim~\ref{lemma-T0free-pyramid-claim-Fi-clique}, we have that $F_1$ is a clique, and that $B_1$ is complete to $C_1$. 

Suppose first that $A$ is not complete to $B$. Then Claims~\ref{lemma-T0free-pyramid-claim-YC-AB} and~\ref{lemma-T0free-pyramid-claim-Y-outcomes} together guarantee that $Y = \emptyset$,\footnote{Indeed, suppose that $Y \neq \emptyset$. Then, since $G$ contains no simplicial vertices, Claim~\ref{lemma-T0free-pyramid-claim-Y-outcomes} guarantees that $Y$ is complete to $C$. But then by Claim~\ref{lemma-T0free-pyramid-claim-YC-AB}, $A$ is complete to $B$, a contradiction.} whereas Claim~\ref{lemma-T0free-pyramid-claim-Xi-nonempty-AB-complete} guarantees that $X_1 = \emptyset$. So, $V(G) = V(Q) \cup F_1$. In view of~(2), we now deduce that $t = 3$, and that $G$ is a 5-basket, with an associated 5-basket partition $(A;B_1,B_2,B_3;C_1,C_2,C_3;F_1)$. 

Suppose now that $A$ is complete to $B$. Our goal is to show that $G$ is a $t$-mansion, with an associated $t$-mansion partition $(A;B_1,\dots,B_t;C_1,\dots,C_t;F_1;X_1;Y)$, and with $j^* = 1$.\footnote{Here, $j^*$ is the index from the definition of a $t$-mansion.} First of all, by~(1), $Q$ is a $t$-villa, with an associated $t$-villa partition $(A;B_1,\dots,B_t;C_1,\dots,C_t)$. Next, recall that $F = F_1 \neq \emptyset$, that $W = \emptyset$,\footnote{The fact that $W = \emptyset$ follows from the fact that $V(G) = V(Q) \cup F_1 \cup X_1 \cup Y$.} and that $B_1$ is complete to $C_1$. By Claims~\ref{lemma-T0free-pyramid-claim-DiXi-clique} and~\ref{lemma-T0free-pyramid-claim-Fi-nonempty-Xi-complete-ABiFi}, $X_1$ is a clique, complete to $A \cup B_1 \cup F_1$ and anticomplete to $V(G) \setminus (A \cup B_1 \cup F_1 \cup X_1)$; in particular, $X_1$ is anticomplete to $Y$. Further, since $G$ contains no simplicial vertices, Claim~\ref{lemma-T0free-pyramid-claim-Y-outcomes} implies that $Y$ is a clique, complete to $C \cup F_1$ and anticomplete to $V(G) \setminus (C \cup F_1 \cup Y)$.\footnote{If $Y = \emptyset$, then $Y$ is obviously a clique, complete to $C \cup F_1$ and anticomplete to $V(G) \setminus (C \cup F_1 \cup Y)$. On the other hand, if $Y \neq \emptyset$, then Claim~\ref{lemma-T0free-pyramid-claim-Y-outcomes} applies.} The result is now immediate. 
\end{proof}

\begin{theorem} \label{thm-T0free-pyramid-iff} For any graph $G$, the following are equivalent: 
\begin{enumerate}[(a)]
\item $G$ is a $(2P_3,C_4,C_6,C_7,T_0)$-free graph that contains an induced 3-pentagon and contains no simplicial vertices; 
\item $G$ contains exactly one nontrivial anticomponent, and this anticomponent is a 5-basket, a villa, or a mansion; 
\item $G$ can be obtained from a 5-basket, a villa, or a mansion by possibly adding universal vertices to it. 
\end{enumerate} 
\end{theorem} 
\begin{proof} 
Fix a graph $G$. By Propositions~\ref{prop-5-basket-in-class},~\ref{prop-t-villa-in-class},~\ref{prop-t-mansion-in-class}, we know that 5-baskets, villas, and mansions are anticonnected, and obviously, they all have at least two vertices. So, by Proposition~\ref{prop-non-trivial-anticomp-univ-vertices}, (b) and (c) are equivalent. It remains to show that (a) and (b) are equivalent. 

Suppose first that (a) holds, so that $G$ is a $(2P_3,C_4,C_6,C_7,T_0)$-free graph that contains an induced 3-pentagon and contains no simplicial vertices. Since $G$ contains no simplicial vertices, we know that $G$ is not a complete graph. Therefore, $G$ contains at least one nontrivial anticomponent; since $G$ is $C_4$-free, Proposition~\ref{prop-C4-free-one-anticomp} guarantees that $G$ in fact contains exactly one nontrivial anticomponent, call it $Q$. Since $G$ is $(2P_3,C_4,C_6,C_7,T_0)$-free, so is $Q$. Since $G$ contains an induced 3-pentagon, and since the 3-pentagon contains no universal vertices, Proposition~\ref{prop-H-free-no-universal} guarantees that $Q$ contains an induced 3-pentagon. Moreover, since $G$ contains no simplicial vertices, Proposition~\ref{prop-one-nontrivial-anticomp-simplicial} guarantees that $Q$ contains no simplcial vertices. Finally, since $Q$ is an anticonnected graph on at least two vertices, we know that $Q$ contains no universal vertices. Therefore, Lemma~\ref{lemma-T0free-pyramid} guarantees that $Q$ is a 5-basket, villa, or mansion. So, (b) holds. 

Suppose conversely that (b) holds. Then $G$ contains exactly one nontrivial anticomponent, call it $Q$, and this anticomponent is a 5-basket, villa, or mansion. By Propositions~\ref{prop-5-basket-in-class},~\ref{prop-t-villa-in-class}, and~\ref{prop-t-mansion-in-class}, $Q$ is a $(2P_3,C_4,C_6,C_7,T_0)$-free graph that contains an induced 3-pentagon and contains no simplicial vertices. Since none of the graphs $2P_3,C_4,C_6,C_7,T_0$ contains a universal vertex, Proposition~\ref{prop-H-free-no-universal} guarantees that $G$ is $(2P_3,C_4,C_6,C_7,T_0)$-free. Moreover, since $Q$ contains an induced 3-pentagon, so does $G$. Finally, since $Q$ contains no simplicial vertices, Proposition~\ref{prop-one-nontrivial-anticomp-simplicial} guarantees that $G$ contains no simplicial vertices. Thus, (a) holds. 
\end{proof}

\section{$\boldsymbol{(2P_3,C_4,C_6,C_7,\text{3-pentagon})}$-free graphs} \label{sec:3-pentagon-free}

In this section, we study those $(2P_3,C_4,C_6,C_7,T_0)$-free graphs that are additionally 3-pentagon-free. Recall that the 3-pentagon is an induced subgraph of $T_0$; consequently, a graph is $(2P_3,C_4,C_6,C_7,T_0,\text{3-pentagon})$-free if and only if it is simply $(2P_3,C_4,C_6,C_7,\text{3-pentagon})$-free. The main goal of this section in to prove Theorem~\ref{thm-non-chordal-pyramid-free-iff}, which gives a full structural description of $(2P_3,C_4,C_6,C_7,\text{3-pentagon})$-free graphs that contain no simplicial vertices. 

We begin with some definitions. A {\em theta} is any subdivision of the complete bipartite graph $K_{2,3}$; in particular, $K_{2,3}$ is a theta. A {\em pyramid} is any subdivision of the complete graph $K_4$ in which one triangle remains unsubdivided, and of the remaining three edges, at least two edges are subdivided at least once. A {\em prism} is any subdivision of $\overline{C_6}$ (the complement of $C_6$) in which the two triangles remain unsubdivided; in particular, $\overline{C_6}$ is a prism. A {\em three-path-configuration} (or {\em 3PC} for short) is any theta, pyramid, or prism; the three types of 3PC are represented in Figure~\ref{fig:Truemper}. Note that any 3PC contains exactly three holes. Moreover, any theta and any prism contains an even hole. On the other hand, every pyramid contains an odd hole, and it is easy to see that the 3-pentagon is the only pyramid in which all three holes are of length five. Since all holes in a $(2P_3,C_4,C_6,C_7)$-free graph are of length five, it follows that all $(2P_3,C_4,C_6,C_7,\text{3-pentagon})$-free graphs are 3PC-free. 

A {\em wheel} is a graph that consists of a hole and an additional vertex that has at least three neighbors in the hole. If this additional vertex is adjacent to all vertices of the hole, then the wheel is said to be a {\em universal wheel}; if the additional vertex is adjacent to three consecutive vertices of the hole, and to no other vertices of the hole, then the wheel is said to be a {\em twin wheel}. A {\em proper wheel} is a wheel that is neither a universal wheel nor a twin wheel. Note that every proper wheel has at least two holes of two different lengths. Since all holes in a $(2P_3,C_4,C_6,C_7)$-free graph are of length five, it follows that all $(2P_3,C_4,C_6,C_7)$-free graphs are proper-wheel free, and moreover, that the only wheels that $(2P_3,C_4,C_6,C_7)$-free graphs may possibly contain as induced subgraphs are the two wheels represented in Figure~\ref{fig:TwoWheels}. We summarize these remarks in Proposition~\ref{prop-no-3-pentagon-GUT} (below) for future reference. 

\begin{proposition} \label{prop-no-3-pentagon-GUT} Any $(2P_3,C_4,C_6,C_7,\text{3-pentagon})$-free graph is $(\text{3PC},\text{proper wheel})$-free. 
\end{proposition} 

A {\em ring} (originally introduced in~\cite{VIK}) is a graph $R$ whose vertex set can be partitioned into $k \geq 4$ nonempty sets, say $X_0,\dots,X_{k-1}$ (with indices understood to be in $\mathbb{Z}_k$), such that for all $i \in \mathbb{Z}_k$, $X_i$ can be ordered as $X_i = \{u_1^i,\dots,u_{|X_i|}^i\}$ so that $X_i \subseteq N_R[u_{|X_i|}^i] \subseteq \dots \subseteq N_R[u_1^i] = X_{i-1} \cup X_i \cup X_{i+1}$. (Note that this implies that $X_0,\dots,X_{k-1}$ are all cliques. Moreover, note that for each $i \in \mathbb{Z}_k$, $u_1^i$ is complete to $X_{i-1}  \cup X_{i+1}$. Consequently, for all $i \in \mathbb{Z}_k$, every vertex in $X_i$ is complete to $\{u_1^{i-1},u_1^{i+1}\}$, and in particular, every vertex in $X_i$ has a neighbor both in $X_{i-1}$ and in $X_{i+1}$.) Under these circumstances, we also say that the ring $R$ is of {\em length} $k$, as well as that $R$ is a {\em $k$-ring}. Furthermore, we say that $(X_0,\dots,X_{k-1})$ is a {\em $k$-ring partition} (or simply a {\em ring partition}) of the ring $R$. 

It may be worth mentioning that rings generalize hyperholes. Indeed, for any integer $k \geq 4$, if $H$ is a $k$-hyperhole with an associated $k$-hyperhole partition $(X_0,X_1,\dots,X_{k-1})$, then clearly, $H$ is also a $k$-ring with an associated $k$-ring partition $(X_0,X_1,\dots,X_{k-1})$. 

A {\em long hole} is a hole of length at least five, and a {\em long ring} is a ring of length at least five. Theorem~\ref{thm-GUT-decomp} (below) is a decomposition theorem for $(\text{3PC},\text{proper wheel})$-free graphs proven in~\cite{VIK}. The proof of Theorem~\ref{thm-non-chordal-pyramid-free-iff} (the main result of this section) will rely heavily on Theorem~\ref{thm-GUT-decomp}. 

\begin{theorem} [Theorem~1.6 of~\cite{VIK}] \label{thm-GUT-decomp} Let $G$ be a $(\text{3PC},\text{proper wheel})$-free graph. Then one of the following holds: 
\begin{itemize} 
\item $G$ has exactly one nontrivial anticomponent, and this anticomponent is a long ring;
\item $G$ is (long hole, $K_{2,3}$, $\overline{C_6}$)-free; 
\item $\alpha(G) = 2$, and every anticomponent of $G$ is either a 5-hyperhole or a $(C_5,\overline{C_6})$-free graph; 
\item $G$ admits a clique-cutset. 
\end{itemize} 
\end{theorem} 

Before we can use Theorem~\ref{thm-GUT-decomp} to prove our main theorem of the section (Theorem~\ref{thm-non-chordal-pyramid-free-iff}), we need to examine rings a bit more closely. First of all, Lemma~\ref{lemma-ring-hole} and Proposition~\ref{prop-k-ring-k-hole} (below) together guarantee that, for each integer $k \geq 4$, a ring $R$ contains a $k$-hole if and only if the ring $R$ is itself of length $k$. Since all holes in a $(2P_3,C_4,C_6,C_7)$-free graph are of length five, it follows that only rings that can possibly be $(2P_3,C_4,C_6,C_7)$-free are rings of length five. 

\begin{lemma} [Lemma~2.4(b) of~\cite{VIK}] \label{lemma-ring-hole} Let $k \geq 4$ be an integer. Then every hole in a $k$-ring is of length $k$. 
\end{lemma}

\begin{proposition} \label{prop-k-ring-k-hole} Let $k \geq 4$ be an integer. Then every $k$-ring contains a hole of length $k$. 
\end{proposition} 
\begin{proof} 
Let $R$ be a $k$-ring, let $(X_0,\dots,X_{k-1})$ be a ring partition of $R$ (with indices understood to be in $\mathbb{Z}_k$), and for all $i \in \mathbb{Z}_k$, let $X_i = \{u_1^i,\dots,u_{|X_i|}^i\}$ be an ordering of $R$ such that $X_i \subseteq N_R[u_{|X_i|}^i] \subseteq \dots \subseteq N_R[u_1^i] = X_{i-1} \cup X_i \cup X_{i+1}$, as in the definition of a $k$-ring. Then $x_1^0,x_1^1,\dots,x_1^{k-1},x_1^0$ is a $k$-hole in $R$. 
\end{proof} 

\begin{proposition} \label{prop-5-ring-4K1-2P3} A 5-ring is $4K_1$-free if and only if it is $2P_3$-free. 
\end{proposition} 
\begin{proof} 
Let $R$ be a 5-ring, and let $(X_0,X_1,X_2,X_3,X_4)$ be a ring partition of $R$ (with indices understood to be in $\mathbb{Z}_5$). Since $4K_1$ is an induced subgraph of $2P_3$, it is clear that if $R$ is $4K_1$-free, then $R$ is $2P_3$-free. Assume now that $R$ is not $4K_1$-free, and let $\{a_1,a_2,a_3,a_4\}$ be a stable set of size four in $R$. Since $X_0,X_1,X_2,X_3,X_4$ are all cliques, it is clear that none of them contains more than one of $a_1,a_2,a_3,a_4$. By symmetry, we may now assume that $a_1 \in X_1$, $a_2 \in X_2$, $a_3 \in X_3$, and $a_4 \in X_4$. By the definition of a ring, some vertex $x_1 \in X_1$ is complete to $X_2$, and some vertex $x_4 \in X_4$ is complete to $X_3$.\footnote{Since $a_1$ is nonadjacent to $a_2 \in X_2$, we see that $a_1$ is not complete to $X_2$, and consequently, $x_1 \neq a_1$. Similarly, $x_4 \neq a_4$.} But then $R[a_1,x_1,a_2,a_3,x_4,a_4]$ is a $2P_3$, and it follows that $R$ is not $2P_3$-free. 
\end{proof}

A {\em 5-crown} (originally defined in~\cite{4K1C4C6C7Free}) is a 5-ring $R$ with ring partition $(X_0,X_1,X_2,X_3,X_4)$ such that for some index $i^* \in \mathbb{Z}_5$, we have that $X_{i^*-1}$ is complete to $X_{i^*-2}$, and $X_{i^*+1}$ is complete to $X_{i^*+2}$. Under such circumstances, we say that $(X_0,X_1,X_2,X_3,X_4)$ is a {\em 5-crown partition} of the 5-crown $R$. A 5-crown with $i^* = 0$ is represented in Figure~\ref{fig:5crown}. 

Note that the vertex set of any 5-crown can be partitioned into three cliques;\footnote{Indeed, if $(X_0,X_1,X_2,X_3,X_4)$ is a 5-crown partition of a 5-crown $R$, with $i^*$ as in the definition of a 5-crown, then $X_{i^*}$, $X_{i^*-1} \cup X_{i^*-2}$, and $X_{i^*+1} \cup X_{i^*+2}$ are cliques of $R$, and clearly, these three cliques partition $V(R)$.} therefore, 5-crowns are $4K_1$-free. Interestingly, in the context of 5-rings, a converse of sorts also holds, as per the following lemma. 

\begin{figure} 
\begin{center}
\includegraphics[scale=0.5]{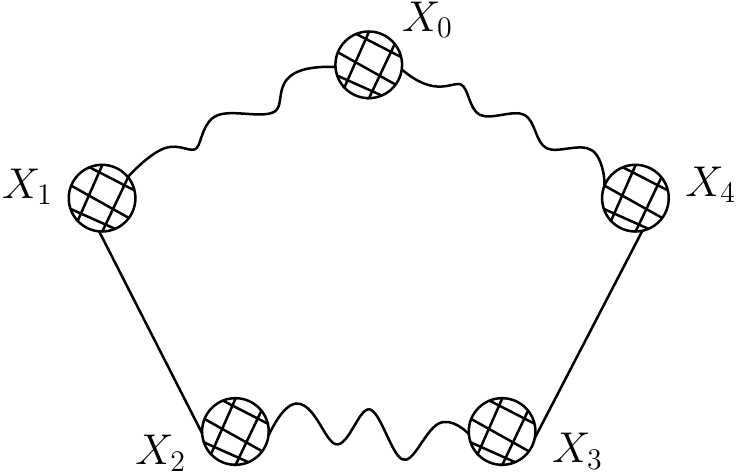}
\end{center} 
\caption{A 5-crown with an associated 5-crown partition $(X_0,X_1,X_2,X_3,X_4)$ and with $i^* = 0$. All cliques represented by crosshatched disks are nonempty.} \label{fig:5crown} 
\end{figure}

\begin{lemma} [Lemma~2.10 of~\cite{4K1C4C6C7Free}] \label{lemma-5crown-4K1-free} Let $R$ be a graph. Then the following are equivalent: 
\begin{itemize} 
\item $R$ is a $4K_1$-free 5-ring; 
\item $R$ is a 5-crown. 
\end{itemize} 
Moreover, any ring partition of a $4K_1$-free 5-ring is a 5-crown partition. 
\end{lemma}

\begin{proposition} \label{prop-5-crown-in-class} Every 5-crown is $(2P_3,C_4,C_6,C_7,\text{3-pentagon})$-free. Moreover, every 5-crown is anticonnected and contains no universal and no simplicial vertices. 
\end{proposition} 
\begin{proof} 
Let $R$ be a 5-crown, and let $(X_0,X_1,X_2,X_3,X_4)$ be a 5-crown partition of $R$. Let $i^*$ be as in the definition of a 5-crown; by symmetry, we may assume that $i^* = 0$. 

\begin{adjustwidth}{1cm}{1cm} 
\begin{claim} \label{prop-5-crown-in-class-claim-2P3C4C5C7-free} 
$R$ is $(2P_3,C_4,C_6,C_7)$-free. 
\end{claim} 
\end{adjustwidth} 
{\em Proof of Claim~\ref{prop-5-crown-in-class-claim-2P3C4C5C7-free}.} By Proposition~\ref{lemma-5crown-4K1-free}, $R$ is $4K_1$-free; since $4K_1$ is an induced subgraph of $2P_3$, it follows that $R$ is $2P_3$-free. Further, the 5-crown $R$ is, in particular, a 5-ring; so, by Lemma~\ref{lemma-ring-hole}, all holes in $R$ are of length five, and it follows that $R$ is $(C_4,C_6,C_7)$-free.~$\blacklozenge$

\begin{adjustwidth}{1cm}{1cm} 
\begin{claim} \label{prop-5-crown-in-class-claim-3-pentagon-free} 
$R$ is 3-pentagon-free. 
\end{claim} 
\end{adjustwidth} 
{\em Proof of Claim~\ref{prop-5-crown-in-class-claim-3-pentagon-free}.} Suppose otherwise, and let $P$ be an induced 3-pentagon in $R$. Since $P$ contains seven vertices, we see that for some $i \in \mathbb{Z}_5$, $X_i$ contains at least two vertices of $P$, call them $u$ and $v$. Since $u,v$ are distinct vertices of $X_i$, we see that one of $u,v$ dominates the other in $R$, and therefore in $P$ as well. But this is impossible since $P$ is a 3-pentagon, and so no vertex of $P$ dominates any other vertex of $P$, a contradiction.~$\blacklozenge$

\begin{adjustwidth}{1cm}{1cm} 
\begin{claim} \label{prop-5-crown-in-class-claim-anticonn} 
$R$ is anticonnected and contains no universal and no simplicial vertices. 
\end{claim} 
\end{adjustwidth} 
{\em Proof of Claim~\ref{prop-5-crown-in-class-claim-anticonn}.} Since $X_0$ and $X_2 \cup X_3$ are nonempty, disjoint, and anticomplete to each other, we see that $R[X_0 \cup X_2 \cup X_3]$ is anticonnected. Since $X_1$ is anticomplete to $X_3 \neq \emptyset$, and $X_4$ is anticomplete to $X_2 \neq \emptyset$, we deduce that $R[(X_0 \cup X_2 \cup X_3) \cup (X_1 \cup X_4)] = R$ is anticonnected. Since $R$ contains more than one vertex, this implies, in particular, that $R$ contains no universal vertices. 

It remains to show that $R$ contains no simplicial vertices. But this follows immediately from the fact that for all $i \in \mathbb{Z}_5$, every vertex in $X_i$ has a neighbor both in $X_{i-1}$ and in $X_{i+1}$, and from the fact that $X_{i-1}$ and $X_{i+1}$ are disjoint and anticomplete to each other.~$\blacklozenge$ 

\medskip 

We are now done: Claims~\ref{prop-5-crown-in-class-claim-2P3C4C5C7-free},~\ref{prop-5-crown-in-class-claim-3-pentagon-free}, and~\ref{prop-5-crown-in-class-claim-anticonn} together prove the proposition. 
\end{proof}

\begin{lemma} \label{lemma-non-chordal-pyramid-free} Let $G$ be a $(2P_3,C_4,C_6,C_7,\text{3-pentagon})$-free graph. Then either $G$ is a 5-crown, or $G$ contains a simplicial vertex or a universal vertex. 
\end{lemma} 
\begin{proof} 
By Proposition~\ref{prop-no-3-pentagon-GUT}, $G$ is $(\text{3PC},\text{proper wheel})$-free, and so Theorem~\ref{thm-GUT-decomp} applies. However, before applying this theorem, we make some observations. First of all, we may assume that $G$ does not admit a clique-cutset, for otherwise, Proposition~\ref{prop-2P3C4-free-clique-cut-simplicial} guarantees that $G$ contains a simplicial vertex, and we are done. Further, since $G$ is $(2P_3,C_4,C_6,C_7)$-free, we know that all holes in $G$ are of length five. If $G$ is additionally $C_5$-free, then $G$ is chordal and consequently contains a simplicial vertex (by~\cite{D61}), and once again we are done. So, we may assume that $G$ contains an induced $C_5$. Further, we may assume that $G$ contains no trivial anticomponents, for otherwise, $G$ contains a universal vertex, and we are done. Since $G$ is $C_4$-free, Proposition~\ref{prop-C4-free-one-anticomp} now guarantees that $G$ is anticonnected. 

We now have that our graph $G$ is $(\text{3PC},\text{proper wheel})$-free and anticonnected, contains an induced $C_5$ (and is therefore not long-hole-free), and does not admit a clique-cutset. So, by Theorem~\ref{thm-GUT-decomp}, $G$ is either a long ring or a 5-hyperhole. In the latter case, we are done (because 5-hyperholes are 5-crowns). We may therefore assume that $G$ is a long ring. Since all holes in $G$ are of length five, Proposition~\ref{prop-k-ring-k-hole} guarantees that $G$ is a 5-ring. But now Proposition~\ref{prop-5-ring-4K1-2P3} and Lemma~\ref{lemma-5crown-4K1-free} together imply that $G$ is a 5-crown. 
\end{proof}


\begin{theorem} \label{thm-non-chordal-pyramid-free-iff} 
For any graph $G$, the following are equivalent: 
\begin{enumerate}[(a)] 
\item $G$ is a $(2P_3,C_4,C_6,C_7,\text{3-pentagon})$-free graph that contains no simplicial vertices; 
\item $G$ contains exactly one nontrivial anticomponent, and this anticomponent is a 5-crown; 
\item $G$ can be obtained from a 5-crown by possibly adding universal vertices to it. 
\end{enumerate} 
\end{theorem} 
\begin{proof} 
Fix a graph $G$. By Proposition~\ref{prop-5-crown-in-class}, any 5-crown is anticonnected, and obviously, it contains at least two vertices. So, by Proposition~\ref{prop-non-trivial-anticomp-univ-vertices}, (b) and (c) are equivalent. It remains to show that (a) and (b) are equivalent. 

Suppose first that (a) holds. Since $G$ contains no simplicial vertices, $G$ is not a complete graph. Consequently, $G$ contains at least one nontrivial anticomponent. Since $G$ is $C_4$-free, Proposition~\ref{prop-C4-free-one-anticomp} then implies that $G$ contains exactly one nontrivial anticomponent, call it $Q$. Since $G$ is $(2P_3,C_4,C_6,C_7,\text{3-pentagon})$-free, so is its anticomponent $Q$. Since $G$ contains no simplicial vertices, Proposition~\ref{prop-one-nontrivial-anticomp-simplicial} guarantees that $Q$ contains no simplicial vertices either. Since $Q$ is an anticonnected graph on at least two vertices, $Q$ contains no universal vertices. So, by Lemma~\ref{lemma-non-chordal-pyramid-free}, $Q$ is a 5-crown. Thus, (b) holds. 

Suppose conversely that (b) holds. Then $G$ contains exactly one nontrivial anticomponent, call it $Q$, and this anticomponent is a 5-crown. By Proposition~\ref{prop-5-crown-in-class}, the 5-crown $Q$ is $(2P_3,C_4,C_6,C_7,\text{3-pentagon})$-free; so, by Proposition~\ref{prop-H-free-no-universal}, $G$ is also $(2P_3,C_4,C_6,C_7,\text{3-pentagon})$-free. Moreover, by Proposition~\ref{prop-5-crown-in-class}, $Q$ contains no simplicial vertices; consequently, by Proposition~\ref{prop-one-nontrivial-anticomp-simplicial}, $G$ contains no simplicial vertices either. Thus, (a) holds. 
\end{proof} 

\begin{corollary} For any graph $G$, the following are equivalent: 
\begin{enumerate}[(a)] 
\item $G$ is $(2P_3,C_4,C_6,C_7,\text{3-pentagon})$-free and contains no simplicial vertices; 
\item $G$ is $(4K_1,C_4,C_6,C_7,\text{3-pentagon})$-free and contains no simplicial vertices. 
\end{enumerate} 
\end{corollary} 
\begin{proof} 
Since $4K_1$ is an induced subgraph of $2P_3$, it is clear that (b) implies (a). We now assume (a) and prove (b). It suffices to show that $G$ is $4K_1$-free, for the rest follows immediately from (a). Since (a) holds, Theorem~\ref{thm-non-chordal-pyramid-free-iff} guarantees that $G$ has exactly one nontrivial anticomponent, call it $Q$, and this nontrivial anticomponent is a 5-crown. By Lemma~\ref{lemma-5crown-4K1-free}, $Q$ is $4K_1$-free. Since $4K_1$ contains no universal vertices, Proposition~\ref{prop-H-free-no-universal} now implies that $G$ is also $4K_1$-free. This proves (a). 
\end{proof}

\section{The main structural results} \label{sec:structure}

We now put the results of the previous sections together in order to state and prove our main structural results: Theorems~\ref{thm-structure-C7T0Free} and~\ref{thm-main-structure} (below). Theorem~\ref{thm-structure-C7T0Free} gives a structural description of $(2P_3,C_4,C_6,C_7,T_0)$-free graphs that contain no simplicial vertices, whereas Theorem~\ref{thm-main-structure} gives a structural description of $(2P_3,C_4,C_6)$-free graphs that contain no simplicial vertices. (We remind the reader that the 3-pentagon was defined in subsection~\ref{subsec:BasicGraphsDef}, and it is also represented in Figure~\ref{fig:tPentagon}. Further, 5-baskets, villas, and mansions were defined in subsection~\ref{subsec:BasicGraphsDef}, and 5-crowns were defined in section~\ref{sec:3-pentagon-free}. Meanwhile, graphs $T_0$ and $T_1$ are represented in Figure~\ref{fig:T0T1}, and the family $\mathcal{M}$ was defined in section~\ref{sec:Part1}.)

\begin{theorem} \label{thm-structure-C7T0Free} For every graph $G$, the following are equivalent: 
\begin{enumerate}[(a)] 
\item $G$ is $(2P_3,C_4,C_6,C_7,T_0)$-free and contains no simplicial vertices; 
\item $G$ has exactly one nontrivial anticomponent, and this anticomponent is a 5-basket, a villa, a mansion, or a 5-crown. 
\item $G$ can be obtained from a 5-basket, a villa, a mansion, or a 5-crown, by possibly adding universal vertices to it. 
\end{enumerate} 
\end{theorem} 
\begin{proof} 
Recall that the 3-pentagon is an induced subgraph of $T_0$, and consequently, a graph is $(2P_3,C_4,C_6,C_7,T_0,\text{3-pentagon})$-free if and only if it is $(2P_3,C_4,C_6,C_7,\text{3-pentagon})$-free. The result now follows immediately from Theorems~\ref{thm-T0free-pyramid-iff} and~\ref{thm-non-chordal-pyramid-free-iff}. 
\end{proof} 

\begin{theorem} \label{thm-main-structure} For every graph $G$, the following are equivalent: 
\begin{enumerate}[(a)] 
\item $G$ is a $(2P_3,C_4,C_6)$-free graph that contains no simplicial vertices; 
\item $G$ has exactly one nontrivial anticomponent, and this anticomponent is one of the following: 
\begin{itemize} 
\item a thickening of a graph in $\mathcal{M} \cup \{T_0,T_1\}$, 
\item a 5-basket, a villa, a mansion, or a 5-crown. 
\end{itemize} 
\item $G$ can be obtained from one of the following graphs by possibly adding universal vertices to it: 
\begin{itemize} 
\item a thickening of a graph in $\mathcal{M} \cup \{T_0,T_1\}$, 
\item a 5-basket, a villa, a mansion, or a 5-crown. 
\end{itemize} 
\end{enumerate} 
\end{theorem} 
\begin{proof} 
This follows immediately from Theorems~\ref{thm-main-withC7T0} and~\ref{thm-structure-C7T0Free}. 
\end{proof}

\section{Clique-width} \label{sec:cwd}

Our goal in this section is to prove that $(2P_3,C_4,C_6,C_7,T_0)$-free graphs that contain no simplicial vertices have bounded clique-width (see Theorem~\ref{thm-cwd-in-class-with-C7T0Free}). Combined with Theorem~\ref{thm-cwd-in-class-with-C7-T0}, this will imply that $(2P_3,C_4,C_6)$-free graphs that contain no simplicial vertices have bounded clique-width (see Theorem~\ref{thm-main-cwd}). The following lemma was proven in the first part of our series. 

\begin{lemma} [Lemma~6.4 of~\cite{2P3C4C6FreePart1}] \label{lemma-add-univ-cwd} If a graph $G$ can be obtained from a graph $Q$ by possibly adding universal vertices to $Q$, then $\text{cwd}(Q) \leq \text{cwd}(G) \leq \max\big\{\text{cwd}(Q),2\big\}$. 
\end{lemma} 

Theorem~\ref{thm-structure-C7T0Free} and Lemma~\ref{lemma-add-univ-cwd} reduce the problem of showing that $(2P_3,C_4,C_6,C_7,T_0)$-free graphs that contain no simplicial vertices have bounded clique-width, to the problem of showing that 5-baskets, villas, mansions, and 5-crowns have bounded clique-width. But in fact, it was already shown in~\cite{4K1C4C6C7Free} that 5-baskets and 5-crowns have bounded clique-width. More precisely, we have the following.

\begin{lemma} [Lemma~3.4 of~\cite{4K1C4C6C7Free}] \label{lemma-5-basket-cwd} Every 5-basket $Q$ satisfies $\text{cwd}(Q) \leq 5$. 
\end{lemma} 

\begin{lemma} [Lemma~3.5 of~\cite{4K1C4C6C7Free}] \label{lemma-5-crown-cwd} Every 5-crown $Q$ satisfies $\text{cwd}(Q) \leq 5$. 
\end{lemma} 

This reduces the problem to showing that villas and mansions have bounded clique-width. 

\medskip 

We start by defining ``labeled graphs'' and their clique-width, as in~\cite{4K1C4C6C7Free}. A {\em labeling} of a graph $G$ is any function whose domain is $V(G)$. A {\em labeled graph} is an ordered pair $(G,L)$, where $G$ is a graph, and $L$ is a labeling of $G$; for a vertex $v \in V(G)$, $L(v)$ is the {\em label} of $v$. The {\em disjoint union} of two labeled graphs on disjoint vertex sets is defined in the natural way. To simplify notation, for a labeled graph $(G,L)$ and an induced subgraph $H$ of $G$, we often write $(H,L)$ instead of $(H,L \upharpoonright V(H))$.\footnote{As usual, for a function $f:A \rightarrow B$ and a set $A' \subseteq A$, we denote by $f \upharpoonright A'$ the restriction of $f$ to $A'$.} 

The {\em clique-width} of a labeled graph $(G,L)$, denoted by $\text{cwd}(G,L)$, is the minimum number of labels needed to construct $(G,L)$ using the following four operations:\footnote{Note that these are the same four operations that we had in the definition of the clique-width of nonlabeled graphs. The only difference is that, here, we insist that the labeling of $G$ at the end of the procedure be precisely the labeling $L$.} 
\begin{enumerate} 
\item creation of a new vertex $v$ with label $i$; 
\item disjoint union of two labeled graphs; 
\item joining by an edge every vertex labeled $i$ to every vertex labeled $j$ (where $i \neq j$); 
\item renaming label $i$ to label $j$.  
\end{enumerate} 
Thus, at the end of the procedure, each vertex $v \in V(G)$ is supposed to have label $L(v)$. Clearly, $\text{cwd}(G) \leq \text{cwd}(G,L)$. 

\begin{lemma} [Lemma~3.1 of~\cite{4K1C4C6C7Free}] \label{lemma-cwd-complete} Let $G$ be a complete graph, and let $L:V(G) \rightarrow C$ be a labeling of $G$. Then $\text{cwd}(G,L) \leq |C|+1$. 
\end{lemma}

For a positive integer $t$, a {\em $t$-spike} is a graph $Q$ whose vertex set can be partitioned into nonempty cliques $B_1,\dots,B_t,C_1,\dots,C_t$ that satisfy the following: 
\begin{itemize} 
\item cliques $B_1,\dots,B_t$ are pairwise anticomplete to each other; 
\item cliques $C_1,\dots,C_t$ are pairwise complete to each other; 
\item for all distinct $i,j \in \{1,\dots,t\}$, $B_i$ is anticomplete to $C_j$; 
\item for all $i \in \{1,\dots,t\}$, $B_i$ can be ordered as $B_i = \{b_1^i,\dots,b_{r_i}^i\}$ so that $N_Q(b_{r_i}^i) \cap C_i \subseteq \dots \subseteq N_Q(b_1^i) \cap C_i$. 
\end{itemize} 
Under these circumstances, we also say that $(B_1,\dots,B_t;C_1,\dots,C_t)$ is a {\em $t$-spike partition} of the $t$-spike $Q$. 

\begin{proposition} \label{prop-t-spike-cwd} For all positive integers $t$, all $t$-spikes $Q$ with an associated $t$-spike partition $(B_1,\dots,B_t;C_1,\dots,C_t)$, and all labelings $L$ of $Q$ that assign one label to all vertices of $B_1 \cup \dots \cup B_t$ and another label to all vertices of $C_1 \cup \dots \cup C_t$,\footnote{To be fully precise, we mean that there exist distinct $p,q$ such that: 
\begin{itemize} 
\item $L(b) = p$ for all $b \in B_1 \cup \dots \cup B_t$; 
\item $L(c) = q$ for all $c \in C_1 \cup \dots \cup C_t$. 
\end{itemize}} we have that $\text{cwd}(Q,L) \leq 4$. 
\end{proposition} 
\begin{proof} 
We proceed by induction on $t$, and in the base case (see Claim~\ref{prop-t-spike-cwd-claim-base-case}), we proceed by induction on the number of vertices of the graph.

\begin{adjustwidth}{1cm}{1cm} 
\begin{claim} \label{prop-t-spike-cwd-claim-base-case} 
For all 1-spikes $Q$, with an associated 1-spike partition $(B,C)$, and for all labelings $L$ of $Q$ that assign one label to all vertices of $B$ and another label to all vertices of $C$, we have that $\text{cwd}(Q,L) \leq 4$. 
\end{claim} 
\end{adjustwidth} 
{\em Proof of Claim~\ref{prop-t-spike-cwd-claim-base-case}.} Fix a 1-spike $Q$, and assume inductively that the claim holds for 1-spikes on fewer than $|V(Q)|$ many vertices. Let $(B,C)$ be a 1-spike partition of $Q$, and fix a labeling $L$ of $Q$ that assigns one label to all vertices of $B$ and another label to all vertices of $C$. Clearly, we may assume that $L(b) = 1$ for all $b \in B$ and $L(c) = 2$ for all $c \in C$. Next, using the definition of a 1-spike, we fix an ordering $B = \{b_1,\dots,b_r\}$ of $B$ such that $N_Q(b_r) \cap C \subseteq \dots \subseteq N_Q(b_1) \cap C$. 

Set $B_1 := B \setminus \{b_1\} = \{b_2,\dots,b_r\}$. Next, set $C_2 := N_Q(B_1) \cap C$, $C_1 := \big(N_Q(b_1) \setminus N_Q(B_1)\big) \cap C$, and $C_0 := C \setminus (C_1 \cup C_2)$. Thus, $C_2$ is the set of all vertices in $C$ that have a neighbor in $B_1 = B \setminus \{b_1\}$; $C_1$ is the set of all vertices in $C$ whose only neighbor in $B$ is $b_1$; and $C_2$ is the set of all vertices in $C$ that are anticomplete to $B$. Clearly, sets $C_0,C_1,C_2$ form a partition of $C$. Moreover, it follows from our ordering of $B$ that $b_1$ is complete to $C_1 \cup C_2$ and anticomplete to $C_0$. 

Suppose first that $B_1 = \emptyset$, so that $B = \{b_1\}$. Then $C_2 = \emptyset$ and $C = C_0 \cup C_1$. We now define $L_C:C \rightarrow \{2,3\}$ by setting $L_C(c) = 2$ for all $c \in C_1$ and setting $L_C(c) = 3$ for all $c \in C_0$. By Lemma~\ref{lemma-cwd-complete}, we know that $\text{cwd}(Q[C],L_C) \leq 3$. So, we construct $(Q[C],L_C)$ using only labels $1,2,3$. Next, we create the vertex $b_1$ with label $1$, and make vertices labeled $1$ adjacent to vertices labeled $2$. (Now $b_1$ is complete to $C_1$ and anticomplete to $C_0$.) Finally, we rename label $3$ to label $2$. We have now constructed $(Q,L)$ using only labels $1,2,3$, and we deduce that $\text{cwd}(Q,L) \leq 3$. 

From now on, we assume that $B_1 \neq \emptyset$. Since $b_1 \notin B_1 \cup C_2$, we see that $Q[B_1 \cup C_2]$ is a proper induced subgraph of $Q$. Moreover, if $C_2 = \emptyset$, then $Q[B_1 \cup C_2] = Q[B_1]$ is a complete graph, whereas if $C_2 \neq \emptyset$, then $Q[B_1 \cup C_2]$ is a 1-spike with an associated 1-spike partition $(B_1,C_2)$, and this 1-spike has fewer than $|V(Q)|$ vertices. 

We now define $L':B_1 \cup C_2 \rightarrow \{1,2\}$ by setting $L'(b) = 1$ for all $b \in B_1$ and setting $L'(c) = 2$ for all $c \in C_2$. By Lemma~\ref{lemma-cwd-complete}, or by the induction hypothesis, we have that $\text{cwd}(Q[B_1 \cup C_2],L') \leq 4$, and we construct the labeled graph $(Q[B_1 \cup C_2],L')$ using only labels $1,2,3,4$. 

Next, define $L'':\{b_1\} \cup C_1 \rightarrow \{3,4\}$ by setting $L''(b_1) = 3$ and setting $L''(c) = 4$ for all $c \in C_1$. Note that $Q[\{b_1\} \cup C_1]$ is a complete graph, and so by Lemma~\ref{lemma-cwd-complete}, we have that $\text{cwd}(Q[\{b_1\} \cup C_1],L'') \leq 3$. We now construct the labeled graph $(Q[\{b_1\} \cup C_1],L'')$ using only labels $1,3,4$. We now take the disjoint union of the labeled graphs $(Q[B_1 \cup C_2],L')$ and $(Q[\{b_1\} \cup C_1],L'')$, we make vertices labeled $3$ adjacent to vertices labeled $1$ or $2$, and we make vertices labeled $4$ adjacent to vertices labeled $2$. Then, we rename label $3$ to $1$, and we rename label $4$ to $2$. 

So far, we have constructed the graph $(Q \setminus C_0,L)$ using only labels $1,2,3,4$. If $C_0 = \emptyset$, then we are done. So, we may assume that $C_0 \neq \emptyset$. We now define $L_0:C_0 \rightarrow \{3\}$ by setting $L_0(c) = 3$ for all $c \in C$. Now, $Q[C_0]$ is a complete graph, and so Lemma~\ref{lemma-cwd-complete} guarantees that $\text{cwd}(Q[C_0],L_0) \leq 2$. We then create the labeled graph $(Q[C_0],L_0)$ using only labels $1,3$. We now take the disjoint union of the labeled graphs $(Q \setminus C_0,L)$ and $(Q[C_0],L_0)$, we make vertices labeled $3$ adjacent to vertices labeled $2$, and we rename label $3$ as label $2$. We have now created the labeled graph $(Q,L)$ using only labels $1,2,3,4$, and we are done.~$\blacklozenge$ 

\medskip 

We now complete the proof of the proposition. We proceed by induction on $t$. The base case (``$t = 1$'') follows immediately from Claim~\ref{prop-t-spike-cwd-claim-base-case}. Now, fix a positive integer $t$, and assume that the proposition is true for $t$-spikes. Fix a $(t+1)$-spike $Q$, with an associated $(t+1)$-spike partition $(B_1,\dots,B_t,B_{t+1};C_1,\dots,C_t;C_{t+1})$. Let $L$ be a labeling of $Q$ that assigns one label to all vertices of $B_1 \cup \dots \cup B_t \cup B_{t+1}$, and assigns another label to all vertices of $C_1 \cup \dots \cup C_t \cup C_{t+1}$. By symmetry, we may assume that $L(b) = 1$ for all $b \in B_1 \cup \dots \cup B_t \cup B_{t+1}$, and that $L(c) = 2$ for all $c \in C_1 \cup \dots \cup C_t \cup C_{t+1}$. Set $Q_0 := Q \setminus (B_{t+1} \cup C_{t+1})$. By the induction hypothesis, we have that $\text{cwd}(Q_0,L) \leq 4$, and we create the labeled graph $(Q_0,L)$ using only labels $1,2,3,4$. Next, define $L':B_{t+1} \cup C_{t+1} \rightarrow \{1,3\}$ by setting $L'(b) = 1$ for all $b \in B_{t+1}$, and setting $L'(c) = 3$ for all $c \in C_{t+1}$. By Claim~\ref{prop-t-spike-cwd-claim-base-case}, we know that $\text{cwd}(Q[B_{t+1} \cup C_{t+1}],L') \leq 4$; so, we create the labeled graph $(Q[B_{t+1} \cup C_{t+1}],L')$ using only labels $1,2,3,4$. Then, we take the disjoint union of the labeled graphs $(Q_0,L)$ and $(Q[B_{t+1} \cup C_{t+1}],L')$, we make vertices labeled $2$ adjacent to vertices labeled $3$, and we rename label $3$ to label $2$. We have now created the labeled graph $(Q,L)$ using only labels $1,2,3,4$, and we deduce that $\text{cwd}(Q,L) \leq 4$. This completes the induction. 
\end{proof} 

\begin{lemma} \label{lemma-t-villa-cwd} All villas $Q$ satisfy $\text{cwd}(Q) \leq 4$. 
\end{lemma} 
\begin{proof} 
Let $Q$ be a $t$-villa ($t \geq 3$) with an associated $t$-villa partition $(A;B_1,\dots,B_t;C_1,\dots,C_t)$. Set $B := B_1 \cup \dots \cup B_t$ and $C := C_1 \cup \dots \cup C_t$. Let $L:V(Q) \rightarrow \{1,2,3\}$ be defined by setting $L(a) = 1$ for all $a \in A$, setting $L(b) = 2$ for all $b \in B$, and setting $L(c) = 3$ for all $c \in C$. First of all, since $A$ is a clique, Lemma~\ref{lemma-cwd-complete} guarantees that $\text{cwd}(Q[A],L) \leq 2$. On the other hand, $Q[B \cup C]$ is a $t$-spike with an associated $t$-spike partition $(B_1,\dots,B_t;C_1,\dots,C_t)$, and so by Proposition~\ref{prop-t-spike-cwd}, we have that $\text{cwd}(Q[B \cup C],L) \leq 4$. We now create the labeled graphs $(Q[A],L)$ and $(Q[B \cup C],L)$ using only labels $1,2,3,4$, we take the union of the two labeled graphs, and we make vertices labeled $1$ adjacent to vertices labeled $2$. We have now created the labeled graph $(Q,L)$ using only labels $1,2,3,4$, and we deduce that $\text{cwd}(Q) \leq \text{cwd}(Q,L) \leq 4$. 
\end{proof}

\begin{lemma} \label{lemma-t-mansion-cwd} All mansions $Q$ satisfy $\text{cwd}(Q) \leq 5$. 
\end{lemma} 
\begin{proof} 
Let $Q$ be a $t$-mansion ($t \geq 3$), let $(A;B_1,\dots,B_t;C_1,\dots,C_t;F;X;Y)$ be an associated $t$-mansion partition of $Q$, and let $j^*$ be as in the definition of a $t$-mansion; by symmetry, we may assume that $j^* = 1$. Set $B := B_1 \cup \dots \cup B_t$ and $C := C_1 \cup \dots \cup C_t$. Define $L_1:(B \cup C \cup F) \rightarrow \{1,2,3,4,5\}$ by setting: 
\begin{itemize} 
\item $L_1(b) = 1$ for all $b \in B \setminus B_1$; 
\item $L_1(b) = 4$ for all $b \in B_1$; 
\item $L_1(c) = 2$ for all $c \in C \setminus C_1$; 
\item $L_1(c) = 5$ for all $c \in C_1$; 
\item $L_1(f) = 3$ for all $f \in F$. 
\end{itemize} 

\begin{adjustwidth}{1cm}{1cm} 
\begin{claim} \label{lemma-t-mansion-cwd-claim-L1} $\text{cwd}(Q[B \cup C \cup F],L_1) \leq 5$. 
\end{claim} 
\end{adjustwidth} 
{\em Proof of Claim~\ref{lemma-t-mansion-cwd-claim-L1}.} First, note that $Q[(B \setminus B_1) \cup (C \cup C_1)]$ is a $(t-1)$-spike with an associated $(t-1)$-spike partition $(B_2,\dots,B_t;C_2,\dots,C_t)$. Therefore, by Proposition~\ref{prop-t-spike-cwd}, we have that $\text{cwd}(Q[(B \setminus B_1) \cup (C \setminus C_1)],L_1) \leq 4$. Next, $B_1 \cup C_1$ and $F$ are both cliques, and $L_1$ uses two colors on the former and one color on the latter. So, by Lemma~\ref{lemma-cwd-complete}, $\text{cwd}(Q[B_1 \cup C_1],L_1) \leq 3$ and $\text{cwd}(Q[F],L_1) \leq 2$. Since the labeling $L_1$ uses labels $1,2,3,4,5$, we conclude that we can create labeled graphs $\text{cwd}(Q[(B \setminus B_1) \cup (C \cup C_1)],L_1)$, $(Q[B_1 \cup C_1],L_1)$, and $\text{cwd}(Q[F],L_1)$ using only labels $1,2,3,4,5$. Then, we take the disjoint union of these three graphs. Finally, we make all vertices labeled $3$ adjacent to all vertices labeled $1$ or $2$, and we make all vertices labeled $5$ adjacent to all vertices labeled $2$. We have now created the labeled graph $(Q[B \cup C \cup F],L_1)$ using only labels $1,2,3,4,5$, and we deduce that $\text{cwd}(Q[B \cup C \cup F],L_1) \leq 5$.~$\blacklozenge$ 

\medskip 

Next, we define $L_2:(B \cup C \cup F \cup Y) \rightarrow \{1,2,3,4,5\}$ by setting 
\begin{itemize} 
\item $L_2(b) = 1$ for all $b \in B \setminus B_1$; 
\item $L_2(b) = 4$ for all $b \in B_1$; 
\item $L_2(c) = 2$ for all $c \in C$; 
\item $L_2(f) = 3$ for all $f \in F$; 
\item $L_2(y) = 5$ for all $y \in Y$. 
\end{itemize} 

\begin{adjustwidth}{1cm}{1cm} 
\begin{claim} \label{lemma-t-mansion-cwd-claim-L2} 
$\text{cwd}(Q[B \cup C \cup F \cup Y],L_2) \leq 5$. 
\end{claim} 
\end{adjustwidth} 
{\em Proof of Claim~\ref{lemma-t-mansion-cwd-claim-L2}.} First, using Claim~\ref{lemma-t-mansion-cwd-claim-L1}, we create the graph $(Q[B \cup C \cup F],L_1)$ using only labels $1,2,3,4,5$, and we rename label $5$ to $2$. We have now created the labeled graph $(Q[B \cup C \cup F],L_2)$ using only labels $1,2,3,4,5$. Thus, if $Y = \emptyset$, then we are done. So, assume that $Y \neq \emptyset$. Since $Y$ is a clique, Lemma~\ref{lemma-cwd-complete} guarantees that $\text{cwd}(Q[Y],L_2) \leq 2$. We now create the labeled graph $(Q[Y],L_2)$ using only labels $4,5$, and then we take the disjoint union of the graphs $(Q[B \cup C \cup F],L_2)$ and $(Q[Y],L_2)$, and we make the vertices labeled $5$ adjacent to vertices labeled $2$ or $3$. We have now created the labeled graph $(Q[B \cup C \cup F \cup Y],L_2)$ using only labels $1,2,3,4,5$, and we deduce that $\text{cwd}(Q[B \cup C \cup F \cup Y],L_2) \leq 5$.~$\blacklozenge$ 

\medskip 

Next, we define $L_3:(A \cup B \cup C \cup F \cup Y) \rightarrow \{1,2,3,4,5\}$ by setting 
\begin{itemize} 
\item $L_3(a) = 4$ for all $a \in A$; 
\item $L_3(b) = 1$ for all $b \in B \setminus B_1$; 
\item $L_3(b) = 3$ for all $b \in B_1$; 
\item $L_3(c) = 2$ for all $c \in C$; 
\item $L_3(f) = 3$ for all $f \in F$; 
\item $L_3(y) = 2$ for all $y \in Y$. 
\end{itemize}

\begin{adjustwidth}{1cm}{1cm} 
\begin{claim} \label{lemma-t-mansion-cwd-claim-L3}
$\text{cwd}(Q[A \cup B \cup C \cup F \cup Y],L_3) \leq 5$. 
\end{claim} 
\end{adjustwidth} 
{\em Proof of Claim~\ref{lemma-t-mansion-cwd-claim-L3}.} Using Claim~\ref{lemma-t-mansion-cwd-claim-L2}, we create the labeled graph $(Q[B \cup C \cup F \cup Y],L_2)$ using only labels $1,2,3,4,5$. Next, we rename label $5$ to $2$, and we rename label $4$ to $3$. So far, we have created the labeled graph $(Q[B \cup C \cup F \cup Y],L_3)$. Next, by Lemma~\ref{lemma-cwd-complete}, we know that $\text{cwd}(Q[A],L_3) \leq 2$. So, we create the labeled graph $(Q[A],L_3)$ using only labels $3,4$, and then we take the disjoint union of the labeled graphs $(Q[A],L_3)$ and $(Q[B \cup C \cup F \cup Y],L_3)$. Finally, we make vertices labeled $4$ adjacent to the vertices labeled $1$ or $3$. We have now created the labeled graph $(Q[A \cup B \cup C \cup F \cup Y],L_3)$ using only labels $1,2,3,4,5$, and we deduce that $\text{cwd}(Q[A \cup B \cup C \cup F \cup Y],L_3) \leq 5$.~$\blacklozenge$

\medskip 

We now proceed as follows. We define $L:V(Q) \rightarrow \{1,2,3,4,5\}$ by setting 
\begin{itemize} 
\item $L(x) = 5$ for all $x \in X$; 
\item $L(x) = L_3(x)$ for all $x \in V(Q) \setminus X$. 
\end{itemize} 
Using Claim~\ref{lemma-t-mansion-cwd-claim-L3}, we first create the labeled graph $(Q[A \cup B \cup C \cup F \cup Y],L_3) = (Q \setminus X,L)$ using only labels $1,2,3,4,5$. If $X = \emptyset$, then we are done. So, we may assume that $X \neq \emptyset$. By Lemma~\ref{lemma-cwd-complete}, we know that $\text{cwd}(X,L) \leq 2$. We now create the labeled graph $(Q[X],L)$ using only labels $4,5$. Then, we take the disjoint union of the graphs $(Q \setminus X,L)$ and $(Q[X],L)$, and we make all vertices labeled $5$ adjacent to all vertices labeled $3$ or $4$. We have now created the labeled graph $(Q,L)$ using only labels $1,2,3,4,5$, and we deduce that $\text{cwd}(Q) \leq \text{cwd}(Q,L) \leq 5$. This completes the argument. 
\end{proof}

\begin{theorem} \label{thm-cwd-in-class-with-C7T0Free} Let $G$ be a $(2P_3,C_4,C_6,C_7,T_0)$-free graph that contains no simplicial vertices. Then $\text{cwd}(G) \leq 5$. 
\end{theorem} 
\begin{proof} 
By Theorem~\ref{thm-structure-C7T0Free}, $G$ can be obtained from 5-basket, a villa, a mansion, or a 5-crown, by possibly adding universal vertices to it. By Lemma~\ref{lemma-add-univ-cwd}, it is therefore enough to show that if $Q$ is any 5-basket, villa, mansion, or 5-crown, then $\text{cwd}(Q) \leq 5$. But this follows from Lemmas~\ref{lemma-5-basket-cwd},~\ref{lemma-5-crown-cwd},~\ref{lemma-t-villa-cwd}, and~\ref{lemma-t-mansion-cwd}. 
\end{proof} 

\begin{theorem} \label{thm-main-cwd} If $G$ is a $(2P_3,C_4,C_6)$-free graph that contains no simplicial vertices, then $\text{cwd}(G) \leq 12$. 
\end{theorem} 
\begin{proof} 
This follows immediately from Theorems~\ref{thm-cwd-in-class-with-C7-T0} and~\ref{thm-cwd-in-class-with-C7T0Free}. 
\end{proof}

\section{Coloring $\boldsymbol{(2P_3,C_4,C_6)}$-free graphs} \label{sec:coloring}

\begin{theorem} \cite{Rao} \label{thm-Rao} \textsc{Graph Coloring} can be solved in polynomial time for graphs of bounded clique-width. 
\end{theorem} 

\begin{theorem} \label{thm-main-coloring} \textsc{Graph Coloring} can be solved in polynomial time for $(2P_3,C_4,C_6)$-free graphs. 
\end{theorem} 
\begin{proof} 
By simply examining the neighborhood of each vertex of an input graph, we can either find a simplicial vertex in the graph, or determine that the graph has no simplicial vertices. Moreover, if $v$ is a simplicial vertex of a graph $G$ on at least two vertices, then $\chi(G) = \max\{d_G(v)+1,\chi(G \setminus v)\}$, and so for any integer $k$, we have that $G$ is $k$-colorable if and only if $d_G(v) \leq k-1$ and $G \setminus v$ is $k$-colorable. On the other hand, by Theorem~\ref{thm-main-cwd}, $(2P_3,C_4,C_6)$-free graphs that contain no simplicial vertices have bounded clique-width, and so by Theorem~\ref{thm-Rao}, \textsc{Graph Coloring} can be solved in polynomial time for such graphs. The result is now immediate. 
\end{proof}

By Theorem~\ref{thm-main-coloring}, \textsc{Graph Coloring} can be solved in polynomial time for $(2P_3,C_4,C_6)$-free graphs. However, since it relies on bounded clique-width, the algorithm is in fact very slow (see~\cite{CWcol}). So, it is natural to ask whether there might be a faster coloring algorithm for $(2P_3,C_4,C_6)$-free graphs. In the first part of our series~\cite{2P3C4C6FreePart1}, we explained that $(2P_3,C_4,C_6)$-free graphs that contain an induced $C_7$ or $T_0$ can be colored in $O(n^3)$ time, although the algorithm relies on integer programming (and in particular, on the results of~\cite{Koutecky, Lampis}) and is therefore not combinatorial. The question is therefore whether there is reasonably fast coloring algorithm for $(2P_3,C_4,C_6,C_7,T_0)$-free graphs. Simplicial and universal vertices can easily be handled in the context of graph coloring, and so Theorem~\ref{thm-structure-C7T0Free} reduces the problem of coloring $(2P_3,C_4,C_6,C_7,T_0)$-free graphs to that of coloring 5-baskets, villas, mansions, 5-crowns. However, since 5-baskets and 5-crowns are $(4K_1,C_4,C_6)$-free (see Lemmas~\ref{lemma-5-basket-4K1Free} and~\ref{lemma-5crown-4K1-free}), they can in fact be colored in $O(n^3)$ time (see~\cite{4K1C4C6Alg}). This further reduces the problem to that of developing a reasonably fast coloring algorithm for villas and mansions. To summarize, combined with the previously obtained results (listed above), any $O(n^k)$ time coloring algorithm for villas and mansions will yield an $O(n^{\text{max}\{k,3\}})$ time coloring algorithm for $(2P_3,C_4,C_6)$-free graphs.

\small{
 
}

\end{document}